\documentclass[a4paper,10pt,reqno]{amsart}
\usepackage{amsfonts,amssymb,amscd,amsmath,latexsym,amsbsy,enumerate,a4wide, tikz,  ifthen, mathrsfs, mathtools}
\usepackage[all]{xy}
\usepackage{todonotes}
\usetikzlibrary{calc,shapes,arrows}
\usepackage{tikz-cd}
\usepackage{mathrsfs}
\usepackage{enumitem}

\newtheorem{thm}{Theorem}[section]
\newtheorem{defn}[thm]{Definition}
\newtheorem{lem}[thm]{Lemma}
\newtheorem{cor}[thm]{Corollary}
\newtheorem{prop}[thm]{Proposition}

\theoremstyle{definition}
\newtheorem{rmk}[thm]{Remark}

\newtheorem{assumption}[thm]{Condition}

\numberwithin{equation}{section}

\def\la{\lambda}

\def\C{\mathbb{C}}
\def\N{\mathbb{N}}
\def\Z{\mathbb{Z}}
\def\cA{\mathcal A}
\def\cD{\mathcal D}

\def\cH{\mathcal H}

\newcommand{\rFs}[5]{\,_{#1}F_{#2} \left( \genfrac{.}{.}{0pt}{}{#3}{#4};#5 \right)}

\newcommand{\SU}{\mathrm{SU}}
\newcommand{\GL}{\mathrm{GL}}
\newcommand{\U}{\mathrm{U}}

\def\cH{\mathcal H}

\newcommand{\tr}{\mathrm{tr}}

\newcommand{\fA}{\mathscr{A}}
\newcommand{\fB}{\mathscr{B}}
\newcommand{\D}{\mathcal{D}}
\newcommand{\MN}{M_N(\mathbb{C})}
\newcommand{\sW}{\mathscr{W}}

\title[Duality for discrete matrix valued polynomials]{Duality and difference operators for matrix valued discrete polynomials on the nonnegative integers}
\author{Bruno Eijsvoogel}
\address{IMAPP, Radboud Universiteit Nijmegen (The Netherlands), Department of Mathematics, KU Leuven (Belgium): b.eijsvoogel@math.ru.nl }
\author{Lucía Morey}
\address{FaMAF-CIEM, Universidad Nacional de C\'ordoba (Argentina): lmorey@unc.edu.ar}
\author{Pablo Rom\'an}
\address{FaMAF-CIEM, Universidad Nacional de C\'ordoba (Argentina): pablo.roman@unc.edu.ar}

\date{\today}

\begin{document}

\maketitle

\begin{abstract}
In this paper we introduce a notion of duality for matrix valued orthogonal polynomials with respect to a measure supported on the nonnegative integers. We show that the dual families are closely related to certain difference operators acting on the matrix orthogonal polynomials. These operators belong to the so called Fourier algebras, which play a key role in the construction of the families.

In order to illustrate duality, we describe a family of Charlier type matrix orthogonal polynomials with explicit shift operators which allow us to find explicit formulas for three term recurrences, difference operators and square norms. These are the essential ingredients for the construction of different dual families.  
\end{abstract}

\section{Introduction}
Charlier polynomials $c_n^{(a)}$ are orthogonal polynomials with respect to the Poisson distribution, see for instance \cite{NIST:DLMF, KoekS},
\begin{equation}
\label{eq:ortogonalidad-charlier-escalar}
 \sum_{x=0}^{\infty}\frac{a^x}{x!}c_n^{(a)}(x)c_m^{(a)}(x)= \frac{e^a n!}{a^n} \delta_{n,m}, \qquad a> 0,
\end{equation}
and are explicitly written in terms of a ${}_2F_0$ hypergeometric function by
\begin{equation}
\label{eq:hyperg-Charlier}
c_n^{(a)}(x)=\rFs{2}{0}{-n,-x}{-}{-\frac{1}{a}}.
\end{equation}
It is readily seen from the hypergeometric form that these polynomials are invariant if we interchange the role of the degree $n$ and the variable $x$, i.e. $c_n^{(a)}(x) = c_x^{(a)}(n)$. We say that the Charlier polynomials are \emph{self dual}. The Meixner and Krawtchouk polynomials are self dual as well.

A general concept of duality for orthogonal polynomials is introduced by D. Leonard in \cite{Douglas}. Two polynomial sequnces $(p_n)_n$ and $(q_x)_x$ for $x,n \in \{0,1,\ldots, K\}$, with $K$ possibly infinite, such that $\deg p_n =n$ and $\deg q_x =x$ are \emph{dual} if there exist sequences $k_x$ and $\ell_n$ called eigenvalues satisfying
\begin{equation}
\label{eq:condition-knln}
k_x\neq k_y, \qquad \ell_n\neq \ell_m, \qquad \text{ if }\, n\neq m, x\neq y,
\end{equation}
such that 
\begin{equation}
    \label{eq:duality-Leonard}
p_n(k_x) = q_x(\ell_n), \qquad \text{ for all }\, n,x \in \{0,\ldots ,K\}.
\end{equation}
In \cite{Douglas} it is shown that the only orthogonal polynomials that have orthogonal duals are the Askey--Wilson polynomials and limiting or sub--families. 

The theory of matrix valued orthogonal polynomials (MVOP) was initiated by Krein \cite{Krein1}  and has connections and applications in different areas of mathematics and mathematical physics such as scattering theory  \cite{Geronimo}, tiling problems \cite{DuitsK}, integrable systems \cite{AAGMM,AGMM,Manas1,IKR2}, spectral theory \cite{GroeneveltIK} and stochastic processes \cite{IglesiaG,Iglesia,IglesiaR}.

As in the case of scalar orthogonal polynomials, there is a particular interest in understanding the families of matrix orthogonal polynomials with the extra property of being eigenfunctions of a second order differential \cite{GPT, DuranG1,DurandlI2008_2, CanteroMV2007, DI, KvPR1,KvPR2}, difference \cite{DuranAlgebra, MR3040334, AKdlR} or $q$-difference operator \cite{AKdlR, AKR1}. The harmonic analysis on compact symmetric spaces has played a fundamental role in the construction of families of MVOP which are eigenfunctions of a second order differential operator. The first example was given in \cite{GPT} for the symmetric pair $(\SU(3),\mathrm{S}(\U(2)\times \U(1))$ and later extended to different rank-one pairs \cite{KvPR1,KvPR2, HeckmanP, vPruijssenR}, quantum groups \cite{AKR1} and higher rank groups \cite{KvPR3, KJ}. In this group theoretic context, some of the important properties of the MVOP such as orthogonality, recurrence relations and differential equations are understood in terms of the representation theory of the corresponding symmetric spaces. 

One of the important results in the last few years is the classification by R. Casper and M. Yakimov \cite{Casper2} of all weight matrices whose associated MVOP are eigenfunctions of a second order differential equation. 
In the scalar case this is the classical Bochner problem, whose solution states that the only families of scalar orthogonal polynomials which are eigenfunctions of a second order differential operator are the classical Hermite, Laguerre and Jacobi. Part of the theory developed in \cite{Casper2} is adapted and            used in this paper. 

The families of orthogonal polynomials in the Askey and $q$-Askey scheme generalize the classical families of Jacobi, Laguerre and Hermite and share many of their properties, as being eigenfunctions of a second order operator (differential, difference or $q$-difference), Pearson equations, the existence of suitable forward and backward operators and Rodrigues formulas. However, it was already noticed in \cite{CanteroMV2005, CanteroMV2007} that, in contrast with the scalar case, one can have a family of matrix orthogonal polynomials which are eigenfunctions of a second order differential operator and whose weight matrix does not satisfy a suitable Pearson equation. The matrix valued setup is therefore more involved and richer than the scalar counterpart.

In view of Leonard's result \cite{Douglas},  it is natural to investigate matrix orthogonal polynomials with the extra property of having duals which are orthogonal polynomials. It is the aim of this paper to develop a notion of duality for matrix valued orthogonal polynomials and to construct nontrivial examples. In this first approach we restrict ourselves to discrete measures supported on an infinite number of points of $\N_0 = \Z_{\geq 0}$.

Let us denote by $\MN[x]$ the space of all $N\times N$ matrix valued polynomials. We say that a polynomial $P\in \MN[x]$ is monic if its leading coefficient is the identity matrix. Let $W:\mathbb{Z} \to \MN$ be a weight function such that $W(x)$ is positive definite for all $x\in \N_0$, $W(x)=0$ for all $x\in -\N$. Assume that $W$ has finite moments of all orders, i.e. 
$$\sum_{x=0}^{\infty} x^n W(x) <\infty,$$
for all $n\in \N_0$, then $W$ defines a matrix valued inner product on $\MN[x]$ by
	\begin{equation}
	\label{eq:d-weight}
	\langle P, Q \rangle_W = \sum_{x=0}^{\infty} P(x) W(x) Q(x)^*,
	\end{equation}
where $\ast$ denotes the conjugate transpose. By standard arguments one can prove that there exists a unique sequence $(P_n)_n$ of monic matrix valued orthogonal polynomials (see for instance \cite{Krein1,DamanikPS}), i.e. $P_n(x)$ is monic, of degree $n$ and
\begin{equation}
 \label{eq:monic-op-def}
\langle P_m(x), P_n(x) \rangle_W = \mathcal{H}_{m} \delta_{m, n}, \qquad n,m\in \mathbb{N}_0,
\end{equation}
where $\mathcal{H}_m$ is a positive definite matrix.

In this paper, we extend Leonard's notion of duality to the matrix valued setting by allowing one of the eigenvalues in \eqref{eq:duality-Leonard} to be a matrix valued function. More precisely we take $k_x\mapsto x, \quad \ell_n\mapsto \rho(n)$, where $\rho(n)$ is a matrix valued function. As in the scalar case \eqref{eq:condition-knln}, we need to assume a condition which guarantees the existence of sufficiently many different $\rho(n)$'s and that they have some invertibility property. In this context, we say that two sequences of matrix polynomials $(P_n)_n$ and $(Q_x)_x$ are dual if they are related by
\begin{equation}
\label{eq:dualidad-introduccion}
P_n(x) = P_n(0) Q_x(\rho(n)) \Upsilon(x),\qquad n,x\in \N_0,
\end{equation}
for a certain matrix valued function $\Upsilon$, see Definition \ref{def:duality} and  Theorem \ref{thm:correspondenceDualalgebras}. The meaning of a matrix polynomial $Q_x$ evaluated at the  matrix valued function $\rho(n)$ is described in detail in \eqref{eq:Qpoly}. In the last section of this paper, we see that the study of the duals of the dual families, leads to explicit examples where both $k_x$ and $\ell_n$ are extended to matrix valued functions.

The structure of this paper, which is schematized in Figure \ref{fig:diagram}, has three main parts. The first part consists of Sections \ref{sec:discrete-weights}--\ref{sec:duality} and describes the main tools
and concepts in some generality.  In Section \ref{sec:discrete-weights}
we review some concepts of MVOP with respect to a discrete measure and difference operators acting on the MVOP. Following the work of \cite{Casper2}, we introduce the Fourier algebras associated with the weight matrix $W$. We proceed to distinguish between
\textit{weak} and \textit{strong} Pearson equations which are conditions on the weight matrix that allow us to find relations for the MVOP. We introduce backward and forward shift operators and we use these to obtain Rodrigues' formulas, difference operators and expressions for the square norms and the coefficients of the three term recurrence relation of the MVOP. The last two will be essential for the construction of dual families.

Section \ref{sec:duality} introduces the concept a sequence of MVOP that is \textit{dual} to another sequence of MVOP by extending the ideas of \cite{Douglas} to the matrix valued setting. We describe the relation between second order difference operators having the MVOP as eigenfunctions and the dual families. We also prove that, under mild conditions, the dual families satisfy orthogonality relations which involve a matrix valued weight related to the inverse of the square norm of the MVOP. We also establish the relation between the Fourier algebras of the MVOP and those of their duals.

The second part of the paper consists of Sections \ref{sec:charlier}--\ref{sec:7} and is devoted to the construction of a family of matrix valued Charlier polynomials. 
In Section \ref{sec:charlier} from a simple weak Pearson equation for a fairly general Charlier type weight we introduce two first order difference operators $\mathcal{D}, \mathcal{D}^\dagger$ which are each others adjoint. Following the ideas in \cite{DER}
 we find a nonlinear equation for the norms of the polynomials.

In Section \ref{sec:L} we specialize the general Charlier weight and we give an explicit expression for the 0-th square norm which determines all square norms. After that we find a LDU descomposition of the norm. In Section \ref{sec:pearson-charlier} we construct a one parameter family of matrix weights $W^{(\lambda)}$ for $\lambda \in \mathcal{V}=\mathbb{N}_{0}$ in such a way that the strong Pearson equation \eqref{eq:Pearson} holds true. As a consequence, we obtain a one parameter family of matrix weights with explicit shift operators, whose square norms and three-term recurrence relations are given explicitly. In Section
\ref{sec:7} we can then explicitly calculate the entries of the MVOP.

Finally, in Sections \ref{sec:dualOP}--\ref{sec:dual-dual} we construct explicit families of dual polynomials. In Section \ref{sec:dualOP} we
apply the results from Section
\ref{sec:duality} to find
three different sequences of MVOP
which are dual to the Charlier
MVOP. Finally in Section \ref{sec:dual-dual}
we iterate the duality process and construct two families of MVOP which are dual to one of the dual families.

\tikzstyle{decision} = [diamond, draw, fill=blue!20, 
	text width=4.5em, text badly centered, node distance=3cm, inner sep=0pt]
	\tikzstyle{block} = [rectangle, draw, 
	text width=10em, text centered, rounded corners, minimum height=4em]
	\tikzstyle{line} = [draw, -latex']
	\tikzstyle{cloud} = [draw, ellipse,fill=red!20, node distance=3cm,
	minimum height=2em]

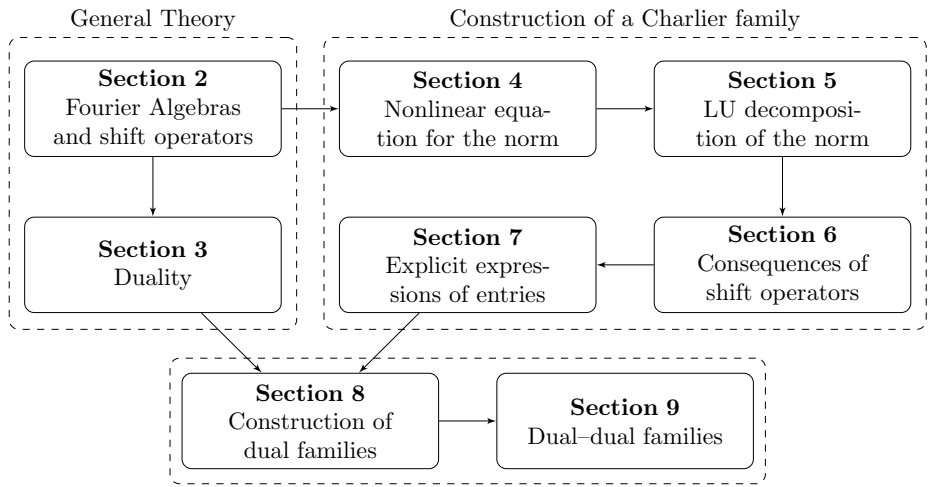
\begin{figure}[t]
    \centering
    \scalebox{0.9}{
    \begin{tikzpicture}[node distance = 2.3cm, auto]
	
	\node [block] (S2) {\textbf{Section 2} \\ Fourier Algebras and shift operators};
	\node [block, below of=S2] (S3) {\textbf{Section 3} \\ Duality};
	\path [line] (S2) -- (S3);
	
	\node[right of=S2] (r1){}; 
	\node [block, right of=r1] (S4) {\textbf{Section 4} \\ Nonlinear equation for the norm};
	\path [line] (S2) -- (S4);
	\node[right of=S4] (d2) {}; 
	\node [block, right of=d2] (S5) {\textbf{Section 5} \\ LU decomposition of the norm};
	\path [line] (S4) -- (S5);
	\node[block, below of=S5] (S6) {\textbf{Section 6} \\  Consequences of shift operators};
	\path [line] (S5) -- (S6);
	\node[block, below of=S4] (S7) {\textbf{Section 7} \\ Explicit  expressions of entries};
	\path [line] (S6) -- (S7);

	\node[below of=r1] (dr1){};
	\node[block, below of=dr1] (S8) {\textbf{Section 8} \\ Construction of dual families};
	\node[below of=d2] (dd2) {}; 
	\node[block, below of=dd2] (S9) {\textbf{Section 9} \\ Dual--dual families};
	\path [line] (S3) -- (S8);
	\path [line] (S7) -- (S8);
	\path [line] (S8) -- (S9);
	
	\node[rectangle,draw, minimum width = 4.2cm, 
	minimum height = 4.3cm, rounded corners, dashed] (l) at (0,-1.1) {};
	
	\node[rectangle,draw, minimum width = 8.8cm, 
	minimum height = 4.3cm, rounded corners, dashed] (r) at (6.9,-1.1) {};
	
	\node[rectangle,draw, minimum width = 8.7cm, 
	minimum height = 1.85cm, rounded corners, dashed] (l) at (4.6,-4.6) {};
	
	\node at (0,1.3) {General Theory};
	\node at (6.9,1.3) {Construction of a Charlier family};
	
	\end{tikzpicture}
	}
    \caption{This diagram shows the structure of the paper. Sections \ref{sec:discrete-weights}--\ref{sec:duality} are the heart of the paper and are devoted to developing the main ideas of duality. Sections \ref{sec:charlier}--\ref{sec:7} describe with great detail a family of matrix Charlier polynomials. The explicit expressions of the norms, three-term recurrence coefficients and the entries of these polynomials allow us finally to construct the dual families in Section \ref{sec:dualOP}.}
    \label{fig:diagram}
\end{figure}

We end this introduction with an example of a nontrivial family of $2\times 2$ matrix orthogonal polynomials with a dual family. This is a particular case of the first dual family considered in Section \ref{sec:dualOP}. Here we state the main properties or this example. The proofs of these follow from the general theory developed in Section \ref{sec:duality} and the considerations in Section \ref{sec:dualOP}.  We consider the one parameter family of weight matrices:
\begin{equation*}
	W^{(\lambda)}(x) = \frac{a^{x+\lambda}\lambda!}{2^\lambda x!} \begin{pmatrix} 1 & \frac{x+\lambda}{\sqrt{a}} \\ \frac{x+\lambda}{\sqrt{a}} & \frac{(x+\lambda)^2 + a( \lambda + 1)}{a} \end{pmatrix}, \qquad \lambda \in \N_0.
\end{equation*}
The unique family of monic orthogonal polynomials $P^{(\lambda)}_n(x)$ with respect to $W^{(\lambda)}$ is given explicitly by
$$P^{(\lambda)}_n(x) = (-a)^n c^{(a)}_n(x)I + \Omega_1(n)c^{(a)}_{n-1}(x)+\Omega_2(n)c^{(a)}_{n-2}(x),$$
where $I$ is the identity matrix and $\Omega_1, \Omega_2$ are given by
\begin{align*}
	\Omega_1(n) &= \frac{(-1)^n n a^{n-1}}{(\lambda+n+1)\sqrt{a}} \begin{pmatrix} (1-a-\lambda-n)\sqrt{a} & a \\ - (a^2+a\lambda+an+\lambda^2+2\lambda n+n^2-2a-\lambda-n) & (a + \lambda + n)\sqrt{a} \end{pmatrix},\\
	\Omega_2(n) &=  \frac{(-1)^n n a^{n-1}}{(\lambda+n+1)\sqrt{a}} \begin{pmatrix} (n-1)\sqrt{a} & 0 \\ (n-1) (a+\lambda+n) & 0 \end{pmatrix}.
\end{align*}
The orthogonality of this sequence can be directly verified by replacing the expression of $P_n^{(\lambda)}$ in \eqref{eq:d-weight} and using the orthogonality of the scalar Charlier polynomials. This simple expression for the sequence of monic polynomials is only valid in the case $2\times 2$. In Section \ref{sec:7} we give an explicit expression for an arbitrary dimension which involves sums of  products of two Charlier and a dual Hahn polynomial.

In order to describe a family of dual polynomials for $P_n^ {(\lambda)}$ we need the following matrices.
$$A = \begin{pmatrix} 1 & 0 \\ \frac{1}{\sqrt{a}} & 1 \end{pmatrix}, \qquad J=\begin{pmatrix} 1 & 0 \\ 0 & 2 \end{pmatrix},\qquad L_0
=
\left(
\begin{array}{cc}
 1 & 0 \\
 -\sqrt{a} & 1 \\
\end{array}
\right).$$
In Section \ref{sec:dualOP} we introduce three families of dual polynomials for $P_n^ {(\lambda)}$. The simplest family $(Q_{x,1}^ {(\lambda)})_x = (Q_x^ {(\lambda)})_x$ is given in Subsection \ref{subsec:primera-familia}.
 This family is associated to a difference operator \eqref{eq:Dfrak-8} and, by Remark \ref{rmk:recurrencefordual}, is defined by the three term recurrence relation
$$\mathcal{N} Q^{(\lambda)}_x(\mathcal{N}) = Q^{(\lambda)}_{x+1}(\mathcal{N}) -Q^{(\lambda)}_x(\mathcal{N})( J + x +\la(I+A)^{-1}) + Q^{(\lambda)}_{x-1}(\mathcal{N}) ax,$$
where $Q_0=I$, $Q_{-1} =0$. Here $\mathcal{N}$ is a matrix valued variable and we define the polynomials
evaluated at a matrix in
\eqref{eq:Qpoly}.
In Section \ref{sec:dualOP} we will show that the sequences $(P_n^{(\lambda)})_n$ and $(Q_x^ {(\lambda)})_x$ are dual with each other in sense of \eqref{eq:dualidad-introduccion}:
$$P_n^{(\la)}(x) = P^{(\la)}_n(0) Q_{x}^{(\la)}(\rho^{(\la)}(n)) \Upsilon^{(\la)}(x),$$
where
$$
\rho^{(\la)}(n)
=
\frac{1}{a^2 (\lambda +n+1)}
\left(
\begin{array}{cc}
 -n a(a+\lambda ) &
   na\sqrt{a} \\
 \sqrt{a}(a+\lambda ) (a \lambda +a-\lambda  n) &
   n a(a+\lambda )\\
\end{array}
\right)+a-\lambda -n-J,
$$
$$ P_n^{(\lambda)}(0) = \frac{a^{n-1} n (-1)^n}{\lambda+n+1} \begin{pmatrix} \frac{-\lambda n+a \lambda+a}{n} & \sqrt{a}\\ \frac{-(a^2+a \lambda+\lambda^2-a+\lambda n)}{\sqrt{a}} & \frac{2 a n+\lambda n+n^2+a \lambda+a}{n}
\end{pmatrix}, \quad 
\Upsilon^{(\la)}(x)= a^{-x}
\begin{pmatrix}
1 & 0 \\
-\frac{x}{\sqrt{a}} & 1
\end{pmatrix}.
$$
We observe that evaluation of the matrix polynomials $Q_x$ at the matrix valued rational function $\rho(n)$ is defined as \eqref{eq:Qpoly}. 
Moreover, the polynomials $Q_x^{(\lambda)}$ satisfy the following orthogonality relations:
$$
\langle Q^{(\la)}_y,Q^{(\la)}_x\rangle^d 
= 
\sum_{n=0}^\infty Q^{(\la)}_y(\rho(n))^\ast U(n) Q^{(\la)}_x(\rho(n)) 
= \mathscr{W}_x^{(\la)}\delta_{x,y},
$$
where $\sW_{x}^{(\la)} = 
(\Upsilon^{(\la)}(x)
W^{(\la)}(x)\Upsilon^{(\la)}(x)^\ast )^{-1}$, and the weight matrix $U(n)$ is
$$
(L_0^\ast)^{-1}(I+A^\ast)^{\la} U^{(\la)}(n) (I+A)^\la L_0^{-1}
=
\frac{e^{-a} 2^{\lambda } a^{n-\lambda }}{(\lambda +1)! n!}
\left(
\begin{array}{cc}
 \lambda +n+1 & \frac{n}{\sqrt{a}} \\
 \frac{n}{\sqrt{a}} & \frac{n^2}{a (\lambda +n+1)}+\frac{\lambda
   +1}{\lambda +n+2} \\
\end{array}
\right).
$$

\noindent
\textbf{Acknowledgements.} The authors are immensely grateful to Erik Koelink for countless useful comments. The authors would also like to thank Riley Casper for fruitful discussions at an earlier stage of the project.

The support of Erasmus+ travel grant is grate\-fully acknowledged. The work of Luc\'ia Morey and Pablo Rom\'an was supported by a FONCyT grant PICT 2014-3452 and by SeCyTUNC.

\section{Matrix valued orthogonal polynomials and discrete Fourier Algebras}
\label{sec:discrete-weights}

In this section we associate to the sequences of monic orthogonal polynomials with respect to \eqref{eq:d-weight} the so called Fourier algebras  of difference operators, following the work of Casper and Yakimov \cite{Casper2} in the case of weight matrices supported on real intervals.

\subsection{Difference operators}
We consider the difference operators $\eta^{j}$ and $\delta^j$ acting on the sequence of monic orthogonal polynomials $(P_n)_n$. They act from the right on the variable $x$ and from the left on the variable $n$ respectively:
$$(P_n \cdot \eta^j)(x)=P_n(x+j), \qquad (\delta^{j}\cdot P_n)(x)=P_{n+j}(x).$$
Here we assume that $P_{n-j}(x)=0$ if $n<j$. 
\begin{rmk}
It should be noted that the operation on the variable
$n$ is a slight abuse of notation. It would be more
precise to have
$$
(\delta^j\cdot P)_n(x) = P_{n+j}(x),
$$
but in the authors opinion this is more cumbersome
to read and the abuse will not likely lead to any
confusion.
\end{rmk}
For any matrix valued function $G:\mathbb{C}\to \mathbb{C}^{N\times N}$, we also consider the difference operators, $\Delta = \eta-1$ and $\nabla=1-\eta^{-1}$, acting on the variable $x$ from the right, defined by
$$G\cdot \Delta(x)=G(x+1) - G(x),\qquad  G\cdot \nabla (x)=G(x)-G(x-1).$$
\begin{rmk}
	Given a matrix polynomial $P$ of degree $n$, is easy to check that $P\cdot\Delta(x)$ is a polynomial of degree $n-1$. In general, for $k\leq n$, $P\cdot\Delta^k$ is a polynomial of degree $n-k$ and if $k>n$, $P\cdot\Delta^k$ is equal to zero.
\end{rmk}
For matrix valued functions $F, G:\mathbb{C} \to \mathbb{C}^{N\times N}$, we have the following analog of the summation by parts
\begin{equation}	
        \label{eq:suma-por-partes}
		\sum_{x = 0}^N G\cdot\Delta(x)F(x)= G(N+1)F(N+1)-G(0)F(0) - \sum_{x=0}^{N}G(x+1)F\cdot\Delta(x),
\end{equation}
and the following discrete analog of the Leibniz's rule for $\Delta$
\begin{equation}
		\label{thm:derivacion-producto-discreta}
		\ (FG\cdot\Delta)(x)= F(x+1)G\cdot\Delta(x) + F\cdot\Delta(x)G(x)=F\cdot\Delta(x)G(x+1)+F(x)G\cdot\Delta(x).
\end{equation}

\subsection{Discrete Fourier algebras}
\label{subsec:algebras-fourier}
In this section we introduce the Fourier algebra
formalism, following \cite{Casper2}, but for 
\textit{discrete} MVOP.
Let $\mathcal{M}_N$ and $\mathcal{N}_N$ be the algebras of difference operators acting on the polynomials from the left and right respectively:
\begin{align*}
\mathcal{M}_N
&=
\{ D=\sum_{j=-\ell}^{m} \eta^j F_j(x) : \quad F_j:\mathbb{C} \to
M_N(\mathbb{C}) \text{ is an entrywise\footnotemark rational function of $x$}\},
\\ 
\mathcal{N}_N
&=
\{M=\sum_{j=-t}^{s}G_j(n)\delta^{j} :  \quad
G_j:\mathbb{N}_0 \to M_N(\mathbb{C}) \text{ is a sequence}\}.
\end{align*}
\footnotetext{From now on we will refer to \textit{entrywise} rational matrix functions as just matrix valued rational functions.}
\begin{defn}\label{def:orden0}
When $D\in \mathcal{M}_N$
is such that $F_j(x)=0$ for all $x$ and all $j\neq 0$, then we call $D$ of \textit{order 0}. Similarly
for operators in $\mathcal{N}_N$.
\end{defn}

\begin{rmk}
	An operator $D\in\mathcal{M}_N$ acts on the variable $x$ from the right, and an operator $M \in\mathcal{N}_N$ acts on the variable $n$ from the left, i.e.
	
	$$
	P_n\cdot D(x)
	=
	\sum_{j=-\ell}^m 
	(P_n\cdot \eta^j)(x)F_j(x)
	=
	\sum_{j=-\ell}^m P_n(x+j)F_j(x),
	$$
	$$
	M\cdot P_n(x)
	=
	\sum_{j=\max(-t,-n)}^s G_j(n)
	(\delta^j\cdot P_n)(x)
	=
	\sum_{j=\max(-t,-n)}^s G_j(n)P_{n+j}(x).
	$$
\end{rmk}

The Fourier algebras are given by
\begin{align*}
\mathcal{F}_R(P)
&=
\{D\in \mathcal{M}_N : \exists M\in \mathcal{N}_N, \quad M\cdot P=P \cdot D\},
\\ 
\mathcal{F}_L(P)
&=
\{M\in \mathcal{N}_N : \exists D\in \mathcal{M}_N, \quad M\cdot P=P \cdot D\}.
\end{align*}

In the following proposition we establish an isomorphism between the left and right Fourier algebras.
\begin{prop}
	\label{prop:isom-Fourier-alg}
For all $M\in\mathcal{F}_L(P),$ there exists a unique $D\in\mathcal{F}_R(P)$ such that $M\cdot P=P \cdot D$. Conversely, for all $D\in\mathcal{F}_R(P),$ there exists a unique $M\in\mathcal{F}_L(P)$ such that $M\cdot P=P\cdot D.$
\end{prop}
\begin{proof}
Let $D\in\mathcal{F}_R(P)$ be such that $P\cdot D = 0$ for all $P\in \MN[x]$. In order to prove the first statement of the proposition it is enough to show that $D=0$. Assume that 
$$D=\sum_{j=-\ell}^{m}\eta^jF_j(x).$$
We first show that
\begin{equation}
	\label{eq:induccion-unicidad-operadores}
	\sum_{j=-\ell}^{m}j^k F_j(x)
	=0, 
	\qquad \text{ for all } k\in\mathbb{N}_0,
\end{equation}
by induction over $k$. For $k=0$, then \eqref{eq:induccion-unicidad-operadores} is easily checked since the condition $I\cdot D = 0$, where $I$ is the identity matrix, implies that $\sum_{j=-\ell}^{m}F_j(x)=0.$ If we assume \eqref{eq:induccion-unicidad-operadores} for all $0\leq i\leq k-1$ and we consider the scalar polynomial $p_k(x)=x^k+a_{k-1}x^{k-1}+\cdots+a_0$, then we have 
\begin{equation*}
0=(p_k\cdot D)(x)=\sum_{j=-\ell}^{m}(x+j)^kF_j(x)+\sum_{p=0}^{k-1}a_p\sum_{j=-\ell}^m(x+j)^pF_j(x),  
\end{equation*}
which can be written as 
\begin{equation}
\label{eq:paso-inductivo-unicidad-operadores}
0=(p_k\cdot D)(x)=\sum_{j=-\ell}^mj^kF_j(x)+ \sum_{q=1}^k \binom{k}{q}x^q\sum_{j=-\ell}^{m}j^{k-q}F_j(x)+\sum_{p=0}^{k-1}a_p\sum_{j=-\ell}^m(x+j)^pF_j(x).    
\end{equation}
Applying the induction hypothesis in the second and third sum of \eqref{eq:paso-inductivo-unicidad-operadores}, we get \eqref{eq:induccion-unicidad-operadores} as desired. 

The system \eqref{eq:induccion-unicidad-operadores} can be now written as $A(F_{-\ell}\cdots F_{0}\cdots F_{m})^{T}=0,$ where $A$ is a block $(m+\ell+1)\times(m+\ell+1)$ matrix with entries $A_{i,j}=(-\ell+j-1)^{i-1}I.$ This matrix is the transpose of an block Vandermonde matrix. Since all the
blocks are just multiples of the identity $I$, it is
easy to see this matrix is invertible. However it is also a special case of a block Vandermonde matrix for which we find a formula for the determinant in Lemma \ref{lem:BVM}. We conclude that $F_j=0$ for all $-\ell\leq j\leq m,$ and therefore $D=0$. Finally, if $M\in \mathcal{F}_L(P)$ is such that $M\cdot P =P\cdot D_1 = P\cdot D_2 $, then $P\cdot (D_1 -D_2) =0$ for all polynomials $P$ and thus $D_1=D_2$. The unicity of $M$ is analogous to \cite[Lemma 1]{DER}.
\end{proof}
It follows from Proposition \ref{prop:isom-Fourier-alg} that the map
\begin{equation}
\label{eq:iso}
\psi:\mathcal{F}_L(P)\to\mathcal{F}_R(P), \qquad \psi(M)=D,\quad \quad M\cdot P=P\cdot D,
\end{equation}
is a well defined algebra isomorphism; in \cite{Casper2} this isomorphism is called the \emph{generalized Fourier map}. Following \cite{Casper2}, we introduce the bispectral algebras $\mathcal{B}_L(P)$ and $\mathcal{B}_R(P)$:
\begin{equation*}
	\begin{split}
		\mathcal{B}_L(P)&=\{ M\in \mathcal{F}_L(P)\colon \, \mathrm{order}(\psi(M)) =0 \},\\
		\mathcal{B}_R(P)&=\{ D\in \mathcal{F}_R(P)\colon \, \mathrm{order}(\psi^{-1}(D)) =0 \}.
	\end{split}
\end{equation*}
It is easy to verify that the monic orthogonal polynomials satisfy the three term recurrence relation
\begin{equation}\label{eq:3TR}
x P_n(x) 
= 
P_{n+1}(x) + B_n P_n(x) + C_n P_{n-1}(x),
\end{equation}
where $B_n$ and $C_n$ are $N\times N$ matrices with the following properties:
\begin{equation}
 \label{eq:relations-B-C-gral}
B_n \cH_n = \cH_n B_n^\ast,\qquad C_n = \cH_{n} (\cH_{n-1})^{-1}.
\end{equation}
Let $\mathcal{L}$ be the difference operator in $n$ that corresponds
to tree term recurrence relation,
\begin{equation}
\label{eq:3TROP}
\mathcal{L}= \delta + B_n + C_n\delta^{-1}.
\end{equation}
Then the recurrence relation \eqref{eq:3TR} can be written as $\mathcal{L}\cdot P_n = P_n \cdot x$, so that $\mathcal{L}\in \mathcal{B}_L(P)\subset \mathcal{F}_L(P)$ and $x\in \mathcal{F}_R(P)$.

As shown in \cite{Casper2} for the continuous case, there is a natural adjoint in $\mathcal{N}_N$, namely
\begin{equation}
	\label{eq:definition-Mdagger}
	M^\dagger = \sum_{j=-t}^s \cH_n G_j(n-j)^\ast \cH_{n-j}^{-1}\delta^{-j} \quad \Longrightarrow \quad M^\dagger \cdot P_n(x)
	=
	\sum_{j=-t}^{\min(s,n)} \cH_n G_j(n-j)^\ast
	\cH_{n-j}^{-1} P_{n-j}(x).
\end{equation}

Given a pair $(M,D)$ with $M\in \mathcal{F}_L(P)$ and $D\in \mathcal{F}_R(P)$, a relation of the form
\begin{equation*}
	M\cdot P = P \cdot D,
	\qquad
	\text{ where } \quad M=\sum_{j=-t}^s G_j(n) \, \delta^j.
\end{equation*}
is called a \emph{ladder relation}. If the operator $M$ only contains nonpositive  (non\-negative) powers of $\delta$, we say that it is a \emph{lowering (raising) relation}.

\subsection{Weak Pearson equations and Fourier Algebras}
\label{sec:weak}
In this section we consider a class of weight matrices with distinguished elements in the left and right Fourier algebras. We say that a weight $W$ satisfies a system of \emph{weak Pearson equations} if there exists a nonempty set of integers $\{-\ell,\dots,  m\}$ and matrix polynomials $F_j$ and $\widetilde F_j$ for $j=-\ell,\ldots,m$ such that
\begin{equation}
	\label{eq:weak-pearson}
	F_j(x-j)W(x-j)=W(x)\widetilde F_j(x)^\ast,\qquad \text{for all } x\in \mathbb{N}_0, \text{ and } j=-\ell,\ldots, m.
\end{equation}
Note that if $\ell>0$, since $W$ vanishes on the negative integers, the weak Pearson equations \eqref{eq:weak-pearson} imply that
$$
F_j(x-j) = 0, \qquad -\ell \leq j \leq -1, \quad j \leq x \leq -1.
$$
Similarly for $m > 0$ we have
$$
\widetilde{F}_j(x) =0, 
\qquad 1 \leq j \leq m, \quad 0 \leq x \leq j-1.
$$
We associate to the weight $W$, the set of difference operators $\{D_j\}_{j=-\ell}^m$
\begin{equation}
	\label{eq:def-D-Ddagger-gral}
	D_j=\eta^jF_j(x).
\end{equation}
In the following propositions we show that $D_j$ has
an uncomplicated adjoint $D_j^\dagger$ and that both
are elements of $\mathcal{F}_R(P)$.
\begin{prop}
\label{prop:D-dagger}
	Let $D_j$ be the operator \eqref{eq:def-D-Ddagger-gral}, then its adjoint is given by
	$$
	D_j^{\dagger}= \eta^{-j} \, \widetilde F_j(x).
	$$
	In other words, if the weight $W$ satisfies 
	the weak Pearson equations 
	\eqref{eq:weak-pearson}, then
	$$
	\langle P\cdot D_j,Q \rangle_W = \langle P, Q\cdot D_j^\dagger \rangle_W,
	\qquad 
	j \in \{-\ell, \dots, m\},
	$$
	for all matrix valued polynomials $P,Q$.
\end{prop}
\begin{proof}
	It follows from the explicit expression of the inner product \eqref{eq:d-weight} that
	\begin{equation}
		\label{eq:inn-prod-D}
	\langle P\cdot D_j, Q \rangle_W=\sum_{x=0}^{\infty}  P(x+j)F_j(x)W(x)Q(x)^\ast.
	\end{equation}
	If we make the change of variables $y=x+j$ in \eqref{eq:inn-prod-D} and we use the weak Pearson equations \eqref{eq:weak-pearson}, we obtain
	\begin{align*}
		\langle P\cdot D_j, Q  \rangle_W&=\sum_{x=0}^{\infty}  P(x)F(x-j)W(x-j)Q(x-j)^\ast = \sum_{x=0}^{\infty}P(x)W(x)\left[ Q(x-j)\widetilde F_j(x)\right]^\ast\\
		&=\langle P, Q \cdot D_j^{\dagger}\rangle_W.
	\end{align*}
This completes the proof of the proposition.
\end{proof}

From here on out we will want a non-zero sum of
such operators
\begin{equation}\label{eq:nuevoD}
D = \sum_{j\in\mathcal{I}}\eta^j F_j(x),
\qquad
\mathcal{I} \subseteq \{-\ell, \dots , m\}.
\end{equation}

\begin{prop}\label{prop:ladder_sec2}
Let $\left(P_n\right)_{n}$ be the sequence of monic orthogonal polynomials with respect to a positive definite weight $W$  which satisfies the weak Pearson equations \eqref{eq:weak-pearson}. 
We can take any non-empty $\mathcal{I}$ as in \eqref{eq:nuevoD} and then the corresponding
$D$ and its adjoint 
$D^\dagger$ satisfy
$$
(P_n \cdot D)(x) = M\cdot P_n(x),\qquad (P_n \cdot D^\dagger)(x) = M^\dagger\cdot P_n(x),
$$
where $M$ is the difference operator 
$$M = \sum_{j=-t}^{s} G_j(n)\delta^{j}, \qquad G_j(n) = \langle P_n \cdot D, P_{n+j} \rangle \mathcal{H}_{n+j}^{-1},$$
and $M^\dagger$ is defined as in \eqref{eq:definition-Mdagger}. The summation bounds are given by $s, t$ be given by 
$$s = \max_{j \in \mathcal{I}} \deg F_j,\qquad t = \max_{j \in \mathcal{I}} \deg \widetilde  F_j.$$
Moreover $D, D^\dagger \in \mathcal{F}_R(P)$ and $M, M^\dagger \in \mathcal{F}_L(P)$.
\end{prop}

\begin{proof}
It follows from \eqref{eq:def-D-Ddagger-gral} that, $(P_n\cdot D)(x)$ is a polynomial of degree at most $n+s$ and from Proposition \ref{prop:D-dagger} that $(P_m\cdot D^\dagger)(x)$ is a polynomial of degree at most $m+t$. Then
\begin{equation*}
(P_n \cdot D)(x)
=
\sum_{k=0}^{n+s} S_k(n) P_k(x), \qquad (P_m \cdot D^\dagger)(x)
=
\sum_{k=0}^{m+t} T_k(m) P_k(x),
\end{equation*}
where $S_k(n) = \langle P_n \cdot D, P_k \rangle \mathcal{H}_k^{-1}$ and $T_k(m) = \langle P_m \cdot D^\dagger, P_k \rangle \mathcal{H}_k^{-1}$.
Therefore we have
$$
\langle P_n \cdot D , P_m \rangle
=
\langle P_n , P_m \cdot D^{\dagger} \rangle = 
\sum_{k=0}^{m+t} \langle P_n , P_k \rangle T_k(m)^\ast,
$$
so that $\langle P_n \cdot D , P_m \rangle = 0 $ for all $m<n-t$. 
\begin{equation*}
(P_n \cdot D)(x)
=
\sum_{k=n-t}^{n+s} S_k(n) P_k(x) = \sum_{j=-t}^{s} S_{n+j}(n) P_{n+j}(x) \qquad \Longrightarrow \qquad  (P_n \cdot D)(x) = M\cdot P_n(x).
\end{equation*}

On the other hand, $(P_m \cdot D^{\dagger})(x)$ is a polynomial of degree at most $m+t$ and therefore
\begin{align*}
P_m \cdot D^{\dagger}(x)&=\sum_{k=0}^{m+t} \langle P_m \cdot  D^{\dagger} , P_k \rangle \mathcal{H}_k^{-1} P_k(x)=\sum_{k=0}^{m+t} \langle P_k \cdot D , P_m \rangle^* \mathcal{H}_k^{-1} P_k(x)\\
& = \sum_{k=m-s}^{m+t} \sum_{j=-t}^{s} \langle G_j(k)P_{k+j} , P_m \rangle^\ast \mathcal{H}_k^{-1} P_k(x) \\
&= \sum_{j=-t}^{s} \mathcal{H}_m G_j(m-j)^\ast \mathcal{H}_{m-j}^{-1} P_{m-j}(x) = M^\dagger \cdot P_m(x).
\end{align*}
This completes the proof of the proposition.
\end{proof}

\subsection{Strong Pearson equations and Shift operators}
Let $\mathcal{V}$ be a subset of consecutive integers in $\mathbb{N}_0$ with at least two elements and let us define for $\mathcal{V}$ consider a family $\{ W^{(\lambda)}\}_{\lambda \in \mathcal{V}}$ of positive definite matrix valued weights with finite moments of all order and supported on $\mathbb{N}_0$. We denote by $\langle \cdot , \cdot \rangle^{(\lambda)}$ the matrix valued inner product induced by $W^{(\lambda)}$, i.e
\begin{equation}
	\label{eq:producto-interno-discreto-matricial}
	\langle P, Q \rangle ^{(\lambda)}= \sum_{x=0}^{\infty}P(x)W^{(\lambda)}(x)Q^{*}(x).
\end{equation}
We will denote by $P^{(\lambda)}_n$ the unique monic orthogonal polynomials with respect to $W^{(\lambda)}$, by $\mathcal{H}^{(\lambda)}_m$ the square norm and by $B_n^{(\lambda)}, C^{(\lambda)}_n$ the coefficients of the three term recurrence relation \eqref{eq:3TR}.

In the following theorem we will use the following notation: $\mathcal{V}^0 = \mathcal{V}\setminus (\max \mathcal{V})$ if $\mathcal{V}$ is bounded from above and $\mathcal{V}^0=\mathcal{V}$ otherwise. On the other hand, $\mathcal{V}$ is always bounded from below and so we take $\mathcal{V}_0 = \mathcal{V}\setminus (\min \mathcal{V})$. 
\begin{thm}[Shift operators]
	\label{thm:Pearson}
	Let $\mathcal{V}$, $\mathcal{V}^0$, $\mathcal{V}_0$ as above and let us consider a family $(W^{(\lambda)})_{\lambda \in \mathcal{V}}$. We assume, additionally, that there exist two families of polynomials $(\Phi^{(\lambda)})_{\lambda\in \mathcal{V}^0}$ and $(\Psi^{(\lambda)})_{\lambda\in \mathcal{V}^0}$ with $\deg \Phi^{(\lambda)} \leq 2$ and $\deg \Psi^{(\lambda)}\leq 1$ for all $\lambda\in \mathcal{V}^0$, such that
	\begin{equation}\label{eq:Pearson}
		W^{(\lambda+1)}(x)= W^{(\lambda)}(x)\Phi^{(\lambda)}(x), \qquad 	W^{(\lambda+1)} \cdot\nabla (x)
		= W^{(\lambda)}(x)\Psi^{(\lambda)}(x), \qquad \forall \lambda\in \mathcal{V}^0,
	\end{equation}
	with
	$$
	\Phi^{(\la)}(x)
	=
	\mathcal{K}_2^{(\la)} x^2
	+
	\mathcal{K}_1^{(\la)} x
	+
	\mathcal{K}_0^{(\la)} ,
	\quad
	\Psi^{(\la)}(x)
	=
	\mathcal{J}_1^{(\la)} x
	+
	\mathcal{J}_0^{(\la)}.
	$$
	Then
	\begin{enumerate}
		\item The difference operator $\Delta = \eta - I$ is a backward shift operator for $P_n^{(\lambda)}$:
		$$\Delta: L^{2}(W^{(\lambda)})\to L^{2}(W^{(\lambda+1)}), \qquad P^{(\lambda)}_n\cdot\Delta(x) = n P^{(\lambda+1)}_{n-1}(x),\qquad \forall \lambda \in \mathcal{V}^0.$$
		\item The operator $\Delta$ has an adjoint $S^{(\lambda)}: L^{2}(W^{(\lambda+1)}) \to  L^{2}(W^{(\lambda)})$, explicitly given by
		$$
		S^{(\lambda)} = -( \nabla\Phi^{(\lambda)}(x)^\ast + \eta^{-1} \Psi^{(\lambda)}(x)^\ast )
		= - \Phi^{(\lambda)}(x)^\ast
		+
		\eta^{-1}
		( 
		\Phi^{(\lambda)}(x)^\ast
		-
		\Psi^{(\lambda)}(x)^\ast
		),
		$$
		such that
		$$\langle P\cdot\Delta, Q\rangle^{(\lambda+1)}=\langle P, Q\cdot S^{(\lambda)}\rangle^{(\lambda)}$$
		for all matrix polynomials $P, Q.$
		\item The operator $S^{(\lambda)}$ is a forward shift operator, i.e.
		$$P_{n-1}^{(\lambda+1)}\cdot S^{(\lambda)}=G^{(\lambda)}_{n}P_{n}^{(\lambda)},
		\qquad 
		G^{(\lambda)}_n= -(n-1)\, \mathcal{K}_2^{(\la)\ast} - \mathcal{J}_1^{(\la)\ast}.$$
	\end{enumerate}
\end{thm}

\begin{rmk}
We will see
that $G_n^{(\la)}$
is in fact invertible
in the proof of
Theorem \ref{thm:formula-de-rodrigues-discreta-matricial}
in equation
\eqref{eq:pearson-normas}.
\end{rmk}

\begin{proof}
For (1), we take $n\geq 1$. The sequence $( P^{(\lambda)}_n\cdot\Delta)_n$ satisfies 
\begin{equation}
	\label{eq:ortognalidad-derivadas-discretas}
	\sum_{x=0}^\infty((P_{n}\cdot\Delta)W^{(\lambda+1)}(P_{m}\cdot\Delta)^{*})(x)=\lim_{\mathcal{N}\to\infty}\left[ \sum_{x = 0}^\mathcal{N}((P_{n}\cdot\Delta)W^{(\lambda+1)}(P_{m}\cdot\Delta)^{*})(x)\right] =0, \quad m<n.
\end{equation}
This is proved by a lengthy computation involving summation by parts \eqref{eq:suma-por-partes} and the Leibniz rule \eqref{thm:derivacion-producto-discreta}. Since the sequence $P_n \cdot \Delta$ is a matrix polynomial of degree $n-1$ with leading coefficient equal to $nI$, \eqref{eq:ortognalidad-derivadas-discretas} leads to $P_n^{(\lambda)}\cdot\Delta = nP_{n-1}^{(\lambda+1)}$.

\medskip

For (2) we first observe that
\begin{align*}
	\langle P\cdot \Delta, Q \rangle&^{(\lambda+1)} =	\lim_{\mathcal{N}\rightarrow\infty}\sum_{x=0}^\mathcal{N} \left[ P(x+1)W^{(\lambda+1)}(x)Q(x)^\ast - P(x)W^{(\lambda+1)}(x)Q(x)^\ast \right] \\
	 &=
	 \lim_{\mathcal{N}\rightarrow\infty}\sum_{x=0}^\mathcal{N}  P(x+1)W^{(\lambda+1)}(x)Q(x)^\ast -
	 \lim_{\mathcal{N}\rightarrow\infty}
	 \sum_{x=-1}^{\mathcal{N}-1}P(x+1)W^{(\lambda+1)}(x+1)Q(x+1)^\ast  \\
	&= -P(0)W^{(\lambda+1)}(0) Q(0)^\ast + \sum_{x=0}^\infty P(x+1)\left[  W^{(\lambda+1)}(x)Q(x)^\ast - W^{(\lambda+1)}(x+1)Q(x+1)^\ast\right] \\
	&= -P(0)W^{(\lambda+1)}(0) Q(0)^\ast - \sum_{x=0}^\infty P(x+1) (W^{(\lambda+1)}Q^\ast \cdot \Delta)	(x).
\end{align*}
Applying summation by parts, and shifting the index of summation $x\mapsto x-1$, we get
\begin{multline}
		\label{eq:thmpears_0}
\langle P\cdot \Delta, Q \rangle^{(\lambda+1)}  = P(0)W^{(\lambda+1)}(0) (Q^\ast\cdot \Delta)(-1) + P(0)(W^{(\lambda)}\cdot \Delta)(-1) Q^\ast(-1) -P(0)W^{(\lambda+1)}(0) Q(0)^\ast \\
- \sum_{x=0}^\infty P(x) \left[W^{(\lambda+1)}(x) (Q^\ast \cdot \Delta)	(x-1) + (W^{(\lambda+1)}\cdot\Delta)(x-1) Q^\ast(x-1)\right].
\end{multline}
Now, since $W^{(\lambda+1)}(x)=0$ for all $x\in \mathbb{Z}_{<0}$, we have
\begin{equation}
	\label{eq:thmpears_1}
-P(0)W^{(\lambda+1)}(0) Q(0)^\ast + P(0)W^{(\lambda+1)}(0) (Q^\ast\cdot \Delta)(-1) + P(0)(W^{(\lambda)}\cdot \Delta)(-1) Q^\ast(-1) =0,
\end{equation}
and from the Pearson equation \eqref{eq:Pearson}
\begin{multline}
	\label{eq:thmpears_2}
W^{(\lambda+1)}(x) (Q^\ast \cdot \nabla)(x) + 
(W^{(\lambda+1)}\cdot \nabla)(x) Q^\ast(x-1) \\ 
= W^{(\lambda)}(x)\left[(Q\cdot \nabla)(x) \Phi^{(\lambda)}(x)^\ast + Q(x-1)\Psi^{(\lambda)}(x)^\ast\right]^\ast.
\end{multline}
If we replace \eqref{eq:thmpears_1}, \eqref{eq:thmpears_2} in \eqref{eq:thmpears_0}, we obtain (2).

\medskip

Finally we prove (3). First we observe that, since $\Phi^{(\lambda)}$ and $\Psi^{(\lambda)}$ are polynomials of degree at most two and one, respectively, part (2) of the theorem implies that $P^{(\lambda+1)}_{n-1} \cdot S^{(\lambda)}$ is a polynomial of degree, at most $n$. Therefore there exist matrices $\mathcal{G}_{n}^{(\lambda)}$ such that
$$
P_{n-1}^{(\lambda+1)}\cdot S^{(\lambda)}
=
\mathcal{G}_{n}^{(\lambda)}P_{n}^{(\lambda)}+\ldots+\mathcal{G}^{(\lambda)}_{0}P_{0}^{(\lambda)}.
$$
Now, using that $\Delta $ and $S^{(\lambda)}$ are each others adjoints, 
$$
\mathcal{G}_m^{(\la)}
=
\langle P_{n-1}^{(\lambda+1)}\cdot S^{(\lambda)},P_{m}^{(\lambda)}\rangle^{(\lambda)}
=
\langle P_{n-1}^{(\lambda+1)},P_{m}^{(\lambda)}\cdot\Delta\rangle^{(\lambda+1)}
=
\langle P_{n-1}^{(\lambda+1)},mP_{m-1}^{(\lambda+1)}\rangle^{(\lambda+1)}=0,
$$
for all $n\neq m$. Now we denote
the only nonzero coefficient
$G_n^{(\la)} = \mathcal{G}_n^{(\la)}$
to get
$$
P_{n-1}^{(\lambda+1)}\cdot S^{(\lambda)}=G^{(\lambda)}_{n}P_{n}^{(\lambda)}.
$$
The expression of the matrices $G_n^{(\lambda)}$ follows by comparing the leading coefficients in the equation above. This completes the proof of the theorem.
\end{proof}

\begin{rmk}
The requirements on the weight in 
\eqref{eq:Pearson} are called
\textit{strong} Pearson equations. They are called
 Pearson because, like the weak Pearson equations, 
 they reduce to scalar Pearson equations in the 
 scalar case. And they are called strong because
 they imply a system of very specific weak Pearson
 equations with $j \in \{0,1\}$, cf. \eqref{eq:weak-pearson},
 $$
 F_0(x) = \widetilde{F}_0(x) = F_1(x) = \Phi^{(\la)}(x)^\ast,
 \qquad
 \widetilde{F}_1(x) =\Phi^{(\la)}(x)^\ast - \Psi^{(\la)}(x)^\ast,
 $$
 in addition to requiring that 
 $W^{(\la)}(x)\Phi^{(\la)}(x)$ must also be a weight.
 One such weak Pearson is
 obtained by combining two
 strong Pearson equations, so that
 we have only one weight
 $W^{(\la)}$
 $$
 W^{(\la)}(x)\Phi^{(\la)}(x)
 -
 W^{(\la)}(x-1)\Phi^{(\la)}(x-1)
 =
 W^{(\la)}(x)\Psi^{(\la)}(x),
 $$
 and subsequently using the fact
 that $W^{(\la)}(x)\Phi^{(\la)}(x)$
 is again a weight and so must
 be a symmetric matrix. This
 allows us to rearrange the above
 equation to
 $$
 \Phi^{(\la)}(x-1)^\ast W^{(\la)}(x-1)
 =
   W^{(\la)}(x)
  \left(\Phi^{(\la)}(x)
  -
  \Psi^{(\la)}(x)
  \right)
 $$
 The other weak Pearson equation follows directly from the fact that
 $W^{(\la)}(x)\Phi^{(\la)}(x)$ is a
 symmetric matrix,
 $$
 \Phi^{(\la)}(x)^\ast W^{(\la)}(x)
 =
 W^{(\la)}(x)\Phi^{(\la)}(x).
 $$
 These weak Pearson equations then leads to a few additional operators
$$
E_1 = \Phi^{(\la)}(x)^\ast = E_1^\dagger,
$$
$$
E_2 = \eta \Phi^{(\la)}(x)^\ast,
\qquad
E_2^\dagger
=
\eta^{-1}(\Phi^{(\la)}(x)^\ast-\Psi^{(\la)}(x)^\ast),
$$
which are all elements of
$\mathcal{F}_R(P)$ due to Proposition
\ref{prop:ladder_sec2}.
\end{rmk}

\subsection{Applications of the shift operators}
\label{subsec:formula-rodrigues}
In this subsection we obtain explicit structural formulas for the monic polynomials $(P_{n}^{(\lambda)})_n$ by using the shift operators. Rodrigues formulas for discrete matrix orthogonal polynomials were also obtained by a different method in \cite{MR3239839}. For the rest of this paper we assume $\mathcal{V}=\N_0$, so that we have an infinite family of matrix valued weights $W^{(\lambda)}$.

\begin{thm}
\label{thm:formula-de-rodrigues-discreta-matricial}
Let $(W^{(\lambda)})_{\lambda\in \N_0}$ be a family weights, as in Theorem \ref{thm:Pearson}. Then
\begin{enumerate}
\item Square norms: The square norm $\mathcal{H}_n^{(\lambda)}$ of $P_n^{(\lambda)}$ is equal to
\begin{equation}
\label{eq:pearson-normas}
\mathcal{H}_n^{(\lambda)}=n!\mathcal{H}_0^{(\lambda+n)}((G_1^{(\lambda+n-1)})^*)^{-1}\ldots((G_n^{(\lambda)})^*)^{-1}. \end{equation}

    \item Rodrigues formula: The monic orthogonal polynomials respect to $W^{(\lambda)}$ are given by
$$P_{n}^{(\lambda)}(x)=(-1)^{n}(G^{(\lambda+n-1)}_{1}\ldots G^{(\lambda)}_{n})^{-1}W^{(\lambda+n)}\cdot\nabla^{n}(x)(W^{(\lambda)}(x))^{-1},$$
for all $n$ and $\lambda$ in $\mathbb{N}_{0}$ where the matrices $G^{(\lambda)}_n$ are given in Theorem \ref{thm:Pearson}, (3).

\item Three term recurrence relation: The coefficients of the recurrence \eqref{eq:3TR} are given by
$$B^{(\lambda)}_n=nX_1^{(\lambda+n-1)}-(n+1)X_1^{(\lambda+n)} + n, \qquad     C^{(\lambda)}_n=\mathcal{H}^{(\lambda)}_n(\mathcal{H}_{n-1}^{(\lambda)})^{-1},$$
where $X_n^{(\lambda)}$ denotes the subleading coefficient of $P_n^{(\lambda)}$.
\item The second order operators $S^{(\lambda-1)} \Delta$ and $\Delta S^{(\lambda)}$ have the polynomials $P_n^{(\lambda)}$ as eigenfunctions. More precisely
\begin{equation}\label{eq:rec34}
(P_n^{(\lambda)} \cdot S^{(\lambda-1)} \Delta)(x)= (n+1) G_{n+1}^{(\lambda-1)} P_n^{(\la)}(x), \qquad  (P_n^{(\lambda)} \cdot \Delta S^{(\lambda)})(x) = nG_n^{(\lambda)} P_n^{(\la)}(x).
\end{equation}

\end{enumerate}
\end{thm}
\begin{proof}
In order to prove (1) we observe that
for $\la \in \mathcal{V}_0$,
\begin{multline}
	\label{eq:relacion-Hn-Hn-1}
	\mathcal{H}_n^{(\lambda)}= \langle P_n^{(\lambda)}, P_n^{(\lambda)}\rangle^{(\lambda)}= \frac{1}{n+1}\langle P_{n+1}^{(\lambda-1)}\cdot\Delta, P_n^{(\lambda)}\rangle^{(\lambda)}=\frac{1}{n+1}\langle P_{n+1}^{(\lambda-1)}, P_n^{(\lambda)}\cdot S^{(\lambda-1)}\rangle^{(\lambda-1)} \\
	=\frac{1}{n+1}\langle P_{n+1}^{(\lambda-1)}, P_{n+1}^{(\lambda-1)}\rangle^{(\lambda-1)}(G_{n+1}^{(\lambda -1)})^*=\frac{1}{n+1}\mathcal{H}_{n+1}^{(\lambda-1)}(G_{n+1}^{(\lambda -1)})^*.
\end{multline}
Iterating this we obtain \eqref{eq:pearson-normas}. We note that $\eqref{eq:relacion-Hn-Hn-1}$ implies that the matrices $G_{n}^{(\lambda)}$ are invertible for all $n,\lambda$.

\medskip

We prove (2) by induction on $n$, by first
proving for $Q$ any matrix polynomial that 
\begin{equation}
\label{eq:eqinduccionformuladerodrigues}
Q\cdot (S^{(\lambda+n-1)}\cdots S^{(\lambda)})=(-1)^{n}(QW^{(\lambda+n)}\cdot\nabla^{n})(W^{(\lambda)})^{-1} \quad \text{for all $n\in\mathbb{N}$.}
\end{equation}
We proceed in the case $n=1$, using the definition of the operator $S^{(\lambda)}$ we get
\begin{equation}
\label{eq:paso-inductivo-rodrigues-discreta-1}
Q\cdot S^{(\lambda)}(x)=-\left[Q\cdot\Delta(x-1)(\Phi^{(\lambda)}(x))^{*} + Q(x-1)(\Psi ^{(\lambda)}(x))^{*}\right].
\end{equation}
Using equations \eqref{eq:Pearson} and \eqref{thm:derivacion-producto-discreta}, the right hand side of \eqref{eq:paso-inductivo-rodrigues-discreta-1} becomes
\begin{multline*}
-\left[(Q\cdot\Delta)(x-1) W^{(\lambda+1)}(x)+ Q(x-1) (W^{(\lambda+1)}\cdot\Delta)(x-1)\right]W^{(\lambda)}(x)^{-1} \\
= -( QW^{(\lambda+1)}\cdot\Delta)(x-1)(W^{(\lambda)}(x))^{-1}.
\end{multline*}
From the definition of the operator $\nabla$, \eqref{eq:paso-inductivo-rodrigues-discreta-1} is written as
\begin{equation}
\label{eq:otroS}
Q\cdot S^{(\lambda)}(x)= -(QW^{(\lambda+1)}\cdot\nabla)(x)(W^{(\lambda)}(x))^{-1},
\end{equation}
completing the case $n=1$. Now assume that \eqref{eq:eqinduccionformuladerodrigues} holds true for $n$. Then
applying the inductive hypothesis we find
\begin{align}
\label{eq:paso-inductivo-rodrigues-discreta}
Q\cdot(S^{(\lambda+n)}\ldots S^{(\lambda)})(x)&=(Q\cdot S^{((\lambda+1)+n-1)}\ldots S^{(\lambda+1)})\cdot S^{(\lambda)}(x) \nonumber \\
&= (-1)^n \left( (QW^{(\lambda+n+1)} \cdot \nabla^n)(x) W^{(\lambda+1)}(x)^{-1}\right)\cdot S^{(\lambda)} (x) \nonumber\\
&= (-1)^{n+1} \left( (QW^{(\lambda+n+1)}\cdot \nabla^{n+1})(x)\right)W^{(\lambda)}(x)^{-1},
\end{align}
where in the last line we have
used the action of $S^{(\la)}$
as given in \eqref{eq:otroS}.
This proves that \eqref{eq:eqinduccionformuladerodrigues} holds for all $n\in \mathbb{N}$.

Now, from the forward shift relation $(P_{n-1}^{(\la +1)} \cdot S^{(\lambda)}) = G^{(\lambda)}_n P_{n}^{(\la)}$ and \eqref{eq:paso-inductivo-rodrigues-discreta} with $Q$ replaced by the identity matrix $I$, we get
$$
G^{(\lambda+n-1)}_{1}\ldots G^{(\lambda)}_{n}P_{n}^{(\lambda)}
= 
I\cdot (S^{(\lambda)}\ldots S^{(\lambda+n-1)})
=
(-1)^{n}(W^{(\lambda+n)}\cdot\nabla^{n})(W^{(\lambda)})^{-1} .
$$
This concludes the proof of (2).

\medskip

The expression for $C_n^{(\lambda)}$ was already observed in \eqref{eq:relations-B-C-gral}. Denoting the subleading coefficient of $P^{(\lambda)}_n(x)$ by $X^{(\lambda)}_n$, we obtain from \eqref{eq:3TR} that
$B^{(\lambda)}_{n}= X^{(\lambda)}_n-X^{(\lambda)}_{n+1}$. Now the backward shift relation $P^{(\lambda)}_n(x)\cdot\Delta = n P^{(\lambda+1)}_{n-1}(x)$, gives a  a simple recurrence relation for $X_{n}^{(\lambda)}$:
\begin{equation*}
(n-1)X^{(\lambda)}_n + \frac{n(n-1)}{2} = n X^{(\lambda+1)}_{n-1}.
\end{equation*}
Iterating this equation we get
$$X^{(\lambda)}_n=nX_1^{(\lambda+n-1)} - \frac{n(n-1)}{2}.$$
As a consequence 
\begin{equation}
	\label{eq:relacion-Bn-X1}
	B^{(\lambda)}_n=nX_1^{(\lambda+n-1)}-(n+1)X_1^{(\lambda+n)} + n.    
\end{equation}
The proof of (3) is then complete. Finally, (4) follows immediately from Theorem \ref{thm:Pearson}.
\end{proof}

\section{Duality for Matrix valued orthogonal polynomials}
\label{sec:duality}
In this section we introduce a notion of duality for matrix valued orthogonal polynomials. This can be viewed as an extension of the notion of duality for scalar orthogonal polynomials developed in \cite{Douglas}. 

\subsection{Matrix valued dual polynomials}
We consider a sequence $(Q_x)_x$ of matrix valued polynomials of matrix argument $\mathcal{N}$, such that $\deg Q_x = x$ for all $x\in \mathbb{N}_0$ and such that the variable $\mathcal{N}$ multiplies from the left, i.e. there exist matrices $A_{x,x}, \ldots A_{x,0} \in \MN$ such that
\begin{equation}
\label{eq:Qpoly}
Q_x(\mathcal{N}) = \mathcal{N}^x A_{x,x} + \mathcal{N}^{x-1}A_{x,x-1} +\ldots + A_{x,0}.
\end{equation}
If the variable $\mathcal{N}$ is a multiple of the identity matrix, say $\mathcal{N}=nI$, we will write $Q_x(n)=Q_x(nI)$. Additionally, we assume that there exists a $M_N(\mathbb{C})$-valued function $\rho$, which will be called the \emph{eigenvalue} function, such that the following three term recurrence relation holds true:
\begin{equation}
 \label{eq:dual-recurrence}
 \rho(n)Q_x(\rho(n)) = Q_{x+1}(\rho(n)) + Q_x(\rho(n)) \mathcal{Y}_x + Q_{x-1}(\rho(n))\mathcal{Z}_x, \quad x\in\mathbb{N}_0,  \qquad Q_0=I,
\end{equation}
and $Q_{-1}=\mathcal{Z}_0=0$. In order to ensure that the sequence of monic polynomials $(Q_x)_x$ satisfying \eqref{eq:dual-recurrence} is unique, we need to impose some condition on the eigenvalue functions. We note that in \eqref{eq:dual-recurrence}, the coefficients $\mathcal{Y}_x$ and $\mathcal{Z}_x$ multiply from the right, while the function $\rho(n)$ multiplies from the left.
\begin{assumption}
\label{assumption:rho}
Every entry of $\rho(n)$ is a rational function of $n$
without poles in $\mathbb{N}_0$
and for every $x\in \N_0$, there exist $k_0,\ldots k_x \in \N_0$ such that the block Vandermonde matrix
\begin{equation}
\label{eq:vandermonde}
\begin{pmatrix} I & \rho(k_0) & \rho(k_0)^2 & \cdots & \rho(k_0)^x \\
I & \rho(k_1) & \rho(k_1)^2 & \cdots & \rho(k_1)^x \\
\vdots & \vdots & \vdots & \ddots & \vdots \\
I & \rho(k_x) & \rho(k_x)^2 & \cdots & \rho(k_x)^x
\end{pmatrix},
\end{equation}
is invertible. In particular this implies that the $\rho(n)'s$ are different and invertible for infinitely many $n\in \N_0$. In Lemma \ref{lem:BVM} we describe a large family of  eigenvalue functions satisfying Condition \ref{assumption:rho}.
\end{assumption}

\begin{lem}
\label{lem:Unique-seq-Q}
Let $(Q_x)_x$ be a sequence of monic matrix polynomials such that \eqref{eq:dual-recurrence} holds for all $n\in \N_0$. If the eigenvalue function $\rho(n)$ satisfies Condition \ref{assumption:rho}, then $(Q_x)_x$ is unique.
\end{lem}
\begin{proof}
Let $(\widetilde Q_x)_x$ be the unique sequence of monic matrix valued polynomials defined by the three term recurrence relation
\begin{equation*}
n\widetilde Q_x(n) = \widetilde Q_{x+1}(n) + \widetilde Q_x(n) \mathcal{Y}_x + \widetilde Q_{x-1}(n)\mathcal{Z}_x, \qquad \widetilde Q_0=Q_0=I.
\end{equation*}
Then, $\widetilde Q_x(\rho(n))$ satisfies the recurrence \eqref{eq:dual-recurrence} as well. Proceeding recursively in $x$, we obtain
$$Q_x(\rho(n)) = \widetilde Q_x(\rho(n)),\qquad \forall x,n\in \N_0.$$
Now the proof of the lemma will be complete if we show that if $P$ is a matrix polynomial which satisfies $P(\rho(n))=0$ for all $n\in \N_0$, then $P=0$. Suppose that $P(n)=n^x A_x + \cdots +A_0$, with $A_x \neq 0$. Then by Condition \ref{assumption:rho}, there exist $k_0,\ldots k_x\in \N_0$ such that the block Vandermonde matrix \eqref{eq:vandermonde} is invertible. Then the condition
$$P(\rho(k_j))=\rho(k_j)^x A_x + \cdots +\rho(k_j)A_1 + A_0 =0,\qquad j=0,\ldots,x,$$
implies that $A_x=\cdots = A_1 =A_0=0$, and this completes the proof of the lemma.
\end{proof}

We are now ready to introduce the notion of duality for matrix orthogonal polynomials. We consider a weight matrix $W$ and its sequence of monic orthogonal polynomials $(P_n)_n$ as in \eqref{eq:monic-op-def}.
\begin{defn}
\label{def:duality}
 Let $(Q_x)_x$ be a family of matrix polynomials satisfying \eqref{eq:dual-recurrence} and assume that the eigenvalue $\rho(n)$ satisfies Condition \ref{assumption:rho}. Let $M_1(n), M_2(x)$ be invertible matrix valued functions on $\N_0$ and assume, additionally, that $M_2(x+1)^{-1}M_2(x)$ is a rational function.  We say that the 4--tuple $(Q_x, M_1, M_2, \rho)$ is a dual family for $(P_n)_n$ if \begin{equation}
	\label{eq:duality-definition}
	P_n(x)= M_1(n)Q_x(\rho(n))M_2(x), \qquad \text{for all } \, x,n\in \mathbb{N}_0.
\end{equation}
If there is no place for confusion, we will simply write $(Q_x)_x$.
\end{defn}
The dual relation \eqref{eq:duality-definition} could also be written in the form $M_1(n)^{-1} P_n(x)= Q_x(\rho(n))M_2(x)$, where the left hand side is a polynomial in $x$ and the right hand side a polynomial in $n$ evaluated at $\rho(n)$. Note that \eqref{eq:duality-definition} together with the invertibility of $M_1(n)$ and $M_2(n)$ implies that $P_n(0)$ is invertible for all $n\in \N_0$ and $Q_x(\rho(0))$ is invertible for all $x\in \N_0$.

There is an equivalence relation between dual families. Let $(Q_x, M_1, M_2, \rho)$ be a dual family and let $R$ be any invertible constant matrix. We denote by $\breve Q_x$ the polynomial
$$\breve Q_x(\mathcal{N})= \mathcal{N}^x (R^{-1}A_{x,x} R) + \mathcal{N}^{x-1} (R^{-1}A_{x,x-1}R) +\cdots + R^{-1}A_{x,0} R,$$
so that $R^{-1} Q_x(\mathcal{N}) R  = \breve Q_x(R^{-1}\mathcal{N}R)$. We can now rewrite the duality condition \eqref{eq:duality-definition} as
\begin{equation*}
	P_n(x)= (M_1(n)R)(R^{-1}Q_x(\rho(n))R)(R^{-1}M_2(x)) = (M_1(n)R) \breve Q_x(R^{-1} \rho(n) R) (R^{-1}M_2(x)),
\end{equation*}
for all $x,n\in \mathbb{N}_0$. We will say that two dual families $(Q_x,M_1,M_2,\rho)$ and $(S_x,N_1,N_2,\nu)$ are equivalent if there exists a constant matrix $R$ such that 
\begin{equation}
\label{eq:dual-equivalence-relation}
S_x(\mathcal{N}) = R^{-1} Q_x(\mathcal{N}) R = \breve Q_x (R^{-1}\mathcal{N}R), \qquad N_1 = M_1R, \qquad N_2 = R^{-1}M_2,\qquad \nu = R^{-1} \rho R.
\end{equation}
It follows from Corollary \ref{cor:BVM} that if $\rho$ satisfies
Condition \ref{assumption:rho} that $\nu$ does as
well.
The matrix polynomials $(\breve Q_x)_x$ satisfy a three term recurrence relation of the form
\begin{equation}
    \label{eq:recurrence_equivalentQ}
     \nu(n)\breve Q_x(\nu(n)) = \breve Q_{x+1}(\nu(n)) + \breve Q_x(\nu(n)) R^{-1}\mathcal{Y}_xR+ \breve Q_{x-1}(\nu(n)) R^{-1} \mathcal{Z}_xR, 
\end{equation}
for all $x\in\mathbb{N}_0$, with the initial condition $Q_0=I.$

\begin{rmk}
The group of all complex invertible matrices $\GL(n,\C)$ acts on the set of dual families $(Q_x, M_1, M_2, \rho)$ by 
$$(Q_x, M_1, M_2, \rho) \cdot R \mapsto (R^{-1} Q_x R,  M_1R, R^{-1}M_2, R^{-1} \rho R).$$
The equivalence classes of dual families correspond to the orbits of this action. 

Given a family of dual polynomials $(Q_x, M_1, M_2, \rho)$ there will always be a representative $(S_x,N_1,N_2,\nu)$ of its orbit with $N_1(0) = N_2(0) = I$, which is obtained by choosing $R=M_1(0)^{-1}= M_2(0)$. The last equatlity is due to the monicity of $(Q_x)_x$.
\end{rmk}

Without any loss of generality we will restrict ourselves to dual families  $(Q_x,M_1,M_2,\rho)$ with the condition $M_1(0)=M_2(0)=I$. In such a case it follows from \eqref{eq:duality-definition}  that
$$M_1(n) = P_n(0),\qquad M_2(x)=Q_x(\rho(0))^{-1}.$$

\subsection{Characterization of dual families}
Let $(P_n)_n$ be the family on monic polynomials with respect to a weight matrix $W$. As in the scalar case discussed in \cite{Douglas},  the monic polynomials $(P_n)_n$ may have many duals. In this section we identify the dual families with certain second order difference operators having the polynomials $(P_n)_n$ as eigenfunctions.  Let $(Q_x,M_1,M_2,\rho)$ be a dual family and assume that $M_1(0)=M_2(0)=I$. Using \eqref{eq:duality-definition} we can rewrite the three term recurrence relation \eqref{eq:dual-recurrence} in terms of the polynomials $P_n$ in the following way:
\begin{multline}
\label{eq:diferencias-x-rho}
P_n(x+1) M_2(x+1)^{-1}M_2(x)+ P_n(x) M_2(x)^{-1}\mathcal{Y}_x M_2(x)+ P_n(x-1) M_2(x-1)^{-1}\mathcal{Z}_x M_2(x) \\
= M_1(n)\rho(n)M_1(n)^{-1} P_n(x),
\end{multline}
for all $x\in \N, n\in\N_0$, and 
\begin{equation}
\label{eq:diferencias-x-rho-at0}
P_n(1) M_2(1)^{-1}+ P_n(0) \mathcal{Y}_0  
= M_1(n)\rho(n)M_1(n)^{-1} P_n(0),
\end{equation}
for $x=0$. Next, we show that the these relations can be viewed as the action of an element of the bispectral algebra $\mathcal{B}_R(P)$ on the polynomials $(P_n)_n$.
\begin{lem}
\label{lem:dual-to-bispectral}
Let $(Q_x)_x$ be a dual family for $(P_n)_n$ as in Definition \ref{def:duality}. Then there exist rational functions $F_1(x)$, $F_0(x)$ and $F_{-1}(x)$ such that
\begin{equation}
\label{eq:Fs-lemma}
F_1(x) = M_2(x+1)^{-1}M_2(x), \quad F_0(x)=M_2(x)^{-1}\mathcal{Y}_xM_2(x), \quad F_{-1}(x) = M_2(x-1)^{-1}\mathcal{Z}_x M_2(x).
\end{equation}
for all $x\in \N$. Moreover
$$D= \eta \, F_1(x) + F_0(x) + \eta^{-1} F_{-1}(x) \in \mathcal{B}_R(P),$$
and $P_n(0)^{-1}P_n(-1)F_{-1}(0)=\mathcal{Y}_0-F_0(0)$.
\end{lem}
\begin{proof}
Since the monic orthogonal polynomials $(P_n)_{n=0}^m$ generate the space of polynomials of degree less than or equal to $m$, we can write $x^n$ in terms of $P_j$, $j=0,\ldots, n$. If we replace this in \eqref{eq:diferencias-x-rho}, and introduce the sequences $\widetilde F_1(x) = M_2(x+1)^{-1}M_2(x)$, $\widetilde F_0(x) = M_2(x)^{-1}\mathcal{Y}_x M_2(x)$ and $\widetilde F_{-1}(x)= M_2(x-1)^{-1}\mathcal{Z}_x M_2(x)$, we obtain
\begin{equation}
\label{eq:eq-with-x}
(x+1)^n \widetilde F_1(x) + x^n \widetilde F_0(x)+ (x-1)^n \widetilde F_{-1}(x) = \mathcal{R}_n(x), \qquad n\in \N_0, x\in \N,
\end{equation}
where $\mathcal{R}_n$ is a polynomial $x$. Then we have the following system of equations for each $n\in \N_0$, \,\, $x\in \N$:
$$\begin{pmatrix}
\widetilde F_1(x) & \widetilde F_0(x) & \widetilde F_{-1}(x)
\end{pmatrix}
\begin{pmatrix}
(x+1)^n & (x+1)^{n+1} & (x+1)^{n+2} \\
x^n & x^{n+1} & x^{n+2} \\
(x-1)^n & (x-1)^{n+1} & (x-1)^{n+2}
\end{pmatrix} = 
\begin{pmatrix}
\mathcal{R}_n(x) & \mathcal{R}_{n+1}(x) & \mathcal{R}_{n+2}(x)
\end{pmatrix}.
$$
Since the matrix in the equation above is invertible for all $x\in \N_{>1}$, there exist rational functions $F_1, F_0,  F_{-1}$ such that
$$F_1(x) = M_2(x+1)^{-1}M_2(x), \quad F_0(x)=M_2(x)^{-1}\mathcal{Y}_x M_2(x), \quad F_{-1}(x)=M_2(x-1)^{-1}\mathcal{Z}_x M_2(x),$$
for all $x\in \N_{>1}$. We note that by Definition \ref{def:duality}, $M_2(x+1)^{-1}M_2(x)$ is a rational function and, since it coincides with $F_1(x)$ for all $x\in  \N_{>1}$, the equality holds true for all $x\in \C$.

By rewriting \eqref{eq:eq-with-x} in terms of the rational functions $F_1, F_0, F_{-1}$ we get
\begin{equation}
\label{eq:eq-with-x-Fs}
(x+1)^n F_1(x) + x^n F_0(x)+ (x-1)^n F_{-1}(x) = \mathcal{R}_n(x), \qquad n\in \N_0, x\in \N_{>1}.
\end{equation}
 The left and right hand sides of \eqref{eq:eq-with-x-Fs} are rational functions which coincide for all $x\in \N_{>1}$. Therefore \eqref{eq:eq-with-x-Fs} holds true for all $x\in \C$. This in turn implies that
\begin{equation}
    \label{eq:operatorDlemma3.5}
P_n(x+1)F_1(x) + P_n(x)F_0(x) + P_n(x-1)F_{-1}(x) = \Gamma_n P_n(x), \qquad x\in \C,
\end{equation}
where $\Gamma_n = M_1(n)\rho(n)M_1(n)^{-1}$. Therefore $D= \eta \, F_1(x) + F_0(x) + \eta^{-1} F_{-1}(x) \in \mathcal{B}_R(P)$.

If we let $x=0$ in \eqref{eq:dual-recurrence}, we obtain $Q_1(\rho(n)) = \rho(n) - \mathcal{Y}_0$. On the other hand, by setting $x=0$ in \eqref{eq:operatorDlemma3.5} and after some rearrangement using that $F_1(0) = M_2(1)^{-1}M_2(0)$, we obtain
$$Q_1(\rho(n))+F_0(0)+P_n(0)^{-1}P_n(-1)F_{-1}(0)=\rho(n).$$
Therefore we conclude that $P_n(0)^{-1}P_n(-1)F_{-1}(0)= \mathcal{Y}_0-F_0(0)$ is independent of $n$.

In order to complete the proof of the lemma we only need to prove that \eqref{eq:Fs-lemma} holds true for $x=1$. If we replace $x=1$ in \eqref{eq:dual-recurrence} and \eqref{eq:operatorDlemma3.5} we obtain respectively
\begin{align*}
    Q_2(\rho(n)) + Q_1(\rho(n))\mathcal{Y}_1+ \mathcal{Z}_1 & = \rho(n) Q_1(\rho(n)), \\
    Q_2(\rho(n)) + Q_1(\rho(n))M_2(1)F_0(1)M_2(1)^{-1} + F_{-1}(1)M_2(1)^{-1} & = \rho(n) Q_1(\rho(n)).
\end{align*}
These equations imply that
$$Q_1(\rho(n))(\mathcal{Y}_1-M_2(1)F_0(1)M_2(1)^{-1}) + \mathcal{Z}_1 -F_{-1}(1)M_2(1)^{-1} = 0, $$
for all $n\in \N_0$. Now using an argument similar to that in the proof of Lemma \ref{lem:Unique-seq-Q} we get $\mathcal{Y}_1=M_2(1)F_0(1)M_2(1)^{-1}$ and $\mathcal{Z}_1 = F_{-1}(1)M_2(1)^{-1}$. This completes the proof of the lemma.
\end{proof}

\begin{rmk}
Equivalent families of dual polynomials are associated to the same operator $D\in \mathcal{B}_R(P)$. In fact, if $(Q_x,M_1,M_2,\rho)$ and $(\breve Q_x, M_1R, R^{-1}M_2, R^{-1} \rho R)$ are equivalent for a given constant matrix $R$, we obtain from \eqref{eq:recurrence_equivalentQ}  and Lemma \ref{lem:dual-to-bispectral} that they induce the same rational functions $F_1, F_0, F_{-1}$.
\end{rmk}

\begin{defn} 
\label{def:B2P}
We will denote by $\mathcal{B}^2_R(P)$ the subset of all second order difference operators
$$D= \eta \, F_1(x) + F_0(x) + \eta^{-1} F_{-1}(x) \in \mathcal{B}_R(P),$$
such that $P_n(0)^{-1}P_n(-1)F_{-1}(0)$
is independent of $n$, $F_1(x)$ is non singular for all $x\in \N_0$ and there exists an invertible matrix $M_1(n)$
such that $\rho(n) = M_1(n)^{-1}\psi^{-1}(D)M_1(n)$ satisfies Condition \ref{assumption:rho}, where $\psi$ is the generalized Fourier map given in \eqref{eq:iso}.
\end{defn}
\begin{rmk}
\label{rmk:conditionFm10}
If $D$ is a difference operator as in Definition \ref{def:B2P} such that $F_{-1}(0)=0$, then $P_n(0)^{-1}P_n(-1)F_{-1}(0)$ is trivially independent of $n$. This is the case for all the examples developed in Section \ref{sec:dualOP}.
\end{rmk}

\begin{thm}
\label{thm:correspondenceDualalgebras}
Let $(P_n)_n$ be the sequence of monic orthogonal polynomials with respect to a discrete weight $W$ as in Section \ref{sec:discrete-weights}. Every dual family  $(Q_x, M_1, M_2, \rho)$ determines a second order difference operator $D\in \mathcal{B}^2_R(P)$. Conversely, if $D\in \mathcal{B}^2_R(P)$, is given by
$$D = \eta F_1(x) + F_0(x) + \eta^{-1} F_{-1}(x),$$
then there exists a unique sequence of dual polynomials $(Q_x)_x$ which are given by
$$
P_n(x) = P_n(0) Q_x(\rho(n)) \Upsilon(x), \quad \rho(n) 
=
P_n(0)^{-1} \psi^{-1}(D) P_n(0), \quad \Upsilon(x) 
= 
F_1(0)^{-1} \cdots F_1(x-1)^{-1},
$$
$\Upsilon(0) = I$. Moreover, the dual sequence satisfies the three term recurrence relation \eqref{eq:dual-recurrence} with coefficients $$\mathcal{Y}_x = \Upsilon(x)F_0(x) \Upsilon(x)^{-1} , \qquad \mathcal{Z}_x = \Upsilon(x-1) F_{-1}(x) \Upsilon(x)^{-1}, \qquad x\in \N,$$
and $\mathcal{Y}_0 = F_0(0) + P_n(0)^{-1}P_n(-1)F_{-1}(0)$, $\mathcal{Z}_0 = 0$.
\end{thm}
\begin{proof}
We have already shown in Lemma \ref{lem:dual-to-bispectral} that every dual family corresponds to an operator $D= \eta \, F_1(x) + F_0(x) + \eta^{-1} F_{-1}(x) \in \mathcal{B}_R(P)$. From the explicit expressions in Lemma \ref{lem:dual-to-bispectral} we have that $F_1(x)$ is invertible for all $x\in \N_0$ and that $P_n(0)^{-1}P_n(-1)F_{-1}(0)$
is independent of $n$. Finally, $\rho$ satisfies Condition \ref{assumption:rho} by Definition \ref{def:duality}. Therefore $D\in \mathcal{B}^2_R(P)$.

For the converse let $D \in \mathcal{B}_R^2(P)$ so that
\begin{equation}
    \label{eq:PnD=psi}
    P_n \cdot D =P_n(x+1) \, F_1(x) + P_n(x)F_0(x) + P_n(x-1) F_{-1}(x)  = \psi^{-1}(D) \cdot P_n(x).
\end{equation}
Since $D \in \mathcal{B}^2_R(P)$, we have that $F_1(x)$ is invertible for all $x\in \mathbb{N}_0$ and therefore the matrix $\Upsilon(x)=F_1(0)^{-1} \cdots F_1(x-1)^{-1}$ is invertible as well. Let $\hat Q_x(n) = P_n(0)^{-1} P_n(x) \Upsilon(x)^{-1}$. If we replace $P_n(x) = P_n(0) \hat Q_x(n) \Upsilon(x)$ in \eqref{eq:PnD=psi}, we obtain the following recurrence relation for $\hat Q_x$:
$$\rho(n)\hat Q_x(n) = \hat Q_{x+1}(n) + \hat Q_x(n) \mathcal{Y}_x + \hat Q_{x-1}(n)\mathcal{Z}_x,$$
with $\mathcal{Y}_x$ and $\mathcal{Z}_x$ as in the statement of the theorem.
Since $D\in \mathcal{B}^2_R(P)$, we have that $\rho$ satisfies Condition \ref{assumption:rho}, and by Lemma \ref{lem:Unique-seq-Q} there exists a unique matrix polynomial $Q_x$ satisfying \eqref{eq:dual-recurrence}. Therefore
$$Q_x(\rho(n)) = \hat Q_x(n) = P_n(0)^{-1} P_n(x) \Upsilon(x)^{-1},\qquad x\in \N_0.$$
Since $D\in \mathcal{B}_R^2(P) \subset \mathcal{M}_N$, all conditions in Definition \ref{def:duality} are fulfilled and $(Q_x)_x$ is a dual family.
\end{proof}

\begin{rmk}
\label{rmk:recurrencefordual}
From Theorem \ref{thm:correspondenceDualalgebras} and Lemma \ref{lem:Unique-seq-Q}, if
$D = \eta F_1(x) + F_0(x) + \eta^{-1} F_{-1}(x)$ is an element of $\mathcal{B}^2_R(P)$, then the sequence of dual polynomials $(Q_x)_x$ associated to $D$ are defined by
the three term recurrence relation
\begin{equation*}
n Q_x(n) = Q_{x+1}(n) + Q_x(n) \mathcal{Y}_x +  Q_{x-1}(n)\mathcal{Z}_x, \qquad Q_0=I.
\end{equation*}
\end{rmk}

\begin{rmk}
\label{rmk:triple}
There is a one to one correspondence between the set of second order difference operators $\mathcal{B}_R^2(P)$ and the set of monic dual families. Theorem \ref{thm:correspondenceDualalgebras} states that every element of $\mathcal{B}_R^2(P)$ determines a unique sequence of monic dual polynomials $(Q_x, M_1, M_2, \rho)$, where $M_1(n)=P_n(0)$, $M_2(x)=\Upsilon(x)$, $\rho(n) = P_n(0)^{-1} \psi^{-1}(D) P_n(0)$, with the standard normalization $M_1(0)=M_2(0)=I$. 
On the other hand, every dual family $(Q_x, M_1, M_2, \rho)$ determines an operator $D\in \mathcal{B}_R^2(P)$ in a unique way, by Lemma \ref{lem:dual-to-bispectral}. 
\end{rmk}

\begin{rmk}
Let $D\in \mathcal{B}_R^2(P)$ with an associated dual family $(Q_x, P_n(0), \Upsilon(x), \rho(n))$, then for all $\alpha, \beta \in \C$, $\alpha \neq 0$, we have $\widetilde D = \alpha D +\beta \in \mathcal{B}_R^2(P)$. Therefore, the operator $\widetilde D$ has an associated dual family $(\widetilde Q_x, \widetilde M_1, \widetilde M_2, \widetilde \rho)$. A simple computation shows that 
$$\widetilde M_1 = P_n(0), \quad \widetilde M_2(x) = \alpha^x \Upsilon(x), \quad \widetilde \rho(n) = \alpha \rho(n) + \beta, \quad \widetilde Q_x(\widetilde \rho(n)) = \alpha^x Q_x(\rho(n)).$$

If $D_1, D_2\in \mathcal{B}_R^2(P)$ and $\alpha D_1 + \beta D_2\in \mathcal{B}_R^2(P)$, then the  dual polynomials corresponding to $D_1, D_2$ and $\alpha D_1 + \beta D_2$ are not related in such a simple way. An example of this situation is discussed in Section \ref{sec:dualOP}.
\end{rmk}

For the rest of the paper, we will consider the normalization $M_1(n)=P_n(0)$ and $M_2(x)=\Upsilon(x)$ as given in Theorem \ref{thm:correspondenceDualalgebras}.

\subsection{Dual orthogonality relations}
In view of the recurrence relation \eqref{eq:dual-recurrence}, we will be interested in the sequences $(P_n)_n$ with dual orthogonal polynomials $(Q_x)_x$ which are orthogonal polynomials as well. Next we relate the orthogonality measures of the sequences $(P_n)_n$ and $(Q_x)_x$. For this we need the following matrix valued analog  of the  Christoffel-Darboux identity for monic polynomials. This is a standard result for orthonormal matrix polynomials, see for example
\cite[Lemma 2.1]{DuranMarkov}. We include the proof to adapt this result for our monic normalization.

\begin{prop}
For all $n\in \mathbb{N}_0$, the monic matrix valued orthogonal polynomials $(P_n)_n$ satisfy
\begin{equation*}
  \sum_{k=0}^n P_k(y)^\ast \cH_k^{-1} P_k(x) = \frac{1}{y-x} 
  \left( P_{n+1}(y)^\ast \cH_n^{-1} P_n(x) - P_n(y)^\ast \cH^{-1}_n P_{n+1}(x) \right)
  ,\qquad x\neq y.
\end{equation*}
\end{prop}\begin{proof}
Using the three term recurrence relation we obtain
\begin{align*}
xP_k(y)^\ast \cH_k^{-1} P_k(x) &= P_k(y)^\ast \cH_k^{-1} P_{k+1}(x) + P_k(y)^\ast \cH_k^{-1} B_k P_k(x) + P_k(y)^\ast \cH_k^{-1}  C_k P_{k-1}(x),\\
yP_k(y)^\ast \cH_k^{-1} P_k(x) &= P_{k+1}(y)^\ast \cH_k^{-1} P_{k}(x)  + P_k(y)^\ast B_k^\ast \cH_k^{-1}  P_k(x) + P_{k-1}(y)^\ast C_k^\ast \cH_k^{-1}   P_{k}(x).
\end{align*}
Subtracting the two equations above, using the relations \eqref{eq:relations-B-C-gral} and summing over $k$ gives the proposition.
\end{proof}

Now we are ready to identify the orthogonality relations for the dual polynomials $(Q_x)_x$ under the assumption that the inverse of the square norm $\cH_n^{-1}$ decays sufficiently fast as $n\to \infty$.

\begin{thm}
\label{thm:dual-weight}
Let $(P_n)_n$ be a
sequence of monic matrix orthogonal polynomials with a dual sequence $(Q_x)_x$
and a corresponding $\rho$ as in Theorem \ref{thm:correspondenceDualalgebras}, so
that $P_n(x) = P_n(0)Q_x(\rho(n))\Upsilon(x)$. We then
consider the matrix weight 
$U(n)= P_n(0)^\ast \cH_n^{-1} P_n(0)$. 
If we assume that
 \begin{equation}
 	\label{eq:rho-moments-U(n)}
 \sum_{n=0}^\infty F(n) U(n) < \infty, \qquad \forall m\in \mathbb{N}_0,
 \end{equation}
for any matrix valued function $F$ with rational entries and no poles in $\N_0$, and that
 \begin{equation}
     \label{eq:rho-limit-U(n)}
 \lim_{n\to \infty} 
 R_1(\rho(n))P_n(0)\cH_n^{-1} P_{n+1}(0)^\ast R_2(\rho(n)) = 
 0,\qquad 
 \text{for all }R_1, R_2\in M_N(\mathbb{C})[n],
\end{equation}
 then the dual sequence $(Q_x)_{x}$ satisfies the orthogonality relation
$$
\langle Q_x,Q_y\rangle^d 
= 
\sum_{n=0}^\infty Q_x(\rho(n))^\ast U(n) Q_y(\rho(n)) 
= \mathscr{W}_x\delta_{x,y},
$$
where $\mathscr{W}_x\in M_N(\mathbb{C})$. 
\end{thm}

\begin{proof}
	The fact that $\langle Q_x,Q_x\rangle^d$ is finite and equals a matrix $\mathscr{W}_x$ follows directly from the condition \eqref{eq:rho-moments-U(n)}. For $x\neq y$, we substitute \eqref{eq:duality-definition} in the Christoffel-Darboux identity and obtain
	\begin{multline}
	\label{eq:thm-orhtog-dual}
	\sum_{k=0}^n Q_x(\rho(k))^\ast M_1(k)^\ast \cH_k^{-1} M_1(k) Q_y(\rho(k)) =  \\
	\frac{1}{x-y} 
	\left( Q_x(\rho(n+1)^\ast) M_1(n+1)^\ast \cH_n^{-1} M_1(n) Q_y(\rho(n)) \right. \\
	\left. - Q_x(\rho(n))^\ast M_1(n)^\ast \cH^{-1}_n M_1(n+1) Q_y(\rho(n+1) \right).
    \end{multline}
	Now we take the limit as $n\to \infty$ on both sides of \eqref{eq:thm-orhtog-dual}. The limit of the two terms on the right hand side of \eqref{eq:thm-orhtog-dual} is zero by the hypothesis of the theorem. This implies the orthogonality of the sequence $(Q_x)_x$.
\end{proof}

\begin{rmk}
\label{rmk:hypothesis-for-orhtog}
In the rest of this section, we shall assume that hypothesis \eqref{eq:rho-moments-U(n)} and \eqref{eq:rho-limit-U(n)} of Theorem \ref{thm:dual-weight} hold true.
\end{rmk}

Proceeding as in \cite[\S 1.2]{DamanikPS}, we have that $\|F\| = (\tr(\langle F, F \rangle)^{\frac12}$ is a seminorm on the space
$$L^2(W) = \left\{ F:\mathbb{R} \to M_N(\C) \colon \langle F, F\rangle <\infty\right\}.$$
We let $\mathscr{H}$ be the completion of $L^2(W)/\{F\colon \|F\|=0\}$. If the sequence of polynomials $(P_n)_n$ is dense in $\mathscr{H}$, we give a matrix valued analogue of \cite[Theorem 3.8]{KoekoekBook}, which describes the dual square norm.
\begin{thm}\label{thm:dualnorm1}
Let $(Q_x)_x$ be a dual family for $(P_n)_n$ and assume that the dual weight $U(n)$ satisfies the hypothesis of Theorem \ref{thm:dual-weight}. If the sequence of polynomials $(P_n)_n$ is dense in $\mathscr{H}$, then
$$
\langle Q_x,Q_x \rangle^d 
= 
\sum_{n=0}^\infty 
Q_x(\rho(n))^\ast U(n) Q_x(\rho(n)) 
= \left( \Upsilon(x) W(x) \Upsilon(x)^\ast \right)^{-1}.
$$
\end{thm}
\begin{proof}
Let $\hat P_n(x)= \cH_n^{-\frac12} P_n(x)$ be a sequence of orthonormal polynomials and let $F_y(x) = \delta_{x,y} I$. Therefore, we have
\begin{equation}
\label{eq:FyFy1}
\langle F_y, F_y \rangle = \sum_{x=0}^\infty F_y(x)W(x) F_y(x)^\ast  = W(y).
\end{equation}
On the other hand,
$$\langle F_y, \hat P_n \rangle = \sum_{x=0}^\infty F_y(x) W(x) \hat P_n(x)^\ast  = W(y)\hat P_n(y)^\ast,$$
and therefore the Parseval relation in \cite[eq. (1.42)]{DamanikPS} for $F_y$ gives
\begin{align}
\langle F_y, F_y \rangle & = \sum_{n=0}^\infty \langle F_y, \hat P_n \rangle \langle F_y, \hat P_n \rangle^\ast = \sum_{n=0}^\infty W(y)\hat P_n(y)^\ast \hat P_n(y)W(y) \nonumber \\
 & = W(y) \left( \sum_{n=0}^\infty P_n(y)^\ast \cH_n^{-1} P_n(y) \right) W(y).
 \label{eq:FyFy2}
\end{align}
Combining \eqref{eq:FyFy1} and \eqref{eq:FyFy2}, using that $W(y)$ is invertible for all $y\in \N_0$ and the duality condition \eqref{eq:duality-definition} we complete the proof of the theorem.
\end{proof}

\subsection{Dual Fourier algebras}
Proceeding as in Section \ref{subsec:algebras-fourier}, we introduce the Fourier algebras for the dual polynomials $(Q_x)_{x}$. By flipping the roles of $x$ and $n$, we consider the following algebras of difference operators:
\begin{align*}
\mathcal{M}^d_N
&=
\{ \widetilde D=\sum_{j=-\ell}^{m} \eta^j F_j(x) : \, F_j:\mathbb{N}_0 \to \MN \text{ is a sequence}\},
\\ 
\mathcal{N}^d_N
&=
\{\widetilde M=\sum_{j=-t}^{s}G_j(n)\delta^{j} : \, G_j:\mathbb{C} \to \MN \text{ is an entrywise rational function of $n$}\}.
\end{align*}
We are interested in the action of elements in $\mathcal{N}^d_N$ and $\mathcal{M}^d_N$ on the dual sequence $(Q_x)_{x}$. We define the action of the operators $\widetilde M=\sum_{j=-t}^{s}G_j(n)\delta^{j} \in \mathcal{N}_N^d$ and $\widetilde D=\sum_{j=-\ell}^{m} \eta^j F_j(x)  \in \mathcal{M}^d_N$ by
$$
(Q_x \cdot \widetilde D)(\rho(n))
=
\sum_{j=-\ell}^{m} Q_{x+j}(\rho(n))F_j(x),\quad
 (\widetilde M\cdot Q_x)(\rho(n))
=
\sum_{j=-t}^{s} G_j(n)Q_x(\rho(n+j)).
$$ 
The dual Fourier algebras for the sequence $(Q_x)_{x}$ are now given by
\begin{align*}
	\mathcal{F}^d_R(Q)
	&=
	\{\widetilde D\in \mathcal{M}^d_N : \exists \widetilde M\in \mathcal{N}^d_N, \quad \widetilde  M\cdot Q_x =Q_x \cdot \widetilde D\},
	\\ 
	\mathcal{F}^d_L(Q)
	&=
	\{\widetilde  M\in \mathcal{N}^d_N : \exists \widetilde  D\in \mathcal{M}^d_N, \quad \widetilde M\cdot Q_x =Q_x \cdot \widetilde D\}.
\end{align*}
Proceeding as in \eqref{eq:iso} we have that the map $\psi^d:\mathcal{F}^d_R(Q) \to \mathcal{F}^d_L(Q)$ defined by
\begin{equation}
\label{eq:psi-dual}
\psi^d(\widetilde D) = \widetilde M, \qquad \widetilde M\cdot Q_x = Q_x \cdot \widetilde D,    
\end{equation}
is a well defined algebra isomorphism. The proof that this map is well defined is analogous to Proposition \ref{prop:isom-Fourier-alg} and is therefore omitted.  We are now interested in the relation between the Fourier algebras for the polynomials $P_n$ and the dual Fourier algebras for the dual polynomials $Q_x$. Assume that $M\in \mathcal{F}_L(P)$ and $D=\psi(M)\in \mathcal{F}_R(P)$ are given by
$$
M=\sum_{j=-t}^{s}G_j(n)\delta^{j}, \qquad D=\sum_{k=-\ell}^{m} \eta^j F_k(x)  \in \mathcal{M}_N.
$$
The condition $M\cdot P_n = P_n \cdot D$ is given explicitly by
\begin{equation}
\label{eq:explicit-MP=PD}
\sum_{j=-t}^s G_j(n) P_{n+j}(x) = \sum_{k=-\ell}^m P_n(x+k)F_k(x), \qquad n,x \in \N_0.
\end{equation}
One would wish to transform the equation above into one for the dual polynomials by using the the relation \eqref{eq:duality-definition}. However, \eqref{eq:duality-definition} is only valid for $n,x \in \N_0$ and \eqref{eq:explicit-MP=PD} might involve terms with negative values of both $n,x$. In order to overcome these difficulties we will restrict to a smaller subalgebra of the Fourier algebras. Let $\hat{\mathcal{F}}_R(P)$ be given by
$$\hat{\mathcal{F}}_R(P) = \left\{ D = \displaystyle \sum_{k=-\ell}^{m} \eta^k F_k(x) \in \mathcal{F}_R(P) \colon F_k(x) = 0, \quad \, k=-\ell,\ldots, -1,\quad x=0,\ldots, -k-1 \right\}.$$
\begin{lem}
$\hat{\mathcal{F}}_R(P)$ is a subalgebra of $\mathcal{F}_R(P)$.
\end{lem}
\begin{proof}
It is easy to verify that $\hat{\mathcal{F}}_R(P)$ is a subspace. In order to prove that $\hat{\mathcal{F}}_R(P)$ is closed by the product operation we take $D_1, D_2 \in \hat{\mathcal{F}}_R(P)$. Without any loss of generality, and allowing some of the coefficients to be equal to zero, we can assume that
$$ D_1=\sum_{j=-\ell}^{\ell} \eta^j F_j(x), \qquad
D_2=\sum_{k=-\ell}^{\ell} \eta^k G_k(x).$$
The condition that $D_1, D_2 \in \hat{\mathcal{F}}_R(P)$ implies that $F_j(x) = G_j(x) =0 $ for all $j=-\ell,\ldots, -1$, $x=0,\ldots, -j-1$. The product $D_1D_2$ is given by
$$D_1D_2 = \sum_{j,k=-\ell}^\ell \eta^{k+j} F_j(x+k)G_k(x).$$
The proof will be complete if we show that $F_j(x+k)G_k(x)=0$ for all 
\begin{equation}
    \label{eq:conditions_kj}
k+j = -2\ell, \ldots, -1,\qquad x=0,\ldots, -k-j-1. 
\end{equation}
If  $-\ell \leq k \leq -1$, then $G_k(x)=0$ for all $x=0,\ldots, -k-1$. On the other hand, if $-k \leq  x\leq -k-j-1$, then $0\leq x+k \leq -j-1$ which implies that $j<0$ and therefore $F_j(x+k)=0$.

If $k \geq 0$ then the equation on the left of \eqref{eq:conditions_kj} implies that $j<0$. Moreover, the equation on the right of \eqref{eq:conditions_kj} gives
$0\leq k \leq x+k \leq -j-1$ and therefore $F_j(x+k)=0$.
\end{proof}

\begin{rmk}
\label{rmk:weakPearson-hatF}
We note that $x \in \hat{\mathcal{F}}_R(P)$, and thererefore $\hat{\mathcal{F}}_R(P)$ is not empty. Moreover, every operator $D$ which is associated with a weak Pearson equation of the form \eqref{sec:weak} is contained in $\hat{\mathcal{F}}_R(P)$ as well as its adjoint $D^\dagger$. Finally every operator $D$ in $\mathcal{B}^2_R(P)$ satisfying the condition of Remark \ref{rmk:conditionFm10} is contained in $\hat{\mathcal{F}}_R(P)$. This will be the case for all examples in Section \ref{sec:dualOP}.
\end{rmk}

Now we have a natural map $\sigma\colon  \hat{\mathcal{F}}_R(P)  \rightarrow \mathcal{M}_N^d$ given by
\begin{equation}
    \label{eq:definition-sigma}
D=\sum_{j=-\ell}^{m} \eta^j F_j(x) \quad \longmapsto \quad \sigma(D)=\Upsilon D \Upsilon^{-1} = \sum_{j=-\ell}^{m} \eta^j \Upsilon(x+j) F_j(x) \Upsilon^{-1}(x).
\end{equation}
Moreover, the action of $\sigma(D)$ on the dual polynomials is given by
$$(Q_x \cdot \sigma(D))(\rho(n)) = \sum_{j=-\ell}^{m} Q_{x+j}(\rho(n)) \Upsilon(x+j) F_j(x) \Upsilon^{-1}(x),\qquad \text{for all $x\in\N_0,$ and for all $n$}.$$
Since $D\in \hat{\mathcal{F}}_R(P)$, the following equality holds:
\begin{equation}
    \label{eq:relation-sigmaD-D-P-Q}
P_n(0)(Q_x\cdot \sigma(D))(\rho(n))\Upsilon(x) = (P_n\cdot D)(x), \qquad \text{for all }x\in \N_0, n\in \N_0.
\end{equation}
Now we introduce subalgebra $\hat{\mathcal{F}}_L(P) \subset \mathcal{F}_L(P)$ by considering the image of $\hat{\mathcal{F}}_R(P)$ under the algebra homomorphism $\psi^{-1}$:
\begin{equation*}
\hat{\mathcal{F}}_L(P) = \psi^{-1}( \hat{\mathcal{F}}_R(P) ).
\end{equation*}
Let $M=\psi^{-1}(D)\in \hat{\mathcal{F}}_L(P)$ be of the form $M= \sum_{j=-t}^s G_j(n)\delta^j$. Then the relation $M\cdot P_n = P_n \cdot D$ for $n\in \N_0$ is translated into the following relations for the dual polynomials:
\begin{equation}
    \label{eq:condition_G_tilde}
\sum_{j=\max(-t,-n)
}^s \widetilde G_j(n) Q_x(\rho(n+j)) =(Q_x \cdot \sigma(D))(\rho(n)),\qquad n\in \N_0,
\end{equation}
where the sequences $\widetilde G_j(n)$ are given by
$$\widetilde G_j(n) = P_n(0)^{-1}G_j(n)P_{n+j}(0),\qquad j=\max(-t,-n),\ldots, s.$$
 
\begin{lem}
\label{eq:expansion_rho_polyn}
For all $k\in \N$ there exist matrices $M_0,\ldots, M_k$ such that
$$\rho(n)^k = Q_k(\rho(n))M_k+\cdots+Q_0(\rho(n)) M_0.$$
\end{lem}
\begin{proof}
We proceed by induction in $k$. The lemma is trivial for $k=0$, taking into account that $Q_0(\rho(n))$ is the identity matrix. Now we assume that the lemma holds for a given $k \in \N$; using the inductive hypothesis and the dual three term recurrence relation \eqref{eq:dual-recurrence} we obtain
\begin{align*}
\rho(n)^{k+1} 
&= 
\rho(n)( Q_k(\rho(n))M_k+\cdots+Q_0 M_0) 
=  
\rho(n)Q_k(\rho(n)) M_k + \cdots + \rho(n) Q_0 M_0 \\
& = 
Q_{k+1}(\rho(n)) M_k 
+ Q_k(\rho(n))( \mathcal{Y}_kM_k + M_{k-1}) 
+ \cdots + Q_0(\rho(n)) ( \mathcal{Z}_1M_1+ \mathcal{Y}_0 M_0).
\end{align*}
This completes the induction and proves the lemma.
\end{proof}
Next we  will show that the coefficients $\widetilde G_j$ from \eqref{eq:condition_G_tilde} extend to unique rational functions. For this we need the following refinement of Condition \ref{assumption:rho} on the matrix function $\rho$.
\begin{assumption}
\label{assumption:rho2}
For every $\nu \in \N_0$ and $x\in \N_0$, the block Vandermonde matrix in Condition \ref{assumption:rho} is invertible for $k_0=\nu, k_1=\nu+1, \ldots , k_x=\nu+x$.
\end{assumption}

\begin{lem}
\label{lema:extensionDO}
Let $M \in \hat{\mathcal{F}}_L(P)$. Then there exists a unique operator $\tau(M)\in \mathcal{F}^d_L(Q)$ such that
$$(\tau(M) \cdot Q_x)(\rho(n)) =  (Q_x \cdot \sigma(D))(\rho(n)), \quad \text{ for all } n\in \mathbb{C}, x\in \N_0,$$
where $D=\psi(M)$, and 
\begin{equation}
\label{eq:definition-tau}
(\tau(M) \cdot Q_x)(\rho(n)) = (P_n(0)^{-1} M P_n(0) \cdot Q_x)(\rho(n)), \quad \text{ for all } n\in \N_0.
\end{equation}
Moreover $\tau$ is invertible and $\tau^{-1}:\widetilde M \mapsto P_n(0)\widetilde MP_n(0)^{-1}$.
\end{lem}
\begin{proof}
The first step in the proof consists in showing that the sequences $\widetilde G_j$ given in \eqref{eq:condition_G_tilde} extend in a unique way to rational functions of $n$. Let $k\in \N_0$. By Lemma \ref{eq:expansion_rho_polyn} and the relation \eqref{eq:condition_G_tilde}, we get
\begin{align*}
	\sum_{j=-t}^s \widetilde G_j(n)\rho(n+j)^k = \sum_{i=0}^k \left[ \sum_{j=-t}^s \widetilde G_j(n) Q_i(\rho(n+j)) \right]M_i = \sum_{i=0}^k (Q_i\cdot \sigma(D))(\rho(n)) M_i,
\end{align*}
for all $n\in \N_{\geq t}$. In other words, for all $k\in \N_0$ we have
$$
\widetilde G_{-t}(n) \rho(n-t)^k + \cdots + \widetilde G_s(n) \rho(n+s)^k 
= R_k(n),
$$

where $R_k(n)$ is a matrix valued function whose entries are rational in $n$. If we consider $k=0,\ldots,s+t$, the previous equations give the linear system
\begin{equation}
	\label{eq:linear_system_vandemonde}
	\begin{pmatrix} \widetilde G_{-t}(n) & \cdots & \widetilde G_{s}(n)
	\end{pmatrix}
\begin{pmatrix} I & \rho(n-t) & \cdots & \rho(n-t)^{s+t} \\
		\vdots & \vdots & \ddots & \vdots \\
		I & \rho(n+s) & \cdots & \rho(n+s)^{s+t}
	\end{pmatrix} = 
	\begin{pmatrix}
		R_0(n)  & \cdots &  R_{s+t}(n)
	\end{pmatrix},
\end{equation}
for all $n\in \N_{\geq t}$. Since the block Vandermonde matrix in \eqref{eq:linear_system_vandemonde} is invertible by Condition \ref{assumption:rho2} with $\nu=n-t$ and $x=s+t$, there exist uniquely determined $\MN$-valued rational functions in $n$, which coincide with $\widetilde G_{-t}, \ldots, \widetilde G_s$ for all $n\in \N_{\geq t}$. We will also denote these functions as $\widetilde G_{-t}, \ldots, \widetilde G_s$. Moreover, these functions are unique with this property. 

Now we introduce a difference operator $\tau(M) \in \mathcal{N}_N^d$ in the following way:
\begin{equation*}
\tau(M) = 
\sum_{j=-t}^s \widetilde G_j(n) \, \delta^{j},
\end{equation*}
so that \eqref{eq:condition_G_tilde} gives
$$(\tau(M) \cdot Q_x)(\rho(n)) = (Q_x \cdot \sigma(D))(\rho(n)), \qquad \text{for all $x\in\N_0,\, n\in\N_{\geq t}.$}$$ 
Since for each fixed $x\in \N_0$, both $(\tau(M) \cdot Q_x)(\rho(n))$ and $(Q_x \cdot \sigma(D))(\rho(n))$ are rational functions of $n$ and coincide in $\N_{\geq t}$, we conclude that
$$(\tau(M) \cdot Q_x)(\rho(n)) =  (Q_x\cdot \sigma(D))(\rho(n)) \quad \text{for all }n\in \mathbb{C},$$
so that $ \tau(M)\in \mathcal{F}_L^d(Q)$ and $\sigma(D)\in\mathcal{F}_R^d(Q)$. For the second statement of the lemma we have for all $n\in \N_0$:
\begin{align*}
    (\tau(M)\cdot Q_x)(\rho(n)) &= (Q_x \cdot \sigma(D))(\rho(n)) = P_n(0)^{-1} (P_n\cdot D)(x) \Upsilon(x)^{-1} \\
    & = P_n(0)^{-1} (M \cdot P_n)(x) \Upsilon(x)^{-1} = (P_n(0)^{-1} M P_n(0) \cdot Q_x)(\rho(n)),
\end{align*}
using \eqref{eq:relation-sigmaD-D-P-Q} in the second equality. The expression of the inverse of $\tau$ is straightforward.
\end{proof}

Finally, we shall denote $\hat{\mathcal{F}}_L^d(Q) = \tau(\hat{\mathcal{F}}_L(P))$, so that $$\hat{\mathcal{F}}_R^d(Q) = \sigma(\hat{\mathcal{F}}_R(P)), \qquad 
\hat{\mathcal{F}}_L(P) = \psi^{-1}(\hat{\mathcal{F}}_R(P)),  \qquad
\hat{\mathcal{F}}^d_L(Q)) = \tau(\hat{\mathcal{F}}_L(P)).$$
It is readily seen that $\tau: \hat{\mathcal{F}}_L(P)
\to \hat{\mathcal{F}}_L^d(Q)$ and $\sigma: \hat{\mathcal{F}}_R(P) \to \hat{\mathcal{F}}_R^d(Q)$ are algebra isomorphisms.

\begin{thm}
\label{thm:FourierAlgebrasDiagram}
Let $P_n$ and $Q_n$ be dual families and let $\sigma$ and $\tau$ be as in \eqref{eq:definition-sigma} and Lemma \ref{lema:extensionDO} respectively, then the following diagram is commutative.
$$
\begin{tikzcd}
	\hat{\mathcal{F}}_L(P) \arrow[r, "\psi"] \arrow[d, "\tau"]
	& \hat{\mathcal{F}}_R(P) \arrow[d, "\sigma"] \\
	\hat{\mathcal{F}}^d_L(Q) 
	& \arrow[l, "\psi^d"] \hat{\mathcal{F}}^d_R(Q)
\end{tikzcd}
$$
\end{thm}
\begin{proof}
The theorem follows directly from Lemma \ref{lema:extensionDO}. For all $M\in \hat{\mathcal{F}}_L(P)$ we have $\tau(M)\cdot Q_x = Q_x \cdot \sigma(\psi(M))$. Then \eqref{eq:psi-dual} implies that
$$\tau(M) = \psi^d(\sigma(\psi(M))).$$
This completes the proof.
\end{proof}

The last result of this section relates the adjoint operators on the algebra $\hat{\mathcal{F}}_R(P)$ with respect to the matrix inner product $\langle \, , \,  \rangle$ with the adjoint operators on the dual algebra $\hat{\mathcal{F}}^d_R(Q)$ with respect to the dual inner product $\langle\, , \, \rangle^d$.

\begin{cor}
Let $D \in \hat{\mathcal{F}}_R(P)$ and assume that there exists $D^\dagger \in \hat{\mathcal{F}}_R(P)$. Let $M=\psi^{-1}(D)$. Then
$$\tau(M)^{\dagger}=\tau(M^{\dagger}),\qquad \sigma(D)^\dagger=\sigma(D^\dagger).$$
\end{cor}
\begin{proof}
Since $D, D^\dagger \in \hat{F}_R(P)$ we have that $M, M^\dagger\in \hat{\mathcal{F}}_L(P)$. Suppose $M = \sum_{j=-t}^s G_j(n) \delta^j$, and let $\tau(M)\in \hat{\mathcal{F}}^d_L(Q)$ as in Lemma \ref{lema:extensionDO}. Then we have
\begin{align*}
    \langle \tau(M) \cdot Q_x, Q_y \rangle^d &= \sum_{n=0}^\infty (\tau(M)\cdot Q_x)(\rho(n))^\ast U(n) Q_y(\rho(n)) \\
     & = \sum_{n=0}^\infty \sum_{j=\max(-t,-n)}^s Q_x(\rho(n+j))^\ast P_{n+j}(0)^\ast G_j(n)^\ast (P_n(0)^\ast)^{-1} U(n) Q_y(\rho(n)) \\
     & =\sum_{r=0}^\infty \sum_{j=-t}^{\min(s,r)} Q_x(\rho(r))^\ast U(r) P_r(0)^{-1} \cH_r G_j(r-j)^\ast \cH_{r-j}^{-1} P_{r-j}(0) Q_y(\rho(r-j)) \\
     & = \sum_{r=0}^\infty Q_x(\rho(r))^\ast U(r) (\tau(M^\dagger)\cdot Q_y)(\rho(r)) = \langle Q_x, \tau(M^\dagger) \cdot Q_y \rangle^d,
\end{align*}
where we have used
the expression for
the adjoint of $M$ in
\eqref{eq:definition-Mdagger}. This proves the first formula of the corollary. For the second formula we have
$$\langle Q_x \cdot \sigma(D), Q_y\rangle^d=\langle \tau(M) \cdot Q_x, Q_y\rangle^d=\langle Q_x,\tau(M^{\dagger}) \cdot Q_y\rangle^d=\langle Q_x, Q_y \cdot \sigma(D^{\dagger})\rangle^d,$$
for all $x,y\in \N_0$. The proof of the corollary is now complete.
\end{proof}
The rest of this paper is devoted to the construction of a suitable family of matrix orthogonal polynomials and their dual families.

\section{Matrix valued Charlier polynomials}
\label{sec:charlier}
In this section we start with a weight matrix $W$ satisfying the simplest system of weak Pearson equations of the form \eqref{eq:weak-pearson}, namely a single difference equation. We take $F_j(x)=\fA, \widetilde{F}_j(x)^\ast=x\fB$, where $\fA$ and $\fB$ are 
invertible constant matrices, and we assume that  $W$ satisfies
\begin{equation}
\label{eq:DDdaggerweak}
\fA W(x-1)
=
W(x)\fB x.
\end{equation}
Iterating this equation we get
\begin{equation}\label{eq:iterated}
W(x)=\fA^x\frac{W(0)}{x!}\fB^{-x},\qquad x\in \N_0.
\end{equation}
Since $W$ is a weight matrix, it is positive definite by definition, although it is not direct from \eqref{eq:iterated}. However, we still do not make extra assumptions on $\mathscr{A}$ and $\mathscr{B}$.
The operators $D$ and $D^{\dagger}$, introduced in
equation \eqref{eq:def-D-Ddagger-gral} and
Proposition \ref{prop:D-dagger}, are now given by
\begin{equation} \label{eq:ladderD_sec4}
P\cdot D(x)
=
P(x+1)\fA, 
\qquad 
P\cdot D^{\dagger}(x)
=
P(x-1)\fB^\ast x.
\end{equation}
The corresponding operators in the left Fourier algebra
$\mathcal{F}_L(P)$ in Proposition \ref{prop:ladder_sec2} can now be written in terms of the square norms of the monic MVOP.
\begin{prop}\label{prop:ladder}
Let $\left(P_n\right)_{n}$ be the sequence of monic orthogonal polynomials with respect to a positive definite weight $W$ 
of the form \eqref{eq:iterated}.
Then the polynomials satisfy the
following equations
\begin{alignat}{2}
\label{eq:ec-M}
& P_n\cdot D 
= 
M \cdot P_n,
\qquad
&&
M 
= 
\fA + \cH_n \fB \cH_{n-1}^{-1} \delta^{-1},
\\
\label{eq:ec-Mdagger}
\qquad
& P_n\cdot D^\dagger 
= 
M^\dagger \cdot P_n,
\qquad
&&
M^\dagger 
= 
\fB^\ast \delta + \cH_n \fA^\ast \cH_n^{-1},
\end{alignat}
where $D$ and $D^\dagger$ are 
as in \eqref{eq:ladderD_sec4}.
\end{prop}

\begin{proof}
We are in the situation of Proposition 
\ref{prop:ladder_sec2}, with $s=0$ and $t=1$.
This means that we have ladder equations and that
our left acting operators are of the
form
$$
M
=
G_0(n) + G_{-1}(n)\delta^{-1},
\qquad
M^\dagger
=
\cH_n G_{-1}(n+1)^\ast \cH_{n+1}^{-1}\delta
+
\cH_n G_0(n)^\ast \cH_n^{-1},
$$
where $G_0(n)=\langle P_n\cdot D,P_n\rangle \cH_n^{-1}$ and $G_{-1}(n)=\langle P_n\cdot D,P_{n-1}\rangle \cH_{n-1}^{-1}$. These coefficients can be directly determined 
by looking at the leading coefficients in our ladder relations. Firstly, the $x^n$ coefficient of the relation $P_n\cdot D = M\cdot P_n$, gives us 
$$
G_0(n) = \fA.
$$
On the other hand, the $x^{n+1}$ coefficient of $P_n\cdot D^\dagger = M^\dagger\cdot P_n$ gives
$\fB^\ast = \cH_n G_{-1}(n+1)^\ast \cH_{n+1}^{-1}$ and so
$$
G_{-1}(n) = \cH_n \fB \cH_{n-1}^{-1},
$$
since the square norms are self adjoint matrices.
\end{proof}

\begin{rmk}
The equations \eqref{eq:ec-M}, \eqref{eq:ec-Mdagger} can be written as 
$$
\psi(M)=D, 
\qquad 
\psi(M^{\dagger})=D^{\dagger}.
$$
As a consequence the isomorphism  of the Fourier
algebras $\psi$ preserves the $\dagger$ operation
in the subalgebra generated by $M,$ $M^{\dagger}$ 
and $I$.
\end{rmk}

\begin{cor}\label{cor:Bn}
Let $W$ be a matrix weight of the form \eqref{eq:iterated} and let $(P_n)_n$ be the sequence monic orthogonal orthogonal polynomials. Then the three term recurrence coefficients $B_n$ are given by
\begin{equation}
\label{eq:ecuacion-B}
B_n
=
(\fB^\ast)^{-1}\cH_n\fA^\ast\cH_n^{-1}
+
\cH_n (\fA^\ast)^{-1}\cH_{n-1}^{-1}\fB^\ast, \qquad n\in \N,
\end{equation}
and
\begin{equation}
\label{eq:ecuacion-B0}
B_0 
= 
(\fB^\ast)^{-1}\cH_0\fA^\ast \cH_0^{-1}.
\end{equation}
\end{cor}
\begin{proof}
If we replace $x=0$ in the ladder relation $P_n\cdot D^\dagger = M^\dagger\cdot P_n$ and use \eqref{eq:ec-Mdagger}, we get
\begin{equation}\label{eq:x0Mdagger}
0 = P_{n+1}(0) + (\fB^\ast)^{-1} \cH_n\fA^\ast\cH_n^{-1} P_{n}(0).
\end{equation}
Next we set $x=0$ in the three term recurrence relation \eqref{eq:3TR} for $n\geq 1$: 
$$
P_{n+1}(0)
= -B_n P_n(0) -C_n P_{n-1}(0).
$$
Eliminating $P_{n+1}(0)$ we get
$$
B_n
=
(\fB^\ast)^{-1}\cH_n\fA^\ast\cH_n^{-1}
-
\cH_n\cH_{n-1}^{-1}P_{n-1}(0)P_n(0)^{-1},
$$
and lastly using \eqref{eq:x0Mdagger} again
$$
B_n
=
(\fB^\ast)^{-1}\cH_n\fA^\ast\cH_n^{-1}
+
\cH_n(\fA^\ast)^{-1}\cH_{n-1}^{-1}\fB^\ast.
$$
For $n=0$ the three term recurrence gives us one
term less. The corresponding expression is obtained in an analogous way.
\end{proof}
\begin{rmk}
The matrix $P_n(0)$ is invertible for all $n\in \N_0$. This is trivially true for $P_0(0) = I$. For $n\geq 1$ it follows from \eqref{eq:x0Mdagger} and the fact that  $(\fB^\ast)^{-1} \cH_n\fA^\ast\cH_n^{-1}$ is invertible for all $n\geq 1$.
\end{rmk}

\begin{cor}\label{cor:rec1}
The monic orthogonal polynomials $(P_n)_n$ with respect to a positive definite weight $W$ of the form \eqref{eq:iterated}, satisfy the following 
second order difference equation in $x$,
\begin{multline}
\label{eq:Prec}
P_n(x+1)\fA
+
 \left( 
    \cH_n\fB(\fA^\ast)^{-1}\cH_{n-1}^{-1}\fB^\ast
     -x\cH_n\fB\cH_n^{-1}
      -\fA
   \right)P_n(x)
\\
\qquad+
x \cH_n\fB\cH_n^{-1}(\fB^\ast)^{-1}
P_n(x-1)\fB^\ast  
= 0, \qquad n\in \N.
\end{multline}
\end{cor}
\begin{rmk}
Observe that \eqref{eq:Prec} is a mixed relation. There appear matrices multiplying from the left and also from the right.
\end{rmk}
\begin{proof}
We start from the three term recurrence \eqref{eq:3TR}
$$
x P_n(x) = P_{n+1}(x) + B_n P_n(x) + C_n P_{n-1}(x),
$$
and eliminate $P_{n+1}$ and $P_{n-1}$ with the ladder relations of Proposition \ref{prop:ladder}
\begin{align*}
P_{n-1}(x) 
&= 
\cH_{n-1}\fB^{-1}\cH_n^{-1}
\left( P_n(x+1)\fA - \fA P_n(x) \right),
\\
P_{n+1}(x) 
&=  (\fB^\ast)^{-1}
\left( P_n(x-1)\fB^\ast x - \cH_n\fA^\ast \cH_n^{-1}P_n(x) \right).
\end{align*}
After replacing these expressions in the three term recurrence and multiplying by $\cH_n\fB\cH_n^{-1}$
from the left, it becomes
\begin{multline*}
P_n(x+1)\fA
+
 \left( \cH_n\fB\cH_n^{-1}
		\left(B_n - x -(\fB^\ast)^{-1} \cH_n\fA^\ast\cH_n^{-1} \right) 
		-\fA
   \right)P_n(x)
\\
\qquad+
x \cH_n\fB\cH_n^{-1}(\fB^\ast)^{-1}
P_n(x-1)\fB^\ast  
= 0.
\end{multline*}
We finally use the result of Corollary \ref{cor:Bn} and the equation simplifies to \eqref{eq:Prec}.
\end{proof}

\subsection{Manifestly symmetric weight}
Now we consider the weight $W$ from \eqref{eq:iterated} with
the added restriction that $\fB=(a\fA^\ast)^{-1}$,
for a positive scalar $a$. In this form we derive a  second order recurrence in
$n$ that only involves the square norms.
Our weight then becomes
\begin{equation}\label{eq:manifest}
W(x)
=
\frac{a^x}{x!}\fA^x W(0) (\fA^\ast)^x,
\qquad
a>0.
\end{equation}
We will also from now on assume that $W(0)$ is positive definite so that every matrix $W(x)$ of the form \eqref{eq:manifest} will be positive definite.
\begin{thm}
\label{thm:recurrencia-norma}
Let $W$ be as in \eqref{eq:manifest}. Then the
square norms of the monic orthogonal polynomials
satisfy
\begin{multline}\label{eq:recurrencia-Hn}
\cH_n 
=
a^2 \fA^2 \cH_{n-1} (\fA^\ast)^2
- a^2 \fA \cH_{n-1}\fA^\ast\cH_{n-1}^{-1}\fA\cH_{n-1}\fA^\ast
\\
\qquad +a\fA \cH_{n-1}\fA^\ast
+\fA\cH_{n-1}(\fA^\ast)^{-1}\cH_{n-2}^{-1}\fA^{-1}\cH_{n-1}\fA^\ast, \qquad n\in \N_{\geq 2},
\end{multline}
and
\begin{equation}\label{eq:recurrencia-H1}
\cH_1 
=
a^2 \fA^2 \cH_{0}\fA^\ast (\fA^\ast)^2
- a^2 \fA \cH_{0}\fA^\ast\cH_{0}^{-1}\fA\cH_{0}\fA^\ast
+a\fA \cH_{0}\fA^\ast.
\end{equation}
\end{thm}
\begin{proof}
Keeping in mind that the operators in $x$ act from the
right, we note that $[D,x]=-D$. Then by means of the isomorphism $\psi$ 
in \eqref{eq:iso} we get 
$[M,\mathcal{L}]=-M$, recalling that $\psi^{-1}(x)=\mathcal{L}$ from \eqref{eq:3TROP}.
We introduce the notation $\beta(n)=\frac{1}{a}\cH_n(\fA^\ast)^{-1}\cH_{n-1}^{-1},$ and find by direct calculation that
\begin{multline}
\label{eq:corchete-M-L}
[M,\mathcal{L}]=( \beta(n) - \beta(n+1)+[\fA,B_n]) \\
+ (\beta(n)B_{n-1} - B_n\beta(n) + [\fA, C_n] )\delta^{-1} \\
+ (\beta(n)C_{n-1}-C_n\beta(n-1))\delta^{-2}.
\end{multline}
Now using \eqref{eq:ecuacion-B} for $n\geq2$ we get that the $\delta^{-1}$ coefficient in the equation \eqref{eq:corchete-M-L} is 
\begin{multline}
\label{eq:ecuacion-normas-1}
\cH_n(\fA^\ast)^{-1}\cH_{n-1}^{-1}\fA\cH_{n-1}\fA^\ast \cH_{n-1}^{-1}
+\frac{1}{a^2}\cH_n (\fA^\ast)^{-2}\cH_{n-2}^{-1}\fA^{-1}
\\
-\frac{1}{a^2}\cH_n (\fA^\ast)^{-1}\cH_{n-1}^{-1}\fA^{-1}\cH_n (\fA^\ast)^{-1}\cH_{n-1}^{-1}
 -  \cH_n\cH_{n-1}^{-1}\fA, \qquad n\in \N_{\geq 2},
\end{multline}
and for $n=1$ using \eqref{eq:ecuacion-B} and \eqref{eq:ecuacion-B0} this coefficient is 
\begin{equation}
\label{eq:ecuacion-normas-1-n1}
\cH_1(\fA^\ast)^{-1}\cH_{0}^{-1}\fA\cH_{0}\fA^\ast \cH_{0}^{-1}
-\frac{1}{a^2}\cH_1 (\fA^\ast)^{-1}\cH_{0}^{-1}\fA^{-1}\cH_1 (\fA^\ast)^{-1}\cH_{0}^{-1}
 -  \cH_1\cH_{0}^{-1}\fA.
\end{equation}
We now use $[M,\mathcal{L}]=-M,$ to equate the expression in \eqref{eq:ecuacion-normas-1} to $-\beta(n)$,
\begin{multline}
\label{eq:ecuacion-normas-2}
\cH_n(\fA^\ast)^{-1}\cH_{n-1}^{-1}\fA\cH_{n-1}\fA^\ast \cH_{n-1}^{-1}
+\frac{1}{a^2}\cH_n (\fA^\ast)^{-2}\cH_{n-2}^{-1}\fA^{-1}
-  \cH_n\cH_{n-1}^{-1}\fA
\\
-\frac{1}{a^2}\cH_n (\fA^\ast)^{-1}\cH_{n-1}^{-1}\fA^{-1}\cH_n (\fA^\ast)^{-1}\cH_{n-1}^{-1}
=-\frac{1}{a}\cH_n(\fA^\ast)^{-1}\cH_{n-1}^{-1}, \qquad n\in \N_{\geq 2},
\end{multline}
and \eqref{eq:ecuacion-normas-1-n1} to $-\beta(1),$ 
\begin{multline}
\label{eq:ecuacion-normas-2-n1}
\cH_1(\fA^\ast)^{-1}\cH_{0}^{-1}\fA\cH_{0}\fA^\ast \cH_{0}^{-1}
-\frac{1}{a^2}\cH_1 (\fA^\ast)^{-1}\cH_{0}^{-1}\fA^{-1}\cH_1 (\fA^\ast)^{-1}\cH_{0}^{-1}
 -  \cH_1\cH_{0}^{-1}\fA
 \\
=-\frac{1}{a}\cH_1(\fA^\ast)^{-1}\cH_{0}^{-1}.
\end{multline}
We multiply \eqref{eq:ecuacion-normas-2} and \eqref{eq:ecuacion-normas-2-n1} from the right 
by $\cH_{n-1}$ and $\cH_0$ respectively, and from 
the left by $a\fA^\ast\cH_n^{-1}$ and $a\fA^\ast\cH_1^{-1}$. 
Then we obtain
\begin{multline*}
a\cH_{n-1}^{-1}\fA\cH_{n-1}\fA^\ast 
+\frac{1}{a}(\fA^\ast)^{-1}\cH_{n-2}^{-1}\fA^{-1}\cH_{n-1}
\\
-  a\fA^\ast\cH_{n-1}^{-1}\fA\cH_{n-1}
-\frac{1}{a}\cH_{n-1}^{-1}\fA^{-1}\cH_n (\fA^\ast)^{-1}
=-I
 \qquad n\in \N_{\geq 2},
\end{multline*}
and for $n=1$
\begin{equation*}
\cH_{0}^{-1}\fA\cH_{0}\fA^\ast 
-  a\fA^\ast\cH_{0}^{-1}\fA\cH_{0}
-\frac{1}{a}\cH_{0}^{-1}\fA^{-1}\cH_1 (\fA^\ast)^{-1}
=-I.
\end{equation*}
Now we can isolate the square norm with the highest degree to get
\begin{multline*}
\cH_n =
a^2 \fA^2 \cH_{n-1} (\fA^\ast)^2
- a^2 \fA \cH_{n-1}\fA^\ast\cH_{n-1}^{-1}\fA\cH_{n-1}\fA^\ast
\\
\qquad +a\fA \cH_{n-1}\fA^\ast
+\fA\cH_{n-1}(\fA^\ast)^{-1}\cH_{n-2}^{-1}\fA^{-1}\cH_{n-1}\fA^\ast, \qquad n\in \N_{\geq 2},
\end{multline*}
and for $n=1$
\begin{equation*}
\cH_1 
=
a^2 \fA^2 \cH_{0}\fA^\ast (\fA^\ast)^2
- a^2 \fA \cH_{0}\fA^\ast\cH_{0}^{-1}\fA\cH_{0}\fA^\ast
+a\fA \cH_{0}\fA^\ast.
\end{equation*}
This completes the proof of the theorem.
\end{proof}

\begin{rmk}
As a brief sanity check we can consider what the
previous result reduces to in the scalar case. We then have $\fA =1$ and the equations \eqref{eq:recurrencia-Hn} and \eqref{eq:recurrencia-H1} become
$$\cH_n=\cH_{n-1}\cH_{n-2}^{-1}\cH_{n-1}+a\cH_{n-1}, \qquad \cH_1=a\cH_0.$$ 
It is easily checked that the
square norms for the monic scalar Charlier polynomials $\cH_n=n!a^n e^a$ satisfy this.
\end{rmk}

\begin{rmk}
The manifestly symmetric weight \eqref{eq:manifest}
also allows us to simplify the difference equation 
in $x$ for
the $P_n$ in Corollary \ref{cor:rec1}.
\begin{multline}\label{eq:rec1}
P_n(x+1)\fA +
 \left( 
    \frac{1}{a^2}\cH_n(\fA^\ast)^{-2}\cH_{n-1}^{-1}\fA^{-1}
     -\frac{x}{a}\cH_n(\fA^\ast)^{-1}\cH_n^{-1}
      -\fA
   \right)P_n(x) \\
\qquad+ \frac{x}{a} \cH_n(\fA^\ast)^{-1}\cH_n^{-1}\fA
P_n(x-1)\fA^{-1}
= 0.
\end{multline}
If we then consider the corresponding recursion
in $x$ for $P_n(x)\fA^{x}$, we would get a recursion
with coefficients multiplying the $P_n$ only from 
the left.
But we will save this for Section \ref{sec:7}
where we will be able to make more simplifications.
\end{rmk}

\section{A one parameter family of matrix Charlier polynomials}
\label{sec:L}
In this section we specialize the weight matrix in \eqref{eq:manifest}. This will allow us to give a explicit expression for the 0-th norm which by Theorem \ref{thm:recurrencia-norma} determines all square norms. The choice is motivated by previous works on Gegenbauer, Laguerre and Hermite type matrix valued orthogonal polynomials \cite{AKdlR, IKR2, KoelinkRlaguerre}. 

Let $\lambda \in \N_0$ and let $A, T^{(\lambda)}$ be constant matrices defined by
$$
A_{j,k}
=
\begin{cases}
             \frac{\mu_j }{\mu_{j-1}}, &  j=k+1 \\
             0, &   j\neq k+1
\end{cases},
\qquad
T^{(\la)} = \textrm{diag}(\delta_1^{(\la)} \dots \delta_N^{(\la)}),
$$
with $\delta_j^{(\la)} >0$ and $\mu_j >0$. We let $\mathscr{A} = (A+I)$ and $W(0) = (I+A)^\lambda  T^{(\lambda)} (I+A^\ast)^\lambda$ in \eqref{eq:manifest} so that we have a \textit{family} of weight matrices
\begin{equation}\label{eq:WL}
W^{(\la)}(x)
=
\frac{a^x}{x!}(I+A)^{x+\la}T^{(\la)}(I+A^\ast)^{x+\la},
\qquad
a>0.
\end{equation} 
The family is parametrized by $\la \in \mathbb{N}_0$, and fits in the framework of Section \ref{subsec:formula-rodrigues}. However, the  
parameter $\lambda$ will only start playing a relevant role in the next section. The corresponding inner-product will be denoted by $\langle \cdot , \cdot  \rangle ^{(\lambda)}$ as in \eqref{eq:producto-interno-discreto-matricial}. 

We are now able to explicitly calculate the $0$-th square norm\footnote{Most quantities\ 
will now depend on this new parameter $\la$.} $\cH_0^{(\la)}$
and find a LDU factorization for all the square 
norms, though the diagonal matrix factor can not yet
be determined explicitly. 
Before we do this we should introduce a unipotent lower triangular matrix polynomial that will be very useful in what follows.

\subsection{Properties of the matrix $L$}
We first recall some standard facts about the scalar Charlier polynomials, see for instance \cite{NIST:DLMF, KoekS}. The most basic ones are their hypergeometric representation \eqref{eq:hyperg-Charlier}, the orthogonality relations \eqref{eq:ortogonalidad-charlier-escalar} and the the generating function
\begin{equation}\label{eq:gencharlier}
e^t \left( 1 -\frac{t}{a}\right)^x
=
\sum_{n=0}^\infty \frac{c_n^{(a)}(x)}{n!}t^n,
\qquad |t|< |a| .
\end{equation}
We will also need the forward and backward shift equations respectively
\begin{equation}
\label{eq:delta-charlier}
c^{(a)}_n(x+1)=c^{(a)}_n(x) - 
\frac{n}{a} c_{n-1}^{(a)}(x),
\qquad
c_{n+1}^{(a)}(x)=c_n^{(a)}(x) - \frac{x}{a}c_n^{(a)}(x-1),
\end{equation}
and the second order difference equation in $x$
\begin{equation}\label{eq:scalarrec}
a c_n^{(a)}(x+1)
-
(x+a) c_n^{(a)}(x)
+
x c_n^{(a)}(x-1)
=
-n c_n^{(a)}(x).
\end{equation}
We will also need a slightly less standard result.
\begin{prop}\label{prop:nonstd}
For $a\neq 0$ the scalar Charlier polynomials
satisfy the following convolution
$$
\sum_{m=0}^n 
(-1)^m 
\frac{c_{n-m}^{(a)}(x)}{(n-m)!} 
\frac{c_m^{(-a)}(-x)}{m!}
=
\delta_{n,0}.
$$
\end{prop}
\begin{proof}
It is clear that the sum equals $1$ for $n=0$.
For $n\geq 1$ we take a copy of the
generating function in \eqref{eq:gencharlier} 
but with
the signs changed of $t$, $x$ and $a$
$$
e^{-t}\left(1-\frac{t}{a} \right)^{-x}
=
\sum_{m=0}^\infty (-1)^m \frac{c_{m}^{(-a)}(-x)}{m!}t^m,
$$
and multiply it by the ordinary generating function
to get $1$. So then we have
$$
1 = \sum_{n=0}^\infty \sum_{m=0}^\infty
(-1)^m
\frac{c_n^{(a)}(x)}{n!}
 \frac{c_{m}^{(-a)}(-x)}{m!}t^{n+m}.
$$
We then change summation variables $s=n+m$
$$
1 = \sum_{s=0}^\infty \sum_{m=0}^s
(-1)^m
\frac{c_{s-m}^{(a)}(x)}{(s-m)!}
 \frac{c_{m}^{(-a)}(-x)}{m!}t^s,
$$
and see that the coefficients of all
positive powers of $t$ must vanish,
giving the desired result.
\end{proof}

Let us consider the following unipotent lower triangular
matrix polynomial
\begin{equation}
\label{eq:def-L}
L(x)_{j,k}=  \frac{\mu_j }{\mu_k} (-a)^{j-k}
             \frac{c^{(a)}_{j-k}(x)}{(j-k)!}, \quad j\geq k, \qquad    L(x)_{j,k} = 0, \quad   j<k,
\end{equation}
where $c^{(a)}_{j-k}$ is the scalar Charlier polynomial of degree $j-k$. The matrix $L$ considered in this case is analogous to the matrices that appear related to Gegenbauer, Laguerre  and Hermite type orthogonal polynomials \cite{KdlRR,IKR2, KoelinkRlaguerre}. See also \cite{CaglieroK} for a lower triangular matrix involving Jacobi polynomials.

\begin{lem}
\label{lem:delta-L}
The following properties of $L$ hold true.
\begin{enumerate}
    \item
$\det(L(x))=1$.    
    \item 
$L(x+1)= L(x)(I+A)$.
    \item
$L(x)A=AL(x)$.
    \item
$L(x)=L_0(I+A)^x = (I+A)^x L_0, \quad x\in\mathbb{Z}$.
\end{enumerate}
Here we use the notation $L_0 = L(0)$.
\end{lem}
\begin{proof}
\textit{(1)} and \textit{(3)} follow immediately from the definition.
\textit{(2)} follows from the forward shift equation \eqref{eq:delta-charlier}
and \textit{(4)} follows from \textit{(2)} by iteration.
\end{proof}

\begin{rmk}
\label{rmk:weight-in-L}
By (4) of Lemma \ref{lem:delta-L} we have that $L(x+\lambda) = (I+A)^\lambda L(x)$. Therefore we can rewrite the weight matrix \eqref{eq:WL} as:
\begin{align*}
W^{(\lambda)}(x) & =\frac{a^x}{x!} L_0^{-1} L(x+\lambda)T^{(\lambda)}L(x+\lambda)^\ast(L_0^\ast)^{-1} \\
&= \frac{a^x}{x!} L_0^{-1} (I+A)^\lambda L(x) T^{(\lambda)}L(x)^\ast(I+A^\ast)^\lambda (L_0^\ast)^{-1}.
\end{align*}
\end{rmk}

The proof of the following result
is analogous to the work \cite{CaglieroK} involving
Jacobi polynomials.
\begin{lem}\label{lem:immediate}
The inverse of $L$ is the lower triangular matrix
$$L(x)_{j,k}^{-1}
= \frac{\mu_j}{\mu_k}a^{j-k} \frac{c_{j-k}^{(-a)}(-x)}{(j-k)!}, \quad j \geq k,\qquad L(x)_{j,k}^{-1} = 0,\quad  j < k.
$$
\end{lem}
\begin{proof}
Let us denote the matrix in the statement of the lemma by $K(x)$. Now we observe that all the entries above the main diagonal of $L(x)K(x)$ are immediately 0 by definition. So we are left with $(L(x)K(x))_{j,k}$ for $j\geq k$ which is given by
\begin{multline*}
\sum_{\ell=k}^j 
\frac{\mu_j}{\mu_\ell}
(-a)^{j-\ell}
\frac{c_{j-\ell}^{(a)}(x)}{(j-\ell)!}
\frac{\mu_\ell}{\mu_k}
a^{\ell-k}
\frac{c_{\ell-k}^{(-a)}(-x)}{(\ell-k)!}
=
(-1)^{j-k}
a^{j-k}
\frac{\mu_j}{\mu_k}
\sum_{\ell=k}^j 
(-1)^{k-\ell}
\frac{c_{j-\ell}^{(a)}(x)}{(j-\ell)!}
\frac{c_{\ell-k}^{(-a)}(-x)}{(\ell-k)!}.
\end{multline*}
If we shift the summation index in the last
equation, we get the sum as in Proposition
\ref{prop:nonstd}, so it will give 1 if $j=k$
and $0$ if not. This completes the proof.
\end{proof}

The inverse of $L$ has a similar structure and involves Charlier polynomials with negative parameters. This feature is compatible with the previous results \cite{CaglieroK, KdlRR, IKR2, KoelinkRlaguerre}. For the matrix $L$ considered in this paper, this is  essentially a consequence of the generating functions of the scalar Charlier polynomials, as in the Laguerre and Hermite cases \cite{IKR2, KoelinkRlaguerre}.

We conclude this subsection
by introducing another useful matrix $J=\textrm{diag}(1,\dots, N)$.
\begin{lem}\label{lem:immediate2}
$J$ satisfies the following equations,
for $k \in \mathbb{Z}$.
\begin{enumerate}
    \item 
    $[J,A]=A.$
    \item
    $[J,(I+A)^k] = k A(I+A)^{k-1}=k (I+A)^{k}-k(I+A)^{k-1}.$
    \item
    $[J,L(x)] = xA(I+A)^{-1}L(x) -aAL(x).$
    \item
    $(I+A)^{k}J(I+A)^{-k}
    =
    J - k A(I+A)^{-1}.$
    \item
    $L(x)^{-1}JL(x)=J-aA+x A(I+A)^{-1}.$
\end{enumerate}
\end{lem}
\begin{proof}
$(1)$ follows immediately from the definition.
$(2)$ follows from iterating (1) and the fact that
$A(I+A)^{-1}=I-(I+A)^{-1}$. For $(3)$ we can
use the backward shift equation in \eqref{eq:delta-charlier} to show that
$$
[J,L(x)]
=
-a A L(x) - a A L(x-1),
$$
and then the result follows from Lemma \ref{lem:immediate}. Lastly (4) follows from (2) and (5) from (3).
\end{proof}

\subsection{LDU decomposition of the norm}
In this subsection we find a LDU decomposition for the square norms $\cH_n^{(\lambda)}$. For this we first need to find an expression for $0$-th square norm. The expression of the matrix $L$ in terms of the scalar Charlier polynomials plays a fundamental role in this proof. Then we use the nonlinear recurrence relation in Theorem \ref{thm:recurrencia-norma}, to extend this decomposition to all $\cH_n^{(\lambda)}$.
\begin{prop}\label{prop:norma0}
The $0$-th moment of the weight in \eqref{eq:WL}, has the following LDU factorization
\begin{equation}
\label{eq:H0-D0}
\cH_0^{(\la)}
=
L_0^{-1}(I+A)^\la \D_0^{(\la)} (I+A^\ast)^\la (L_0^\ast)^{-1},
\end{equation}
where $\D_0^{(\la)}$ is a diagonal matrix with entries
\begin{equation}
\label{eq:norma0}
(\D_0)_{jj}^{(\la)}  
= 
\mu_j^2 a^j e^a 
\sum_{\ell=1}^j \frac{\delta_\ell^{(\la)}}{\mu_\ell^2}\frac{a^{-\ell}}{(j-\ell)!}.
\end{equation}
\end{prop}
\begin{proof}
By Remark \ref{rmk:weight-in-L}, we have that
\begin{equation}
\label{eq:H0-D0-sum}
\cH^{(\lambda)}_0 = \sum_{x=0}^\infty W^{(\lambda)}(x) = L_0^{-1} (I+A)^\lambda \left(\sum_{x=0}^\infty \frac{a^x}{x!} L(x) T^{(\lambda)}L(x)^\ast \right) (I+A^\ast)^\lambda (L_0^\ast)^{-1},
\end{equation}
which gives \eqref{eq:H0-D0} by taking $\D_0^{(\la)}$ as the  inner sum in \eqref{eq:H0-D0-sum}. We can now take advantage of the orthogonality relations of the scalar Charlier polynomials. 
Entrywise we have
\begin{align*}
(\D_0^{(\la)})_{j,k}
&=
\sum_{x=0}^\infty
\frac{a^x}{x!}
\sum_{\ell=1}^{\min(j,k)}
(-a)^{j+k-2\ell}
\frac{\mu_j }{\mu_\ell}
\frac{c^{(a)}_{j-\ell}(x)}{(j-\ell)!}
\delta_\ell^{(\la)}
\frac{\mu_k }{\mu_\ell}
\frac{c^{(a)}_{k-\ell}(x)}{(k-\ell)!} \\
&= 
\sum_{\ell=1}^{\min(j,k)}
\frac{\mu_j }{\mu_\ell}
\frac{(-a)^{j+k-2\ell} \delta_\ell^{(\la)}}
{(j-\ell)!(k-\ell)!}
\frac{\mu_k }{\mu_\ell}
\sum_{x=0}^\infty
\frac{a^x}{x!}
c^{(a)}_{j-\ell}(x)
c^{(a)}_{k-\ell}(x) \\
&= \delta_{j,k}
\sum_{\ell=1}^{j}
a^{j-\ell}
\frac{\mu_j^2 }{\mu_\ell^2}
\frac{e^a \delta_\ell^{(\la)}}
{(j-\ell)!}.
\end{align*}
In the last equality we use the orthogonality relations \eqref{eq:ortogonalidad-charlier-escalar}.
\end{proof}

\begin{cor}
\label{cor:descomp-LDU-norma}
The square norm of the monic orthogonal polynomials $(P_n^{(\la)})_n$ with respect to the weight \eqref{eq:WL}, has the following LDU decomposition
$$
\cH_n^{(\la)}
=
L_0^{-1}(I+A)^{n+\la}\D_n^{(\la)}(I+A^\ast)^{n+\la}(L_0^\ast)^{-1},
$$
for some positive definite diagonal matrix $\D_n^{(\la)}$ that satisfies the recurrence
\begin{equation}\label{eq:Dnrec}
\D_n^{(\la)}
=
a^2 A\D_{n-1}^{(\la)}A^\ast
-
a^2 \D_{n-1}^{(\la)}A^\ast
\D_{n-1}^{(\la)-1}A\D_{n-1}^{(\la)}
+
a \D_{n-1}^{(\la)}
+
\D_{n-1}^{(\la)}\D_{n-2}^{(\la)-1}\D_{n-1}^{(\la)}, \quad n\in \N_{\geq 2},
\end{equation}
and
\begin{equation}
\label{eq:D1rec}
\D_1^{(\la)}
=
a^2 A\D_{0}^{(\la)}A^\ast
-
a^2 \D_{0}^{(\la)}A^\ast
\D_{0}^{(\la)-1}A\D_{0}^{(\la)}
+
a \D_{0}^{(\la)}.
\end{equation}
\end{cor}
\begin{proof}
We start by defining
$$
\D_n^{(\la)}
=
L_0(I+A)^{-n-\lambda}\cH_n^{(\la)}(I+A^*)^{-n-\lambda}L_0^*,
$$
and now prove that these are diagonal by induction.
To do this we will first show they satisfy the 
recursion in $n$ given in \eqref{eq:Dnrec} and
\eqref{eq:D1rec}. 
For $n=1$ we use the $n=1$ result in Theorem \ref{thm:recurrencia-norma} to get
\begin{multline*}
L_0(I+A)^{-\la-1}\cH_1^{(\la)}(I+A^*)^{-\la-1}L_0^\ast
 \\ 
=
a^2 (I+A)\D_{0}^{(\la)}(I+A^\ast)
-
a^2 \D_{0}^{(\la)}(I+A^\ast)\D_{0}^{(\la)-1}(I+A)\D_{0}^{(\la)}
+
a \D_{0}^{(\la)}.
\end{multline*}
When we expand the $(I+A)$ we get
\eqref{eq:D1rec}. 
Similarly for $n\geq 2$ we use Theorem \ref{thm:recurrencia-norma} to show that
\begin{multline*}
L_0(I+A)^{-\la-n}\cH_n^{(\la)}(I+A^*)^{-\la-n}L_0^\ast
=
a^2 (I+A)\D_{n-1}^{(\la)}(I+A^\ast)
+
\D_{n-1}^{(\la)}\D_{n-2}^{(\la)-1}\D_{n-1}^{(\la)} \\
-
a^2 \D_{n-1}^{(\la)}(I+A^\ast)\D_{n-1}^{(\la)-1}(I+A)\D_{n-1}^{(\la)}
+
a \D_{n-1}^{(\la)},
\end{multline*}
and we need to expand the $(I+A)$ out again to get
\eqref{eq:Dnrec}.

We already showed that $\D_0^{(\la)}$ is diagonal in Proposition \ref{prop:norma0}. Because of the
subdiagonal form of $A$ it then follows from
\eqref{eq:D1rec} that $\D_1^{(\la)}$ is also
diagonal. The same reasoning holds for $n\geq 2$
and \eqref{eq:Dnrec}. This completes the proof.
\end{proof}

\subsection{Second order difference equation}
\label{subsec:second-order-diff-op}
In this subsection we give a second order difference operator that has the polynomials $P_n^{(\lambda)}$ as eigenfunctions. We observe that the operators $D$ and $D^\dagger$ have a more specific form 
due to the more specific weight \eqref{eq:WL},
\begin{equation}\label{eq:newD}
D = \eta (I+A),
\qquad
D^\dagger = \eta^{-1}\frac{x}{a}(I+A)^{-1}.
\end{equation}
The operator $D + D^\dagger \in \mathcal{F}_R(P)$ is a self-adjoint operator. It is a well known fact, see \cite{MR3040334}, that a self-adjoint second order difference operator that preserves the degree of polynomials, has the corresponding sequence of matrix orthogonal polynomials as eigenfunctions. However, it follows from \eqref{eq:newD} that $D+D^\dagger$ does not preserves the degree of the polynomials. Keeping this in mind, we construct a second  self-adjoint operator in the following sense. 
\begin{prop} \label{prop:Jfrak}
Let the weight $W^{(\la)}$ be as in \eqref{eq:WL}.
Then the operator
\begin{equation}\label{eq:Jfrak}
\mathfrak{J}^{(\la)}
=
J 
+
(x+\la)(I+A)^{-1},
\end{equation}
is a self-adjoint operator with respect to $\langle \cdot,\cdot\rangle^{(\la)}$ i.e
$$
\langle P\cdot \mathfrak{J}^{(\la)}, Q\rangle^{(\la)}
=
\langle P, Q \cdot \mathfrak{J}^{(\la)} \rangle^{(\la)}
$$
for all matrix polynomials $P$ and $Q$.
\end{prop}
\begin{proof} 
We pull $J$ through
the weight by using Lemma \ref{lem:immediate2}
\begin{align*}
JW^{(\la)}(x)
&=
J(I+A)^{x+\la}T^{(\la)}(I+A^\ast)^{x+\la}
\\
&=
(I+A)^{x+\la}T^{(\la)}J(I+A^\ast)^{x+\la}
+
(x+\la)W^{(\la)}(x)
-
(x+\la)(I+A)^{-1}W^{(\la)}(x)
\\
&=
W^{(\la)}(x) J
-
(x+\la)(I+A)^{-1}W^{(\la)}(x)
+
W^{(\la)}(x)(x+\la)(I+A^\ast)^{-1}.
\end{align*}
This amounts to
$\mathfrak{J}^{(\la)}(x)W^{(\la)}(x)=W^{(\la)}(x)\mathfrak{J}^{(\la)}(x)^\ast$, which is a weak Pearson equation \eqref{prop:D-dagger} for $\ell=m=0$, and $F_0(x)=\widetilde F_0(x) = \mathfrak{J}^{(\la)}(x)$. Therefore Proposition \ref{prop:D-dagger} implies that $\mathfrak{J}^{(\la)}$ is a self-adjoint operator.
\end{proof}
\begin{rmk}
$\mathfrak{J}^{(\la)}$ is 
an operator of order zero as described in
Definition \ref{def:orden0}. Throughout this
paper we will denote $\mathfrak{J}^{(\la)}$
for this operator. We will however also denote by $\mathfrak{J}^{(\la)}(x)\in M_N(\C)[x]$
the matrix valued polynomial which is $\mathfrak{J}^{(\la)}$'s only
coefficient.
\end{rmk}

In the proof of  Proposition \ref{prop:Jfrak},
we found the weak Pearson equation $$\mathfrak{J}^{(\la)}(x)W^{(\la)}(x)=W^{(\la)}(x)\mathfrak{J}^{(\la)}(x)^\ast.$$
As in Proposition \ref{prop:ladder_sec2} this implies that $\mathfrak{J}^{(\la)}\in \mathcal{F}_R(P)$.  It is also clear from the definition that $\mathfrak{J}^{(\la)}$ does not preserve the degree of polynomials. This motivates the consideration of the following self-adjoint second order difference operator
\begin{equation}\label{eq:Dfrak}
\begin{aligned}
\mathfrak{D}^{(\la)} &= aD - \mathfrak{J}^{(\la)} + aD^\dagger
\\
&= \eta a(I+A) -J - (x+\la)(I+A)^{-1} + \eta^{-1} x(I+A)^{-1},
\end{aligned}
\end{equation}
because the terms that raise the degree of 
the polynomials cancel.
\begin{thm}
\label{thm:autovalor-Dfrak}
The monic MVOP orthogonal with respect ot $W^{(\la)}$ in \eqref{eq:WL}
satisfy the following recursion in $x$
\begin{equation}\label{eq:rec2}
P_n^{(\la)}\cdot \mathfrak{D}^{(\la)} 
=
\Gamma_n^{(\la)} P_n^{(\la)},
\qquad \Gamma_n^{(\la)} = a(I+A) -J -(n+\la)(I+A)^{-1}.
\end{equation}
\end{thm}
\begin{proof}
It follows from the fact that $\mathfrak{D}^{(\la)}$ is 
self-adjoint and preserves the degree of 
polynomials, that it has the $P_n^{(\la)}$ as eigenfunctions \cite{MR3040334}.
We can then find the eigenvalue $\Gamma_n^{(\la)}$
by looking at the leading coefficient of the equation.
\begin{align*}
P_n^{(\la)}\cdot\mathfrak{D}^{(\la)}(x)
&=
P_n^{(\la)}(x+1)a(I+A)  
-
P_n^{(\la)}(x)J
-
P_n^{(\la)}(x)x(I+A)^{-1}
\\
&\hspace{5cm} -
P_n^{(\la)}(x)\la (I+A)^{-1}
+
P_n^{(\la)}(x-1)x(I+A)^{-1}
\\
&=
\left(a(I+A)-J-\la(I+A)^{-1}-n(I+A)^{-1} \right)x^n +\cdots
\end{align*}
This completes the proof of the theorem.
\end{proof}
\begin{rmk}
The matrix weight $W^{(\lambda)}$ is closely related to one of the examples in \cite{MR3040334}. If we set $\lambda=0$ and consider an upper triangular matrix $A$ we can identify the differential operator $\mathfrak{D}^{(0)}(x)$ with that in \cite{MR3040334}.
\end{rmk}
\begin{rmk}
\label{rmk:psiJ}
From the previous theorem it is easy to apply
\begin{align*}
\psi^{-1}(\mathfrak{J}^{(\la)}) 
&= 
a\psi^{-1}(D) + a\psi^{-1}(D^\dagger) - \psi^{-1}(\mathfrak{D}^{(\la)})
= aM +aM^\dagger - \Gamma_n^{(\la)}
\\
&=
(I+A)^{-1}\delta 
+ \left(J+(n+\la)(I+A)^{-1}+a\cH_n^{(\la)}(I+A^\ast)\cH_n^{(\la)-1} \right)\\
&\hspace{7.5cm}+ \cH_n^{(\la)}(I+A^\ast)^{-1}\cH_{n-1}^{(\la)-1}\delta^{-1}.
\end{align*}
In Sections \ref{sec:dualOP} and \ref{sec:dual-dual} we will see that the
equation $\psi^{-1}(\mathfrak{J}^{(\la)})\cdot P_n=P_n \cdot \mathfrak{J}^{(\la)}$ will function as an alternative three-term recurrence relation.
\end{rmk}

\subsection{The entries of the matrix Charlier polynomials}
In this subsection we write the entries of the monic polynomials $(P_n^{(\lambda)})_n$ explicitly in terms of the scalar Charlier polynomials. The expression involves certain coefficients which will be described in detail in Section \ref{sec:7}. Firstly we define an auxiliary matrix polynomial
\begin{equation*} 
    R_n^{(\la)}(x)
    =
    L_0 (I+A)^{-n-\la}P_n^{(\la)}(x) (I+A)^{\la+x},
\end{equation*}
for which the corresponding
equation to \eqref{eq:rec2}
is diagonal. Note that since $A$ is a nilpotent matrix, $R_n^{(\lambda)}$ is a polynomial in $x$.

\begin{prop}
\label{cor:entradas-R}
The entries of $R_n^{(\la)}(x)$ are given by
\begin{equation} \label{eq:decomp}
\left( R_n^{(\la)}(x) \right)_{j,k} = \xi_{j,k,n}^{(\la)} c_{n+j-k}^{(a)}(x),
\qquad n+j-k \geq 0,
\end{equation}
and equal to $0$ when $n+j-k <0$.
The $\xi_{j,k,n}^{(\la)}$ are independent of $x$.
\end{prop}
\begin{proof}
Using Lemma \ref{lem:immediate2} one can show that \eqref{eq:rec2} for this new quantity becomes
$$
R_n^{(\la)}(x)
\left(\eta aI 
-
(J+(x+\la)I)
+
\eta^{-1} x I\right) 
= ((a-n-\la)I-J) R_n^{(\la)}(x).
$$
Entrywise this is
$$
a\left( R_n^{(\la)}(x+1)\right)_{j,k}
-
(x+a) \left( R_n^{(\la)}(x)\right)_{j,k}
+
x \left( R_n^{(\la)}(x-1)\right)_{j,k}
=
-(n+j-k)\left( R_n^{(\la)}(x)\right)_{j,k}.
$$
This is exactly the difference 
equation \eqref{eq:scalarrec} satisfied
by scalar
Charlier polynomials $c_{n+j-k}^{(a)}(x)$ and, since $R_n^{(\lambda)}$ is a polynomial in $x$, this determines the $x$ dependence of the entries
up to a constant multiple that we denote $
\xi_{j,k,n}^{(\la)}$.
\end{proof}

In Section \ref{sec:7} we will be able 
to determine the 
$\xi_{j,k,n}^{(\la)}$ and will therefore have
explicit expressions for the $P_n^{(\la)}$
as given in the following Corollary.

\begin{cor}\label{cor:PnEntries}
The entries of the monic MVOP are given by
$$
\left(
P_n^{(\la)}(x)
\right)_{jk}
=
\sum_{\ell =1}^j
\sum_{s=1}^{(n+\ell)\wedge N}
\sum_{t=k}^s
\frac{\mu_j \mu_s}{\mu_k \mu_\ell }
\frac{(-1)^{t-k}a^{j+s-\ell-k}\xi_{\ell,s,n}^{(\la)}}{(j-\ell)!(s-t)!(t-k)!}
c_{j-\ell}^{(-a)}(n+\la)
c_{n+\ell-s}^{(a)}(x)
c_{s-t}^{(-a)}(-x-\la),
$$
where $(n+\ell)\wedge N$ denotes $\min((n+\ell), N)$.
\end{cor}
\begin{proof}
To have all the $x$-dependence in terms of
scalar Charlier polynomials, we multiply with $L_0^{-1}L_0$ from the right.
So then
\begin{equation}\label{eq:compare}
P_n^{(\la)}(x) 
= L(-n-\la)^{-1}R_n^{(\la)}(x)L(x+\la)^{-1}L_0.
\end{equation}
Then we only need to collect all the separate
matrix entries from \eqref{eq:def-L}, \eqref{eq:decomp} and Lemma \ref{lem:immediate}.
\end{proof}
\begin{rmk}
The factors of $L^{-1}$ are responsible for the
scalar Charlier polynomials with non-standard parameter
values $c^{(-a)}$.
Something similar happens with the MVOP matrix 
entries in a Hermite-type \cite[Theorem 3.13]{IKR2} and 
Gegenbauer-type \cite[Theorem 3.4]{KdlRR} case.
In those cases the
analogous expressions for \eqref{eq:compare} do not
have the $n$-dependent factor on the left, but do
have the $x$-dependent factor on the right. 
\end{rmk}

\section{Shift operators and recurrence relations for  the matrix valued Charlier polynomials}
\label{sec:pearson-charlier}
The goal of this section is to construct a one parameter family of matrix weights $W^{(\lambda)}$ for $\lambda \in \mathcal{V}=\mathbb{N}_{0}$ in such a way that the strong Pearson equation \eqref{eq:Pearson} holds true. As a consequence, we obtain a one parameter family of matrix weights with explicit shift operators, whose square norms and three-term recurrence relations are given explicitly. These results will be the main ingredient of Section \ref{sec:dualOP}, where we describe the entries of $P_n^{(\lambda)}$ terms of the scalar Charlier polynomials and the dual Hahn polynomials. For this we need to impose conditions on the parameters $\mu_{i}$ and $\delta_{i}^{(\lambda)}$.  We assume that that there exist real numbers $d^{(\lambda)}\neq 0$ and $c^{(\lambda)}$ such that:
\begin{equation}
\label{eq:conditiones-delta}
\frac{\delta_{i}^{(\lambda+1)}}{\delta_{i}^{(\lambda)}}-\frac{\delta_{i+1}^{(\lambda+1)}}{\delta_{i+1}^{(\lambda)}}=-\frac{d^{(\lambda)}a}{2},\quad i=1,\cdots,N-1,
\end{equation}
\begin{equation}
\label{eq:conditiones-mus}
\begin{split}
\frac{\mu^2_{i}}{\mu^2_{i-1}}\frac{\delta_{i-1}^{(\lambda+1)}}{\delta_{i}^{(\lambda)}}-\frac{\mu^2_{i+1}}{\mu^2_{i}}\frac{\delta_{i}^{(\lambda+1)}}{\delta_{i+1}^{(\lambda)}}=d^{(\lambda)}i+c^{(\lambda)}, \quad i=1,\cdots,N.
\end{split}
\end{equation}
Similar nonlinear equations were obtained in the case of Hermite and Laguerre type matrix orthogonal polynomials, see \cite{IKR2, KoelinkRlaguerre}.

\subsection{Shift operators}
In this subsection we establish sufficient conditions for the hypotheses of the Theorem \ref{thm:Pearson} and we therefore obtain explicit shift operators for the family of matrix weights $W^{(\lambda)}$. We first need the following remark.

\begin{rmk} 
\label{rmk:coeficientes-Delta-Polyn-grado2}
For any $M\in \mathbb{C}^{N\times N}$ we have
\begin{align}
\left((A^{*}+I)^{-x}M(A^{*}+I)^{x}\right)\cdot\Delta &=(A^{*}+I)^{-x-1}M(A^{*}+I)^{x+1} - (A^{*}+I)^{-x}M(A^{*}+I)^{x} \nonumber\\
\label{eq:corcheteMA}
& =(A^{*}+I)^{-x-1}\left[M,A^{*}\right](A^{*}+I)^{x}.
\end{align}

If $P(x)=A_{2}x^{2}+A_{1}x+A_{0}$ is a matrix polynomial of degree two,  then 
\begin{equation*}
    A_{2}=\frac{1}{2}(P\cdot\Delta^{2})(x),\qquad  A_{1}= (P\cdot\Delta)(0)-A_2.
\end{equation*}
\end{rmk}

\begin{thm}
\label{thm:PearsonCharlier}
Let $W$ be a matrix weight as in \eqref{eq:WL} such that \eqref{eq:conditiones-delta}, \eqref{eq:conditiones-mus} are satisfied. Let $\Phi^{(\lambda)}$ and $\Psi^{(\lambda)}$ be given by
$$\Phi^{(\lambda)}(x)=W^{(\lambda)}(x)^{-1}W^{(\lambda+1)}(x), \qquad 
\Psi^{(\lambda)}(x)=W^{(\lambda)}(x)^{-1}(W^{(\lambda+1)}\cdot\Delta)(x-1).$$
Then $\Phi^{(\lambda)}$ is a matrix polynomial of degree at most two and $\Psi^{(\lambda)}$ is a matrix polynomial of degree at most one. Moreover, we have
$$\Phi^{(\lambda)}(x) = x^2 \mathcal{K}^{(\lambda)}_2 + x \mathcal{K}^{(\lambda)}_1 +\mathcal{K}^{(\lambda)}_0, \qquad \Psi^{(\lambda)}(x) = x \mathcal{J}^{(\lambda)}_1 + \mathcal{J}^{(\lambda)}_0,$$
where
\begin{align*}
\mathcal{K}^{(\la)}_2 &=-\frac{d^{(\lambda)}}{2}A^{*}(A^{*}+I)^{-1}, \qquad  
 \mathcal{K}^{(\la)}_1 =  \frac{d^{(\lambda)}}{2} \left( 2J-aA^*-(2\lambda+1)A^\ast (A^*+I)^{-1} \right)+c^{(\lambda)},\\
\mathcal{K}^{(\lambda)}_0  &= (A^{*}+I)^{-\lambda} (T^{(\lambda)})^{-1}(A+I)T^{(\lambda+1)}(A^{*}+I)^{\lambda+1},\\
\mathcal{J}_1^{(\lambda)} 
&= 
\mathcal{K}^{(\la)}_2  
+ 
\mathcal{K}^{(\la)}_1
-
\frac{1}{a}(A^{*}+I)^{-\lambda-1} (T^{(\lambda)})^{-1}T^{(\lambda+1)}(A^{*}+I)^{\lambda+1},
\quad \mathcal{J}_0^{(\lambda)}=\mathcal{K}^{(\lambda)}_0.
\end{align*}
\end{thm}
\begin{proof}
We divide the proof in three parts.

\noindent \emph{Part 1:} First we prove that $\Phi^{(\lambda)}$ is a polynomial of degree at most two. We have
\begin{equation}
\label{eq:formula-Phi-generica}
\Phi^{(\lambda)}(x)=(W^{(\lambda)}(x))^{-1}W^{(\lambda+1)}(x)=(A^\ast+I)^{-x-\lambda}(T^{(\lambda)})^{-1}(A+I)T^{(\lambda+1)}(A^\ast+I)^{x+\lambda+1}.
\end{equation}
Therefore, $\Phi^{(\lambda)}$ is a matrix polynomial. We write
$$M^{(\lambda)} = M_1^{\lambda)}+M_2^{(\lambda)},\qquad M_1^{(\lambda)} = (T^{(\lambda)})^{-1}AT^{(\lambda+1)}, \qquad M_2^{(\lambda)} = (T^{(\lambda)})^{-1}T^{(\lambda+1)},$$
so that
\begin{equation}
\label{eq:expresion-phi-charlier}
\Phi^{(\lambda)}(x)=(A^{*}+I)^{-x-\lambda}M^{(\lambda)}(A^{*}+I)^{x+\lambda+1}.
\end{equation}
Since $\Phi^{(\lambda)}$ is a polynomial, is enough to check that $\Phi^{(\lambda)}\cdot\Delta^2$ is a constant. If we apply the operator $\Delta$ to \eqref{eq:expresion-phi-charlier}, by \eqref{eq:corcheteMA} we get
\begin{equation}
\label{eq:delta-phi-charlier}
(\Phi^{(\lambda)}\cdot\Delta)(x)=
(A^{*}+I)^{-x-\lambda-1}[M^{(\lambda)},A^{*}](A^{*}+I)^{x+\lambda+1}.
\end{equation}
Applying the operator $\Delta$ and \eqref{eq:corcheteMA} again, we get
\begin{equation}
\label{eq:coeficiente-principal-phi}
(\Phi^{(\lambda)}\cdot\Delta^{2})(x)=(A^{*}+I)^{-x-\lambda-2}[[M^{(\lambda)},A^{*}],A^{*}](A^{*}+I)^{x+\lambda+1}.
\end{equation}
Using the nonlinear relations \eqref{eq:conditiones-delta} we obtain
\begin{equation}
\label{eq:corchete-M1-A*}
[M^{(\lambda)}_{1},A^{*}]=d^{(\lambda)}J + c^{(\lambda)}, \qquad [M^{(\lambda)}_{2},A^{*}]=-\frac{d^{(\lambda)}a}{2}A^{*}.
\end{equation}
Therefore
$$[[M^{(\lambda)},A^{*}],A^{*}]=d^{(\lambda)}[J,A^{*}]=-d^{(\lambda)}A^{*},$$
so that \eqref{eq:coeficiente-principal-phi} becomes
\begin{equation}
\label{eq:coeficiente-principal-phi-1}
(\Phi^{(\lambda)}\cdot\Delta^{2})(x)=-d^{(\lambda)}A^{*}(A^{*}+I)^{-1}.
\end{equation}
As a conclusion, we get that $\Phi^{(\lambda)}$ is a polynomial of degree less or equal than two.

\medskip

\noindent \emph{Part 2:} Now we prove that $\Psi^{(\lambda)}$ is a polynomial of degree at most one. Using the definitions of $W^{(\lambda)}$, $\Phi^{(\lambda)}$ and $M_2^{(\lambda)}$ we obtain
\begin{equation}
    \label{eq:appendix-eq-Psi}
\Psi^{(\lambda)}(x)=(W^{(\lambda)}(x))^{-1} (W^{(\lambda+1)}\cdot\Delta)(x-1)=\Phi^{(\lambda)}(x)- \frac{x}{a}(A^{*}+I)^{-x-\lambda}M_2^{(\lambda)}(A^{*}+I)^{x+\lambda}.
\end{equation}
Therefore, $\Psi^{(\lambda)}$ is a matrix polynomial. In order to prove that the degree of $\Psi^{(\lambda)}$ is at most one, we will show that $\Psi^{(\lambda)}\cdot \Delta^2=0$. By the Leibniz rule \eqref{thm:derivacion-producto-discreta} for $\Delta$ on the the second term of right hand side of \eqref{eq:appendix-eq-Psi} 

\begin{align*}
\left(-\frac{x}{a}(I+A^\ast)^{-x-\la}M_2^{(\la)} (I+A^\ast)^{x+\la} \right)\Delta
&=
-\frac{1}{a}(I+A^\ast)^{-x-\la-1}M_2^{(\la)} (I+A^\ast)^{x+\la+1}
\\
&\qquad  -
\frac{x}{a}(I+A^\ast)^{-x-\la-1}[M_2^{(\la)},A^\ast] (I+A^\ast)^{x+\la},
\end{align*}
and using \eqref{eq:corcheteMA} and \eqref{eq:corchete-M1-A*} we get
\begin{equation}
\label{eq:appendix-eq-Psi2}
(\Psi^{(\lambda)}\cdot \Delta)(x) = (\Phi^{(\lambda)}\cdot\Delta)(x)- 
\frac{1}{a}(A^\ast+1)^{-x-\lambda-1} M_2^{(\lambda)} (A^\ast+1)^{x+\lambda+1} + \frac{x d^{(\lambda)}}{2} A^\ast(A^\ast+1)^{-1}.
\end{equation}
Applying $\Delta$ on \eqref{eq:appendix-eq-Psi2} together with \eqref{eq:corcheteMA}, \eqref{eq:corchete-M1-A*} and \eqref{eq:coeficiente-principal-phi-1} we obtain
\begin{align*} 
(\Psi^{(\lambda)}\cdot \Delta^2)(x) = (\Phi^{(\lambda)}\cdot\Delta^2)(x)+ \frac{d^{(\lambda)}}{2} A^\ast(A^\ast+1)^{-1} + \frac{d^{(\lambda)}}{2} A^\ast(A^\ast+1)^{-1} = 0.
\end{align*}
We conclude that the degree of $\Psi^{(\lambda)}$ is strictly less than two. 

\emph{Part 3: } From Part 1 and Part 2 we have that
$$\Phi^{(\lambda)}(x) = x^2 \mathcal{K}^{(\lambda)}_2 + x \mathcal{K}^{(\lambda)}_1 +\mathcal{K}^{(\lambda)}_0, \qquad \Psi^{(\lambda)}(x) = x \mathcal{J}^{(\lambda)}_1 + \mathcal{J}^{(\lambda)}_0,$$ 
for certain matrices $\mathcal{K}^{(\lambda)}_2, \mathcal{K}^{(\lambda)}_1, \mathcal{K}^{(\lambda)}_0$ and $\mathcal{J}^{(\lambda)}_1, \mathcal{J}^{(\lambda)}_0$. The expression of $\mathcal{K}^{(\lambda)}_2$ follows directly from \eqref{eq:coeficiente-principal-phi-1} and Remark \ref{rmk:coeficientes-Delta-Polyn-grado2}. The expressions of $\mathcal{K}^{(\lambda)}_0$ and $\mathcal{J}^{(\lambda)}_0$ follow by evaluating \eqref{eq:formula-Phi-generica} and  \eqref{eq:appendix-eq-Psi} at $x=0$. If we replace $x=0$ in \eqref{eq:delta-phi-charlier} we get
$$\mathcal{K}_1^{(\lambda)}  = (\Phi^{(\lambda)}\cdot \Delta)(0)-\mathcal{K}_2^{(\lambda)}.$$
Now the expression of $\mathcal{K}_1^{(\lambda)}$ follows directly from \eqref{eq:corchete-M1-A*} and Lemma \ref{lem:immediate2}. Finally if we set $x=0$ in \eqref{eq:appendix-eq-Psi2} and we use Remark \ref{rmk:coeficientes-Delta-Polyn-grado2}, we obtain
$$
\mathcal{J}_1^{(\lambda)} = (\Psi^{(\lambda)}\cdot \Delta)(0)= \mathcal{K}_2^{(\lambda)}  + \mathcal{K}_1^{(\lambda)}  
- 
\frac{1}{a}(A^\ast+1)^{-\lambda-1}M_2^{(\lambda)} (A^\ast +1)^{\lambda+1}.
$$
This completes the proof of the theorem.
\end{proof}

Now from Theorem \ref{thm:Pearson} we immediately obtain the shift operators $\Delta$ and $S^{(\lambda)}$ such that
\begin{equation}
    \label{eq:shifts-for-Charlier}
P^{(\lambda)}_n\cdot\Delta(x) = n P^{(\lambda+1)}_{n-1}(x), \qquad P_{n-1}^{(\lambda+1)}\cdot S^{(\lambda)}=G^{(\lambda)}_{n}P_{n}^{(\lambda)},
\end{equation}
where $G^{(\lambda)}_n = -(n-1)\mathcal{K}_2^{(\lambda)\ast} - \mathcal{J}_1^{(\lambda)\ast}$. Next use the LDU decomposition of the square norm $\mathcal{H}_n^{(\lambda)}$ to give an explicit diagonalization of $G_n^{(\lambda)}$.

\begin{lem}
\label{lem:Gn-diagonalizable}
The matrices $G_n^{(\lambda)}$ for the Charlier matrix polynomials diagonalize in the following way
\begin{align}
\label{eq:Gn-diagonalizable}
\frac{1}{n+1}G_{n+1}^{(\lambda-1)}=L_0^{-1}(I+A)^{n+\la}\mathcal{D}_n^{(\lambda)}(\mathcal{D}_{n+1}^{(\lambda-1)})^{-1}(I+A)^{-n-\la}L_0.
\end{align}
\end{lem}
\begin{proof}
As we have seen in equation \eqref{eq:relacion-Hn-Hn-1}, square norms of the monic orthogonal polynomials $P_n^{(\lambda)}$ are related by
\begin{equation}
\label{eq:ecuacion-norma}
n\cH_{n-1}^{(\lambda+1)}=G_{n}^{(\lambda)}\cH_{n}^{(\lambda)}.
\end{equation}
From the LDU decomposition of the square norm $\mathcal{H}_n^{(\lambda)}$ in Corollary \ref{cor:descomp-LDU-norma} we get
\begin{equation}
\label{eq:ecuacion-norma-2}
\cH_n^{(\la)}
=
L_0^{-1}(I+A)^{n+\la}\D_n^{(\la)} (I+A^*)^{n+\la}(L_0^{*})^{-1},    
\end{equation}
for some positive definite diagonal matrix $\D_n^{(\la)}$. 
From \eqref{eq:ecuacion-norma} and \eqref{eq:ecuacion-norma-2} we get 
\eqref{eq:Gn-diagonalizable}.
\end{proof}

\subsection{Explicit expression for the norms}
Since the relations \eqref{eq:conditiones-delta} and \eqref{eq:conditiones-mus} are non-linear, we cannot find all the solutions in general. For the rest of the paper we will specify a set of solutions to these equations, which is simple enough to describe in detail all the difference operators, square norms and the entries of the monic orthogonal polynomials. These formulas will be the main ingredient to describe the dual Charlier polynomials in Section \ref{sec:dualOP}. We will take $d^{(\lambda)} =1$, $c^{(\lambda)} = -\frac{(N+1)}{2}$ 
and the parameters of the form
\begin{equation}\label{eq:parameters-sol}
\left(\frac{\mu_j}{\mu_{k}}\right)^2 
= 
a^{k-j} \frac{(N-k)!}{(N-j)!}
,
\qquad
\delta_k^{(\la)}
= \left(\frac{a}{2}\right)^\la \frac{(\la +k-1)!}{(k-1)!}.
\end{equation}
It is a simple computation to check that the $\mu_j$'s and $\delta_j^{(\lambda)}$'s are solutions to \eqref{eq:conditiones-delta} and \eqref{eq:conditiones-mus}. Some products of matrices that will become
useful in this section are
\begin{equation}\label{eq:simplify}
\begin{aligned}
\left( T^{(\la+1)}T^{(\la)-1}\right)_{jj}
&=
\frac{\delta_j^{(\la+1)}}{\delta_j^{(\la)}}
=
\frac{a}{2}(j+\la),
 &T^{(\la+1)}T^{(\la)-1}
 &=
\frac{a}{2}(J+\la I),
\\
\left( T^{(\la+1)}A^\ast T^{(\la)-1}\right)_{j,j+1}
&=
\frac{\delta_j^{(\la+1)}}{\delta_{j+1}^{(\la)}}\frac{\mu_{j+1}}{\mu_j}
=
\frac{\sqrt{a}}{2}j\sqrt{N-j},
 &T^{(\la+1)} A^\ast  T^{(\la)-1}
 &=
\frac{a}{2} J A^\ast .
\end{aligned}
\end{equation}

\begin{thm}
\label{thm:expresion-D}
With the parameters as in \eqref{eq:parameters-sol}
the diagonal matrices $\D_n^{(\la)}$ are given by
\begin{equation}
\label{eq:expresion-entrada-D}
( \mathcal{D}_n^{(\la)})_{jj}= (-1)^{j} e^a n! a^n \left( \frac{a}{2}\right)^\la \frac{(\la+N-j)!}{(j-1)!}
\frac{(-N-\la-n)_{j}}{(\la+n+1)_{N-j+1}},
\end{equation}
for all $n\in \mathbb{N}_0$. Moreover, the following recursion holds true
\begin{equation}
\label{eq:expresion-Dn}
\D_n^{(\la)}
=
n!2^n \D_{0}^{(\la+n)}
\left(\Lambda^{(\la)}_n\right)^{-1},
\qquad
\Lambda^{(\la)}_n
=
\prod_{m=1}^n \left( N+\la+m -J \right).
\end{equation}
\end{thm}
\begin{proof}
We start considering the case $n=0$. If we replace the explicit parameters \eqref{eq:parameters-sol} in \eqref{eq:norma0} we obtain
\begin{align}
\label{eq:entradas-D0}
(\D_0^{(\la)})_{j,j}
&=
e^a \sum_{k=1}^j \left(\frac{\mu_j}{\mu_k}\right)^2 \frac{a^{j-k}}{(j-k)!} \delta_k^{(\la)}
=
e^a \left(\frac{a}{2}\right)^\la \frac{1}{(N-j)!} \sum_{k=1}^j  \frac{(N-k)!(\la + k-1)!}{(j-k)!(k-1)!}\nonumber
\\
&=
e^a \left(\frac{a}{2}\right)^\la \frac{\la !(N-1)!}{(N-j)!(j-1)!} \rFs{2}{1}{1-j,\la +1}{1-N}{1}
=
e^a \left(\frac{a}{2}\right)^\la \la !  \binom{N+\lambda}{j-1} .
\end{align}
In the last line we use a Chu-Vandermonde type result.

Lemma \ref{lem:Gn-diagonalizable} implies that
$$
\mathcal{D}_n^{(\lambda)}(\mathcal{D}_{n+1}^{(\lambda-1)})^{-1}
=
\frac{1}{n+1}L_0(I+A)^{-n-\la}G_{n+1}^{(\lambda-1)}(I+A)^{n+\la}L_0^{-1}.
$$
On the other hand, from Theorem \ref{thm:Pearson} we write $G^{(\lambda)}_n = -(n-1)\mathcal{K}_2^{(\lambda)\ast} - \mathcal{J}_1^{(\lambda)\ast}$. In turn we have nice expressions for these
coefficients given in Theorem 
\ref{thm:PearsonCharlier}. With the new parameters the matrices $\mathcal{K}^{(\lambda)}_2$ and $\mathcal{J}^{(\lambda)}_1$ simplify to
$$
\mathcal{K}^{(\lambda)}_2=-\tfrac{1}{2}A^*(A^*+I)^{-1},
$$
$$
\mathcal{J}^{(\lambda)}_1=\tfrac{1}{2}(J-aA^*-N-1-\la-(\la+1)A^*(I+A^*)^{-1}).
$$
Thus the only contribution to the diagonal of $\frac{1}{n+1}L_0(I+A)^{-n-\la}G_{n+1}^{(\lambda-1)}(I+A)^{n+\la}L_0^{-1}$ is given by the term $\frac{1}{2}(N+\la-J).$
Applying properties of the matrices $J$ and $L(x)$ and the fact that $\D_n^{(\la)}$ is a diagonal matrix we get
\begin{equation}
\label{eq:relacion-D}
\mathcal{D}_n^{(\lambda)}
(\mathcal{D}_{n+1}^{(\lambda-1)})^{-1}
=
\frac{1}{2(n+1)}\left(N+\la-J\right).
\end{equation}
Iterating this equation we get
\begin{equation*}
(\mathcal{D}_{n}^{(\lambda)})^{-1}=\frac{1}{n!2^n}(\mathcal{D}_{0}^{(\lambda+n)})^{-1}\Lambda^{(\lambda)}_n,
\end{equation*}
where $\Lambda^{(\lambda)}_n=(\lambda+n+N-J)\cdots(\lambda+2+N-J)(\lambda+1+N-J).$
So $\Lambda^{(\lambda)}_n$ is a diagonal matrix with entries
$$(\Lambda^{(\lambda)}_n)_{i,i}=(\lambda+N-i+1)_n.$$
From \eqref{eq:entradas-D0} and \eqref{eq:expresion-Dn}, the entries of $\D_n^{(\la)}$ are 
\begin{equation}
\label{eq:D_n-proof}
( \mathcal{D}_n^{(\la)})_{jj}
=
e^a n! a^n \left( \frac{a}{2}\right)^\la \frac{(\la +n)!}{(N+\la -j +1)_n} \binom{N+\la +n}{j-1}.
\end{equation}
Finally we obtain \eqref{eq:expresion-entrada-D} by simple manipulations of the expression \eqref{eq:D_n-proof}.
\end{proof}

\subsection{Three term recurrence relation for the Charlier polynomials}
We are now ready to compute explicitly the coefficients of the three term recurrence relation for monic polynomials $(P_n^{(\lambda)})_n$. The recurrence coefficients will be written in terms of the following upper triangular matrix.
\begin{equation}
    \label{eq:defn-A-cal}
\cA_{\la}=(N+\la -J)^{-1}J A^\ast.
\end{equation}
The matrix $\mathcal{A}_{\la}$ will appear in Section \ref{sec:dualOP} where it will play, for the dual polynomials, a role that is similar to that of the matrix $A$ in this setting.
\begin{thm}
\label{thm:coef-3term}
The coefficient $C_n^{(\lambda)}$ of the three term recurrence relation is
\begin{equation}
\label{eq:expresion-C}
C_n^{(\lambda)}=L_0^{-1}(I+A)^{n+\la}\mathcal{D}_n^{(\lambda)}(I+A^*)(\mathcal{D}_{n-1}^{(\lambda)})^{-1}(I+A)^{-n-\la+1}L_0,
\end{equation}
and the coefficient $B_n^{(\lambda)}$ of the three term recurrence relation is
\begin{align*}
B_n^{(\lambda)} &=a L_0^{-1}(I+A)^{n+\la}\Bigl((n+1)(I+A)\cA_{n+\la+1}-n\cA_{n+\la}(I+A)
\\
&\hspace{8cm} +I+A+\frac{n}{a}
\Bigr)(I+A)^{-\la-n}L_0.
\end{align*}
\end{thm}
\begin{proof}
The first part follows directly from $C_n^{(\lambda)}=\cH_n^{(\lambda)}(\cH_{n-1}^{(\lambda)})^{-1}$ and Corollary \ref{cor:descomp-LDU-norma}. 
For the second part, recall from  
\eqref{eq:relacion-Bn-X1} that 
\begin{equation*}
B_n^{(\lambda)}=nX_{1}^{(\lambda+n-1)}-(n+1)X_{1}^{(\lambda+n)}+n,
\end{equation*}
where $X_1^{(\lambda)}$ is the subleading coefficient of $P_1^{(\lambda)}$. Since $X_1^{(\la)}=-B_0^{(\la)}$, from \eqref{eq:ecuacion-B0} and Proposition \ref{prop:norma0} we get
\begin{align*}
B_0^{(\la)}
&=
a L_0^{-1} (I+A)^{\la+1} \mathcal{D}_0^{(\la)} (I+A^\ast )
 \left(\mathcal{D}_0^{(\la)} \right)^{-1}
 (I+A)^{-\la} L_0.
\end{align*}
Applying \eqref{eq:expresion-entrada-D} we obtain
$$
\frac{(\mathcal{D}_0^{(\la)}  )_{j,j}}{( \mathcal{D}_0^{(\la)}  )_{(j+1),(j+1)}}
=
 \frac{j}{N+\la+1-j},
$$
which tells us that
$$
\mathcal{D}_0^{(\la)} A^\ast 
 \left( \mathcal{D}_0^{(\la)} \right)^{-1}
 =
 \left( N+\la+I - J\right)^{-1}  J A^\ast = \cA_{\la+1}.
$$
Therefore we have
$$
X_1^{(\lambda)}=
-a L_0^{-1}(I+A)^{\lambda+1}(I+\cA_{\la+1})(I+A)^{-\lambda}L_0,
$$
and this completes the proof of the theorem.
\end{proof}
We end this subsection by showing that the matrix valued sequence $C_n^{(\lambda)}$ vanishes at $n=0$. This fact will be crucial in Section \ref{sec:dual-dual}.
\begin{cor}
The coefficient $C^{(\lambda)}_n$ extends to a unique rational function in $n$ such that $C^{(\lambda)}_0=0$.
\end{cor}
\begin{proof}
It follows directly from the expression of $\mathcal{D}^{(\lambda)}_n$ in \eqref{eq:expresion-entrada-D} that the right hand side of $\eqref{eq:expresion-C}$ is a rational function in $n$. More precisely, we have
\begin{align*}
(\mathcal{D}^{(\lambda)}_n (\mathcal{D}^{(\lambda)}_{n-1})^{-1})_{j,j}&= \frac{na(\lambda+n)(N+\lambda+n)}{(n+N+\lambda-j)(N+\lambda+n-j+1)}, \\
(\mathcal{D}^{(\lambda)}_n A^\ast(\mathcal{D}^{(\lambda)}_{n-1})^{-1})_{j,j+1} &= 
\frac{na(\lambda+n)(N+\lambda-j)(N+\lambda+n)j}{(n+N+\lambda-j)^2(N+\lambda+n-j+1)(N+\lambda+n-j-1)},
\end{align*}
which clearly vanish at $n=0$. Since both $L_0^{-1}(I+A)^{n+\la}$ and $(I+A)^{-n-\la+1}L_0$ are polynomials in $n$, we conclude from \eqref{eq:expresion-C} that $C_0^{(\lambda)}=0$.
\end{proof}

\subsection{Difference operators and Darboux transforms}

Throughout this paper, we have proved that the matrix Charlier polynomials $P_n^{(\la)}$ are solutions to four second order difference equations. The difference equation in Corollary \ref{cor:rec1} involves multiplications by nontrivial matrices on both left and right side and we will not discuss it further in this section. 

We are now left with three nontrivial difference operators $\mathfrak{D}^{(\la)}, \Delta S^{(\la)},  S^{(\la-1)}\Delta \in \mathcal{B}_R(P)$, given respectively in \eqref{eq:rec2} and Theorem  \ref{thm:formula-de-rodrigues-discreta-matricial}. We note that the operators $\Delta S^{(\la)}$ and $S^{(\la-1)}\Delta$ are each others Darboux transforms. In the following theorem we show that there exists a linear relation between these operators.

\begin{thm}
\label{thm:darboux-transf-deltaS}
The second order difference operators $\mathfrak{D}^{(\la)}, \Delta S^{(\la)},  S^{(\la-1)}\Delta$ are related by
$$
2(\Delta S^{(\la)} - S^{(\la-1)} \Delta)
=
 \left( a - N -2 \la \right)I
 -
\mathfrak{D}^{(\la)}.
$$
\end{thm}
\begin{proof}
The operator $\mathfrak{D}^{(\la)}$ is given in \eqref{eq:Dfrak} and the operators $\Delta S^{(\la)}$, $S^{(\la-1)} \Delta$ are constructed using Theorem \ref{thm:Pearson}, yielding 
\begin{equation}\label{eq:twoDiffs}
\begin{aligned}
\Delta S^{(\la)}
&=
-
\eta \Phi^{(\la)\ast}(x)
+ 
2\Phi^{(\la)\ast}(x) -\Psi^{(\la)\ast}(x)
+ \eta^{-1} \left( \Psi^{(\la)\ast}(x)
- \Phi^{(\la)\ast}(x)
\right),
\\
S^{(\la-1)} \Delta
&=
-
\eta \Phi^{(\la-1)\ast}(x+1)
+ 
\Phi^{(\la-1)\ast}(x+1)
+
\Phi^{(\la-1)\ast}(x) 
-
\Psi^{(\la-1)\ast}(x+1)
\\
&\hspace{8cm}+
\eta^{-1} \left( \Psi^{(\la-1)\ast}(x)
- \Phi^{(\la-1)\ast}(x)
\right).
\end{aligned}
\end{equation}
With the help of \eqref{eq:simplify}, we obtain
\begin{equation}\label{eq:phidecomp}
\begin{aligned}
\Phi^{(\la)\ast}(x)
&=
\frac{a}{2}(I+A)^{x+\la+1}
(J(I+A^\ast)+\la )
(I+A)^{-x-\la} 
\\
\Phi^{(\la)\ast}(x)
-
\Psi^{(\la)\ast}(x)
&=
\frac{x}{2}
(I+A)^{x+\la}
(J+\la I)
(I+A)^{-x-\la}.
\end{aligned}
\end{equation}

The coefficients of $\eta$ and $\eta^{-1}$ in $2(\Delta S^{(\la)} - S^{(\la-1)} \Delta)$ are 
\begin{align*}
2\left(
\Phi^{(\la-1)\ast}(x+1)
- 
\Phi^{(\la)\ast}(x)
\right)
&= -a(I+A), \\ 
2\biggl( \Psi^{(\la)\ast}(x)
- \Phi^{(\la)\ast}(x)
-
\left( \Psi^{(\la-1)\ast}(x)
- \Phi^{(\la-1)\ast}(x)
\right)
\biggr)
&=
-x(I+A)^{-1},
\end{align*}
respectively, where we have used Lemma \ref{lem:immediate2} in the last step.
The above agrees with the corresponding $\eta^{\pm}$ coefficients
of $-\mathfrak{D}^{(\la)}$. For the $\eta^0$ term in $2(\Delta S^{(\la)} - S^{(\la-1)} \Delta)$, we need to consider
$$
4\Phi^{(\la)\ast}(x) - 2\Psi^{(\la)\ast}(x)
-
2\left(
\Phi^{(\la-1)\ast}(x+1) 
+
\Phi^{(\la-1)\ast}(x)
- \Psi^{(\la-1)\ast}(x+1)
\right),
$$
which can be regrouped into quantities we have calculated
above:
\begin{align*}
2\biggl(
&\Phi^{(\la)\ast}(x) 
- \Psi^{(\la)\ast}(x)
-\left(
\Phi^{(\la-1)\ast}(x)
- \Psi^{(\la-1)\ast}(x)
\right)\biggr)
\\
&\hspace{3cm}+
2\left(
\Phi^{(\la)\ast}(x) 
-
\Phi^{(\la-1)\ast}(x+1) 
\right)
+ 
2\left(\Psi^{(\la-1)\ast}(x+1)
- \Psi^{(\la-1)\ast}(x)\right)
\\
&=
x(I+A)^{-1}
+
a(I+A)
+2 \mathcal{J}_1^{(\la-1)\ast}
\\
&=
\mathfrak{J}^{(\la)}
+
(a-N-2\la)I.
\end{align*}
In the last line we have used \eqref{eq:Jfrak} and
that the leading coefficient of $\Psi^{(\la)}$ simplifies
to
$$
2\mathcal{J}_1^{(\la)}
=
J  - a A^\ast - (\la+1)A^\ast (I+A^\ast)^{-1} - N - 1 - \la.
$$
This completes the proof.
\end{proof}

\section{Explicit expression for the entries of the matrix Charlier polynomials}
\label{sec:7}
In Section \ref{subsec:second-order-diff-op} we found
that a  difference equation for $P_n^{(\la)}$ could
be significantly simplified by considering
the following modification of the polynomials
\begin{equation} 
\label{eq:scale}
    R_n^{(\la)}(x)
    =
    L_0(I+A)^{-n-\la}P_n^{(\la)}(x) (I+A)^{\la+x}.
\end{equation}
Recall from Proposition \ref{cor:entradas-R} that the entries of $R_n^{(\la)}(x)$ are given by
\begin{equation} \label{eq:decomp2}
R_n^{(\la)}(x)_{j,k} = \xi_{j,k,n}^{(\la)} c_{n+j-k}^{(a)}(x),
\end{equation}
with the $\xi_{j,k,n}^{(\la)}$ independent of $x$. The goal of this section is to give explicit expressions for
$\xi_{j,k,n}^{(\la)}$ which in turn gives us the explicit
entries of the monic orthogonal polynomials $P_n^{(\la)}$
due to Corollary \ref{cor:PnEntries}. The main ingredients in this calculation are two difference equations. The first one is the difference equation \eqref{eq:rec1}. The second one is $(P^{(\la)}\cdot \Delta S^{(\la)})_n = n G_n^{(\la)}P_n^{(\la)}$, the left hand side of which we have seen in \eqref{eq:twoDiffs}.

\begin{rmk}\label{rmk:set0}
It is important to emphasize that the $\xi_{j,k,n}^{(\la)}$ have
only been defined for $n+j-k \geq 0$, c.f. \eqref{eq:decomp}. 
It will be convenient to define $\xi_{j,k,n}^{(\la)}=0$ whenever
$j<1$, $k<1$, $j>N$, $k>N$ or $n+j-k<0$.
\end{rmk}

\subsection{A recurrence relation in $k$ for $\xi_{j,k,n}^{(\la)}$}
Firstly, we will write $P_n^{(\la)}\cdot \Delta S^{(\la)} = n G_n^{(\la)}P_n^{(\la)}$ in terms of $R_n^{(\la)}$.
\begin{prop}
\label{prop:difference_for_Rn}
The difference equation $(P^{(\la)}\cdot \Delta S^{(\la)})_n = n G_n^{(\la)}P_n^{(\la)}$ in terms of $R_n^{(\la)}$ from \eqref{eq:scale} is given by
\begin{equation*}
\begin{aligned}
    n\left( (N+\la+1)I-J\right) R_n^{(\la)}(x)
    &=
    -R_n^{(\la)}(x+1)a\left(J(I+A^*) + \la I\right)
    \\
    &\hspace{2cm} + R_n^{(\la)}(x)\left(a (I+A)(J(I+A^*)+\la I)
    +x(J+\la I)\right)
    \\
    &\hspace{5cm} - R_n^{(\la)}(x-1)
    x(I+A)(J+\la I).
\end{aligned}
\end{equation*}
\end{prop}
\begin{proof}
It follows directly from \eqref{eq:scale},  \eqref{eq:relacion-D}, \eqref{eq:twoDiffs} and Lemma \ref{lem:Gn-diagonalizable}.
\end{proof}

In the following proposition we show that the difference equation for $R_n^{(\lambda)}$ in Proposition \ref{prop:difference_for_Rn} induces a three term recurrence relation for $\xi_{j,k,n}^{(\la)}$.
\begin{prop}
\label{prop:recursion-k}
The coefficients $\phi_k \vcentcolon = \xi_{j,k,n}^{(\la)} $ satisfy the following recursion in $k$
\begin{multline*}
    0
    =
    (\sqrt{a}  \left(k+\la\right) \sqrt{N-k}) \, \phi_{k+1} 
    \\
    \qquad 
    +
    \left(
    j \la - n (N+1 - j) - N  -1
   + 
   k (N +  j - \la + n + 2 ) 
   - 
   2 k^2
    \right) \, \phi_k
        \\
    \qquad + \left(a^{-\frac12} (n + j - k + 1)(k-1)\sqrt{N-k+1}\right)\, \phi_{k-1} ,
\end{multline*}
which holds for $1 \leq k \leq K$ with $K\vcentcolon = \min(N,n+j)$.
\end{prop}
\begin{proof}
The main ingredient of the proof is the difference equation in Proposition \ref{prop:difference_for_Rn} but with $x=0.$ This reduces to 
\begin{equation}
\label{eq:deltas-R-0}
n\left( (N+\la+1)I-J\right) R_n^{(\la)}(0)=-a R_n^{(\la)}(1)\left(J(I+A^*) + \la I\right)
+ R_n^{(\la)}(0)a (I+A)(J(I+A^*)+\la I).
\end{equation}
From \eqref{eq:decomp2} we obtain
$R_n^{(\la)}(0)_{j,k} = \xi_{j,k,n}^{(\la)}$ and
$R_n^{(\la)}(1)_{j,k} = \xi_{j,k,n}^{(\la)} c_{n+j-k}^{(a)}(1) 
=\xi_{j,k,n}^{(\la)} (1-\frac{n+j-k}{a})$. For the other matrices involved in the above equation, recall that 
\begin{alignat*}{3}
A_{k+1,k} &=  \frac{1}{\sqrt{a}}\sqrt{N-k}, \qquad   & (AJ)_{k+1,k}  =   \frac{1}{\sqrt{a}} k \sqrt{N-k}, & \qquad & 1 \leq k \leq N-1, &
\\
(JA^\ast)_{k-1,k} &= \frac{1}{\sqrt{a}}(k-1)\sqrt{N-k+1}, \qquad  & (AJA^\ast)_{k,k}  = \frac{1}{a}(k-1)(N-k+1), & \qquad &2 \leq k \leq N, &
\end{alignat*}
where the expression for $AJA^\ast$ does also hold for $k=1$.
Now we use these to look at the $(j,k)$ entry of \eqref{eq:deltas-R-0} to arrive at the desired result.
\end{proof}

\subsection{A recurrence relation in $j$ for $\xi_{j,1,n}^{(\la)}$ }
The recurrence relation in Proposition \ref{prop:recursion-k} determines the coefficients $\xi_{j,k,n}^{(\la)}$ up to the initial values $\xi_{j,1,n}^{(\la)}$. In order to find a recurrence relation for these coefficients, we will write \eqref{eq:rec1} in terms of $R_n^{(\la)}$.
\begin{prop}
\label{prop:Rn-diff-j}
The difference equation \eqref{eq:rec1} in terms of $R_n^{(\la)}$ is given by
\label{prop:recR1}
\begin{equation*}
\begin{aligned}
R_n^{(\la)}(x+1) &
+
 \left( 
    \frac{1}{a^2}\D_n^{(\la)}(I+A^\ast)^{-1}\D_{n-1}^{(\la)-1}
     -\frac{x}{a}\D_n(I+A^\ast)^{-1}\D_n^{-1}
      -(I+A)
   \right)R_n^{(\la)}(x)
\\
&\hspace{5cm}+
\frac{x}{a} \D_n^{(\la)}(I+A^\ast)^{-1}\D_n^{(\la)-1}(I+A)
R_n^{(\la)}(x-1)
= 0.
\end{aligned}
\end{equation*}
\end{prop}
\begin{proof}
Recall the LDU decomposition from Corollary \ref{cor:descomp-LDU-norma} and further that $\mathscr{A}=I+A$.
The rest follows from direct computation.
\end{proof}

As in the previous subsection, in the following proposition we derive a three term recurrence relation for $\xi_{j,1,n}^{(\la)}$ from Proposition \ref{prop:recR1}.
\begin{prop}\label{prop:recursion-j}
The coefficients $\psi_j \vcentcolon = \xi_{j,1,n}^{(\la)} $ satisfy the following
recursion in $j$
\begin{equation*}
\begin{aligned}
    0 &= \psi_{j+1} \left( \frac{1}{\sqrt{a}} j(N+\la-j)(n+j)\sqrt{N-j}\right)
    \\
    &\qquad+ \psi_j \biggl((n+j-1)(N+n+\la-j)(N+n+\la-j+1)
    \\
    &\qquad \qquad- n(\la+n)(N+\la+n)
    + j(N-j)(N+\la-j) \biggr)
    \\
    &\qquad+ \psi_{j-1} \left(\sqrt{a}(N+n+\la-j)(N+n+\la-j+1)\sqrt{N-(j-1)} \right),
\end{aligned}
\end{equation*}
for $1\leq j \leq N$.
\end{prop}

\begin{proof}
The main ingredient of the proof is Proposition \ref{prop:Rn-diff-j}
but with $x=0$. This reduces to
\begin{equation}\label{eq:aunmasrec}
a\D_n^{(\la)}(I+A^\ast)\D_n^{(\la)-1} R_n^{(\la)}(1)
=
\left(
 a\D_n^{(\la)}(I+A^\ast)\D_n^{(\la)-1}(I+A)
-
\frac{1}{a} \D_n^{(\la)}\D_{n-1}^{(\la)-1}
\right)
R_n^{(\la)}(0).
\end{equation}
As before \eqref{eq:decomp2} gives us
$R_n^{(\la)}(0)_{jk} = \xi_{j,k,n}^{(\la)}$ and
$R_n^{(\la)}(1)_{jk} = \xi_{j,k,n}^{(\la)} c_{n+j-k}^{(a)}(1) 
=\xi_{j,k,n}^{(\la)} (1-\frac{n+j-k}{a})$. 
For the other
matrix entries involved in the above equation we look to \eqref{eq:parameters-sol} for $A_{j+1,j} = \frac{1}{\sqrt{a}}\sqrt{N-j}$ and
\eqref{eq:expresion-entrada-D} for
\begin{align*}
    \left(\D_n^{(\la)}\right)_{jj}\left(\D_{n-1}^{(\la)-1}\right)_{jj}
    &=
    \frac{n a (N+\la+n)(\la+n)}{(N+\la+n-j)(N+\la+n+1-j)},
    \\
    \left(\D_n^{(\la)}\right)_{jj}\left(A^\ast\right)_{j,j+1}\left(\D_{n-1}^{(\la)-1}\right)_{j+1,j+1}
    &=
\frac{1}{\sqrt{a}}\frac{j(N+\la-j)\sqrt{N-j}}{(N+\la+n-j)(N+\la+n+1-j)}. 
\end{align*}
Now we use these expressions to look at the
$(j,1)$ entries of \eqref{eq:aunmasrec}
and multiply everything by $(N+\la+n-j)(N+\la+n+1-j)$ to arrive at the
desired result.
\end{proof}

\subsection{The coefficients $\xi_{j,k,n}^{(\la)}$ as dual Hahn polynomials}
In the final step of the determination of the coefficients $\xi_{j,k,n}^{(\la)}$, we solve the recurrence relations of the previous subsections in terms of the dual Hahn polynomials.

The dual Hahn polynomials are given by
$$\mathcal{R}_{k}\bigl(\ell(x); \, \gamma , \delta, \mathcal{N}\bigr) = \rFs{3}{2}{-k, -x, x+\gamma+\delta+1}{\gamma+1, -\mathcal{N}}{1},\qquad k=0,1,\ldots, \mathcal{N},$$
see for instance \cite{KoekS}.
\begin{prop}\label{prop:xi_k}
Let $K=\min(N,n+j)$ and
$\mathscr{K}=\max(N,n+j)$.
Then
for $k\in \{ 1, \dots , K\}$ we have
$$
\xi_{j,k,n}^{(\la)}
=
\xi_{j,1,n}^{(\la)}
a^{-(k-1)/2}(1-K)_{k-1}
\sqrt{\frac{(N-k)!}{(N-1)!}}
\mathcal{R}_{k-1}\left((K-j)(\mathscr{K}-j+\la+1) ;\la,\mathscr{K}-K, K-1\right),
$$
and $0$ for $k > K $. 
\end{prop}
\begin{proof}
First observe that $\xi_{j,k,n,}^{(\la)} = 0$ when $k>K$,
because of Remark \ref{rmk:set0}. 
Next we take Proposition \ref{prop:recursion-k}
and substitute $\zeta_k = (\la+1)_k a^{k/2} \sqrt{\frac{(N-k-1)!}{(N-1)!}} \frac{\xi_{j,k+1,n}^{(\la)}}{\xi_{j,1,n}^{(\la)}}$,
to arrive at a recursion that is more reminiscent
of a three term recurrence of monic orthogonal
polynomials
\begin{equation*}
\begin{aligned}
    0
    &=
     \zeta_{k+1} 
    +
    \left(
    -2 k^2
    +
    k (j-\la+n+N-2)
    + 
    j(n+\la+1)
    -(nN+\la+1)
    \right) \zeta_k
        \\
    &\hspace{8cm} + k(k+\la)
    (k - (n + j ) )
    (k-N)  \zeta_{k-1},
\end{aligned}
\end{equation*}
for $0\leq k \leq K-1 $.
It is easy to check, comparing with
\cite{KoekS}, that this is the three term
recurrence for the monic dual Hahn with $\gamma=\la$,
$\delta=\mathscr{K}-K$, $\mathcal{N}=K-1$ where $(K-j)(\mathscr{K}-j+\la+1)$ plays
the role of the variable. Lastly the monic polynomials
are related to the standard ones by $p_{k-1} =(\la+1)_{k-1}(1-K)_{k-1} \mathcal{R}_{k-1}$.
\end{proof}

\begin{prop}\label{prop:jsolution}
For $j\in \{1, \dots , N \}$ we have
$$
\xi_{j,1,n}^{(\la)}
=
(-1)^{j-1}\xi_{1,1,n}^{(\la)}
\sqrt{\frac{(N-1)!}{(N-j)!}}
\frac{
a^{(j-1)/2}(1-N-\la -n)_{j-1}}{(j-1)!(1-N-\la)_{j-1}}.
$$
\end{prop}
\begin{proof}
We take Proposition \ref{prop:recursion-j}
and substitute $\sigma_j = \frac{j!(n+1)_j(1-N-\la)_j}{a^{j/2}(1-N-\la-n)_j} 
\sqrt{\frac{(N-1)!}{(N-j-1)!}}
\frac{\xi_{j+1,1,n}^{(\la)}}{\xi_{1,1,n}^{(\la)}}$.
Just like in the proof of Proposition \ref{prop:xi_k},
we arrive at a three term recurrence for monic dual Hahn
polynomials
\begin{equation*}
    0=
    \sigma_{j+1} + \left(-2j^2 +j(2N+\la-n-2) +n(N-1)+N-1\right) \sigma_j + j(j-(N+\la))(j-N)(j+n) \sigma_{j-1},
\end{equation*}
but with simpler parameters. So now we have $\sigma_j = p_{j}(0;n, \la, N-1) = (n+1)_j(1-N)_j$. Substituting back gives us the desired result.
\end{proof}

Now all that we are left to obtain is the initial value.
\begin{lem}\label{lem:initial}
The final and initial values are
$$
\xi_{N,1,n}^{(\la)}
=
(-1)^{n+N-1}
\frac{a^{n+(N-1)/2}}{\sqrt{(N-1)!}},
\qquad \qquad
\xi_{1,1,n}^{(\la)}
=
(-a)^n
\frac{ (\la +1)_{N-1} }{(\la + n )_{N-1}}.
$$
\end{lem}

\begin{proof}
We can write
$$
(I+A)^{x+\la}=\exp((x+\la)B), \qquad B \vcentcolon = \log(I+A) = \sum_{s=1}^{N-1} \frac{(-1)^{s+1}}{s}A^s.
$$
Then since $A^N=0$, we have $B^{N-1}=A^{N-1}$. 
This shows us that both are polynomial in $x$
and in particular we have a simple leading term
$$
(I+A)^{x+\la}= x^{N-1} \frac{A^{N-1}}{(N-1)!} +\mathcal{O}(x^{N-2}).
$$
The monicity of $P_n^{(\la)}$ and \eqref{eq:scale} give
$$
R_n^{(\la)}(x)
=
x^{n+N-1}
L_0 (I+A)^{-n-\la}
\frac{A^{N-1}}{(N-1)!} +\mathcal{O}(x^{n+N-2}).
$$
This leading term has only one non-zero entry because 
$L_0 (I+A)^{-n-\la} $ is
lower triangular and $(A^{N-1})_{jk} \neq 0$ only for $(j,k)=(N,1)$.
So then entrywise we get
$$
(R_n^{(\la)}(x))_{N,1} 
=
\xi_{N,1,n}^{(\la)}
c_{n+N-1}^{(a)}(x)
= 
x^{n+N-1}
\frac{(A^{N-1})_{N,1}}{(N-1)!} +\mathcal{O}(x^{n+N-2})
$$
which gives us the first result by comparing leading terms
$$
\xi_{N,1,n}^{(\la)}
=
(-a)^{n+N-1}
\frac{(A^{N-1})_{N,1}}{(N-1)!}
=
(-1)^{n+N-1}
\frac{a^{n+(N-1)/2}}{\sqrt{(N-1)!}}.
$$
We can then directly obtain $\xi_{1,1,n}^{(\la)}$
from Proposition \ref{prop:jsolution} with $j=N$.
\end{proof}

Finally, we summarize the content of Propositions \ref{prop:xi_k}, \ref{prop:jsolution} and Lemma \ref{lem:initial} in the main Theorem of this section.

\begin{thm}
\label{thm:expresion-xi}
The coefficients $\xi_{j,k,n}^{(\la)}$ in \eqref{eq:decomp2} are given in terms of the dual Hahn polynomials by 
\begin{equation*}
\begin{aligned}
\xi_{j,k,n}^{(\la)}
&=
\mathfrak{X}(j,k,n,\la) \times
\begin{cases}
(1-N)_{k-1}\,
\rFs{3}{2}{1-k,j-N,n+\la+1}{\la +1, 1-N}{1} & n + j \geq N
\\
(1-n-j)_{k-1}\,
\rFs{3}{2}{1-k,-n,N-j+\la+1}{\la +1, 1-n-j}{1} & n + j < N,
\end{cases}
\\
&=
\mathfrak{X}(j,k,n,\la) \times
\begin{cases}
(1-N)_{k-1}
\mathcal{R}_{k-1}\bigl(\ell(N-j);\, \la , n+j-N, N-1\bigr) & n + j \geq N,
\\
(1-n-j)_{k-1}
\mathcal{R}_{k-1}\bigl(\ell(n); \, \la , N-(n+j), n+j-1\bigr) & n + j < N,
\end{cases}
\end{aligned}
\end{equation*}
with
$$
\mathfrak{X}(j,k,n,\la)
=
\sqrt{\frac{(N-k)!}{(N-j)!}}
\frac{(-1)^{n+j-1}
a^{n+(j-k)/2}(\la+n)!(N+\la-j)!}{(j-1)!\la!(N+\la+n-j)!}.
$$
\end{thm}

\section{Dual Charlier Polynomials}
\label{sec:dualOP}
In this section we carry out the construction of three non--equivalent families of matrix valued polynomials in the sense of \eqref{eq:dual-equivalence-relation} which are dual to the matrix valued Charlier polynomials $P_n^{(\la)}$. It follows from Theorem \ref{thm:correspondenceDualalgebras} that each family is related to a suitable operator in the set $\mathcal{B}^2_R(P)$. 

\subsection{Explicit expression for $P_n(0)$}
The goal of this subsection is to give an explicit expression $P_n(0)$ which is a key ingredient of the dual families, see Theorem \ref{thm:correspondenceDualalgebras}. Firstly, we recall the matrix sequence $\mathcal{A}_\lambda$ introduced in \eqref{eq:defn-A-cal}. In the following remark we establish some basic properties for $\mathcal{A}_\lambda$.
\begin{rmk}
\label{rmk:propiedades-JAn}
As we have seen in Lemma \ref{lem:immediate2}, the matrices $J$ and $A$ have simple commutation relations. We can easily derive similar results for $\mathcal{A}_\la$
\begin{enumerate}
    \item $[J,\mathcal{A}_\la]=-\mathcal{A}_\la,$
    \item $[J, (I+\mathcal{A}_\la)^k]=-k\mathcal{A}_\la(I+\mathcal{A}_\la)^{k-1},$
    \item $(I+\mathcal{A}_\la)^{-k}J(I+\mathcal{A}_\la)^k=J-k\mathcal{A}_\la (I+\mathcal{A}_\la)^{k-1}.$
\end{enumerate}
\end{rmk}

\begin{prop} Let $R_n^{(\lambda)}$ be the function defined in \eqref{eq:scale}. Then the following relations hold:
\label{lem:relacion-R}
\begin{enumerate}
\item $R_n^{(\lambda)}(0)=-a(I+\mathcal{A}_{n+\la})R_{n-1}^{(\lambda+1)}(0)$.
\item $R_n^{(\lambda)}(0)=(-a)^n(I+\mathcal{A}_{n+\la})^n L_0.$
\end{enumerate}
\end{prop}
\begin{proof}
We relegate the proof of $(1)$ to Appendix \ref{app:R0}. Part $(2)$ follows by iterating part $(1)$ and applying the fact that $R_0^{(\lambda+n)}(0)=L_0.$
\end{proof}
\begin{cor}
\label{cor:Pn-en-cero}
For the monic Charlier polynomials we get that $$P_n^{(\la)}(0)=(-a)^nL_0^{-1}(I+A)^{n+\la}(I+\mathcal{A}_{n+\la})^n (I+A)^{-\la}L_0.$$
\end{cor}
\begin{proof}
The Corollary follows directly from the previous Proposition and equation \eqref{eq:scale} which defines 
$R_n^{(\la)}(x) = L_0(I+A)^{-n-\la}P_n^{(\la)}(x) (I+A)^{\la+x}$.
\end{proof}
\begin{rmk}
\label{rmk:I+An-racional}
Note that since $A^*$ is \emph{strictly upper triangular}, so is $\mathcal{A}_{n+\la}$. From the explicit expression of $\cA_{n+\la}$ in \eqref{eq:defn-A-cal}, we can verify that it has rational entries in $n$. Then $(I+\cA_{n+\la})$ is an upper triangular matrix with diagonal entries equal to $1$. This allows us to conclude that $(I+\cA_{n+\la})^n$ is rational in $n$. In particular, this implies that every entry of $(-a)^{-n}P_n^{(\lambda)}(0)$ is a rational function of $n$.
\end{rmk}

\subsection{Three families of dual polynomials}
In Theorem \ref{thm:correspondenceDualalgebras} we have established a one to one correspondence between the equivalence classes of families of dual polynomials and the subset of second order difference operators $\mathcal{B}^2_R(P)$, see Definition \ref{def:B2P}. In this section we give three families of dual polynomials, each one corresponding to an specific operator of $\mathcal{B}^2_R(P)$.

Recall from \eqref{eq:Dfrak} and \eqref{eq:twoDiffs} that
\begin{align}
\label{eq:Dfrak-8}
\mathfrak{D}^{(\la)} &= \eta a(I+A) -J - (x+\la)(I+A)^{-1} + \eta^{-1} x(I+A)^{-1}, \\
\label{eq:twoDiffs-1}
\Delta S^{(\la)}
&=
-
\eta \Phi^{(\la)}(x)^\ast
+ 
2\Phi^{(\la)}(x)^\ast -\Psi^{(\la)}(x)^\ast
+ \eta^{-1} \left( \Psi^{(\la)}(x)^\ast
- \Phi^{(\la)}(x)^\ast
\right),
\\
\label{eq:twoDiffs-2}
S^{(\la-1)} \Delta
&=
-
\eta \Phi^{(\la-1)}(x+1)^\ast
+ 
\Phi^{(\la-1)}(x+1)^\ast
+
\Phi^{(\la-1)}(x) ^\ast
-
\Psi^{(\la-1)}(x+1)^\ast
\\
\nonumber
&\hspace{7cm} +
\eta^{-1} \left( \Psi^{(\la-1)}(x)^\ast
- \Phi^{(\la-1)}(x)^\ast
\right),
\end{align}
where $\Phi^{(\la)\ast}$ is a matrix polynomial of degree two given explicitly in Theorem \ref{thm:PearsonCharlier}.
However in this section we will have more use for the 
decomposition given in \eqref{eq:phidecomp},
$$
\Phi^{(\la)}(x)^\ast=\frac{a}{2}(I+A)^{x+\la+1}(J(I+A^\ast)+\la I )(I+A)^{-x-\la}.
$$

We label these as $D_1=\mathfrak{D}^{(\la)}$,
$D_2 = \Delta S^{(\la)}$ and
$D_3 = S^{(\la-1)}\Delta$, to
simplify notation and to be able
to discuss all three at the same
time. 
\begin{lem}
The difference operators $D_1= \mathfrak{D}^{(\la)}, D_2=\Delta S^{(\la)}$ and $D_3=S^{(\la-1)}\Delta$ are elements of $\mathcal{B}^2_R(P)$.
For $D_1$ and $D_2$ this holds for
$\la\in \mathbb{N}_0$ and for $D_3$
for $\la\in\mathbb{N}$.
\end{lem}

\begin{proof}
In \eqref{eq:Dfrak-8} the raising coefficient $a(I+A)$ can immediately
be seen to be invertible. The same
follows for the raising coefficients
in \eqref{eq:twoDiffs-1} and \eqref{eq:twoDiffs-2} by
considering \eqref{eq:phidecomp},
although we must take $\la \geq 1$ for
$D_3$. 
So the invertibility condition is met for
all three cases. Is easy to check that in the three cases the lowering coefficient evaluated at $x=0$ is equal to zero, then the condition over $P_n(0)^{-1}P_n(-1)F_{-1}(0)$ is automatically satisfied. We just need to check that
the $\rho_i^{(\la)}(n)=P_n^{(\la)}(0)^{-1}\psi^{-1}(D_i)P_n^{(\la)}(0)$ satisfy Condition \ref{assumption:rho} for all $i=1,2,3.$ Recall from \eqref{eq:rec34} 
\begin{equation*}
(P_n^{(\lambda)} \cdot S^{(\lambda-1)} \Delta)(x)= (n+1) G_{n+1}^{(\lambda-1)} P_n(x), \qquad  (P_n^{(\lambda)} \cdot \Delta S^{(\lambda)})(x) = nG_n^{(\lambda)} P_n(x), 
\end{equation*}
where by Lemma \ref{lem:Gn-diagonalizable} and equation \eqref{eq:relacion-D} $$G_n^{(\la)}=\frac{1}{2}L_0^{-1}(I+A)^{n+\la}((N+\la+1)I-J)(I+A)^{-n-\la}L_0.$$
Additionally Theorem
\ref{thm:autovalor-Dfrak} gives us 
\begin{equation*}
P_n^{(\la)}\cdot \mathfrak{D}^{(\la)} 
=
\Gamma_n^{(\la)} P_n^{(\la)},
\qquad \Gamma_n^{(\la)} = a(I+A) -J -(n+\la)(I+A)^{-1}.
\end{equation*}
Applying Corollary \ref{cor:Pn-en-cero} and Remark \ref{rmk:propiedades-JAn} we get
\begin{align}
\label{eq:rho1}
\rho_1^{(\la)}(n)
& =
L_0^{-1}(I+A)^{\la}
\left((a-n-\la)I-J+n\mathcal{A}_{n+\la}(I+\mathcal{A}_{n+\la})^{n-1}\right)
(I+A)^{-\la}L_0, \\
\label{eq:rho2}
\rho_2^{(\la)}(n)
& =
\frac{n}{2}L_0^{-1}(I+A)^{\la}\left((N+\la+1)I-J+n\mathcal{A}_{n+\la}(I+\mathcal{A}_{n+\la})^{n-1}\right)(I+A)^{-\la}L_0, \\
\label{eq:rho3}
\rho_3^{(\la)}(n)
& =
\frac{n+1}{2}L_0^{-1}(I+A)^{\la}\left((N+\la)I-J+n\mathcal{A}_{n+\la}(I+\mathcal{A}_{n+\la})^{n-1}\right)(I+A)^{-\la}L_0.
\end{align}

From these expressions we can
see that each $\rho_i^{(\la)}$ is an $n$-dependent matrix of the form
$$
\rho_i^{(\la)}(n)
=
L_0^{-1}(I+A)^\la 
( n \mathcal{T} +  \mathcal{T}_0 +\mathrm{s.u.t.})
(I+A)^{-\la}L_0,
$$
where $\mathcal{T}$ and $\mathcal{T}_0$ are $n$-independent diagonal matrices, and $\textrm{s.u.t.}$ stands for \textit{strictly upper triangular}. In addition, Remark \ref{rmk:I+An-racional} tells us
that the
$\textrm{s.u.t.}$ part is rational in $n$.

That means that each $\rho_i^{(\la)}$
is of the form
described in Corollary \ref{cor:BVM} in
Appendix \ref{app:block-vandermonde}.
This then allows us to see that
the determinants of the block Vandermonde matrices in Condition \ref{assumption:rho} and Condition \ref{assumption:rho2} are
nonzero.
\end{proof}

As a consequence, by Theorem \ref{thm:correspondenceDualalgebras}, we can find a dual family of polynomials for each operator $ S^{(\la-1)}\Delta, \Delta S^{(\la)}, \text{ and } \mathfrak{D}^{(\la)}.$

\begin{rmk}
Given $D_1, D_2, D_3 \in \mathcal{B}_R^2(P)$, such that
$$P_n \cdot D_i=\Lambda_n^{(i)}P_n, \quad \text{for $i$ =1,2,3}$$
for some matrices $\Lambda_n^{(i)}$. Suppose that there exist nonzero complex numbers $\alpha, \beta, \gamma$ such that
\begin{equation*}
D_3=\alpha D_1 +\beta D_2 + \gamma I, \qquad \Lambda_n^{(3)}=\alpha \Lambda_n^{(1)} + \beta \Lambda_n^{(2)} + \gamma I.    
\end{equation*}
The corresponding $\rho_3(n)$ is then given by
$$\rho_3(n)=P_n(0)^{-1}\Lambda_n^{(3)}P_n(0)=\alpha\rho_1(n)+\beta\rho_2(n)+\gamma I.$$
In Theorem \ref{thm:darboux-transf-deltaS} 
we have
seen that this is indeed true in our case with $\alpha=\frac12$, $\beta=1$ and $\gamma=-\frac12(a-N-2\la)$.
This is consistent with \eqref{eq:rho3}.
\end{rmk}

We explicitly give three families of dual polynomials associated with the difference operators \eqref{eq:Dfrak-8}, \eqref{eq:twoDiffs-1} and \eqref{eq:twoDiffs-2}.

\subsubsection{Dual family associated to the operator $\mathfrak{D}^{(\la)}$}
\label{subsec:primera-familia}
We start by finding the dual family which corresponds to the operator $\mathfrak{D}^{(\la)}.$
Applying Theorem \ref{thm:correspondenceDualalgebras} we get that the dual polynomials are
\begin{equation}
\label{eq:dual-1}
Q_{x,1}^{(\la)}(\rho_1^{(\la)}(n))=P^{(\la)}_n(0)^{-1}P_n^{(\la)}(x)\Upsilon_1^{(\la)}(x)^{-1},
\end{equation}
where $P_n^{(\la)}(0)$ is given in Corollary \ref{cor:Pn-en-cero}, $\rho_1^{(\la)}$ is given in \eqref{eq:rho1} and
\begin{equation*}
\Upsilon_1^{(\la)}(x)=a^{-x}(I+A)^{-x},
\end{equation*}
which does not depend on $\la$ in this case.

\subsubsection{Dual family associated to the operator $\Delta S^{(\la)}$}
Similarly, using Theorem \ref{thm:correspondenceDualalgebras} we get
\begin{align*}
Q_{x,2}^{(\la)}(\rho_2^{(\la)}(n))&=P^{(\la)}_n(0)^{-1}P_n^{(\la)}(x)\Upsilon^{(\la)}_2(x)^{-1}, \\
\Upsilon^{(\la)}_2(x)&=\frac{(-2)^x}{a^x}(I+A)^{\la}(J(I+A^*)+\la I)^{-x}(I+A)^{-x-\la},
\end{align*}
where $\rho_2^{(\la)}$ is given in \eqref{eq:rho2}.

\subsubsection{Dual family associated to the operator $S^{(\la-1)}\Delta$}
As with the previous two families we use Theorem \ref{thm:correspondenceDualalgebras} to get the dual polynomials and
the related quantities, but in this case only for $\la >0$,
\begin{align*}
Q_{x,3}^{(\la)}(\rho_3^{(\la)}(n))& =P^{(\la)}_n(0)^{-1}P_n^{(\la)}(x)\Upsilon^{(\la)}_3(x)^{-1}, \\
\Upsilon^{(\la)}_3(x)& =\frac{(-2)^x}{a^x}(I+A)^{\la}(J(I+A^*)+(\la-1)I)^{-x}(I+A)^{-x-\la},
\end{align*}
where $\rho_3^{(\la)}$ is given in \eqref{eq:rho3}.

\subsection{Dual orthogonality relations}
We are now ready to identify the orthogonality relations for the dual polynomials $(Q^{(\la)}_{x,i})_x$. 

\begin{thm}
The dual weight matrix for the Charlier polynomials is 
\begin{equation}
\label{eq:dual-weight-charlier}
U^{(\la)}(n) 
=a^{2n}L_0^\ast(I+A^\ast)^{-\la}(I+(\mathcal{A}_{n+\la})^\ast)^n(\cD_n^{(\la)})^{-1}(I+\mathcal{A}_{n+\la})^n(I+A)^{-\la}L_0,
\end{equation}
and (for each $\rho_i$) it satisfies the hypothesis of Theorem  \ref{thm:dual-weight}.
\end{thm}
\begin{rmk}
From Remark \ref{rmk:I+An-racional} and the explicit expression of $\cD_n^{(\la)}$ (see Theorem \ref{thm:expresion-D}), there exists a function $F^{(\la)}$ such that $F^{(\la)}(n)_{i,j}$ is rational in $n$ for all $i,j$ and
$$
U^{(\la)}(n)=\frac{a^n}{n!}F^{(\la)}(n),
$$
and from Remark \ref{rmk:I+An-racional} and its explicit expression we get that $\rho_i^{(\la)}(n)^{k}$ is rational in $n$ for all fixed $k\in\mathbb{N}_0$.
\end{rmk}
\begin{proof}
Equation \eqref{eq:dual-weight-charlier} follows from Theorem \ref{thm:dual-weight}, Corollary \ref{cor:Pn-en-cero} and the LDU decomposition of $\cH_n^{(\la)}$ in Corollary \ref{cor:descomp-LDU-norma}. By Theorem \ref{thm:dual-weight} we just need to verify
 \begin{equation*}
 \sum_{n=0}^\infty \rho^{(\la)}_i(n)^k U^{(\la)}(n) < \infty, \qquad \forall k\in \mathbb{N}_0,
 \end{equation*}
 \begin{equation}
 \label{eq:decaimiento-H}
    \lim_{n\to \infty} F_1(\rho_i^{(\la)}(n))P_n(0)\cH_n^{-1} P_{n+1}(0)^\ast F_2(\rho_i^{(\la)}(n)) = 0,\qquad \text{for all }F_1, F_2\in M_N(\mathbb{C})[n]. 
 \end{equation}
For a fixed $k\in \N_0$ and $s,t \in \{1,\ldots ,N \}$, using Remark \ref{rmk:I+An-racional} we have
$$
\left(\sum_{n=0}^\infty \rho^{(\la)}_i(n)^k U^{(\la)}(n)\right)_{s,t}=\sum_{k=0}^\infty \frac{a^n}{n!}\left( \rho^{(\la)}_i(n)^k F^{(\la)}(n)\right)_{s,t}<\infty,
$$
since each entry of $\rho^{(\la)}_i(n)^k F^{(\la)}(n)$ is rational in $n$. 
On the other hand, for $F_1, F_2\in M_N(\mathbb{C})[n]$, we have that $F_1(\rho_i^{(\la)}(n))$ and $F_2(\rho_i^{(\la)}(n))$ are rational functions in $n$, and therefore using the explicit expressions of $P_n(0)$ and $\cH_n$ we have that
$$F_1(\rho_i^{(\la)}(n))P_n(0)\cH_n^{-1} P_{n+1}(0)^\ast F_2(\rho_i^{(\la)}(n))=\frac{a^n}{n!}R(n),\qquad \text{for all } F_1, F_2\in M_N(\mathbb{C})[n],$$
where $R(n)$ is a rational function in $n$. Equation \eqref{eq:decaimiento-H} then follows.
\end{proof}

So each dual family is orthogonal
with respect to their own distinct dual inner-product. For matrix polynomials $\mathscr{P}, \mathscr{Q}$ we have
$$
\left\langle \mathscr{P},\mathscr{Q} \right\rangle_{i,d}^{(\la)}
=
\sum_{n=0}^\infty
\mathscr{P}\left(\rho_i^{(\la)}(n)\right)^\ast
U^{(\la)}(n) \mathscr{Q}\left(\rho_i^{(\la)}(n)\right).
$$
Note the weight
$U^{(\la)}$
is the same for each family, but
the inner-product requires us to compose with $\rho_i^{(\la)}$ which is different for each $i$.

\subsection{Dual shift operators} In this subsection we show that the ladder relations for the polynomials $P_n^{(\lambda)}$ associated to the Charlier weight induce shift operators for the dual polynomials $Q_x^{(\lambda)}$. We recall the following relations from Proposition \ref{prop:ladder}:
\begin{alignat}{2}
\label{eq:ecuacion-M}
& P^{(\la)}_n\cdot D 
= 
M \cdot P^{(\la)}_n,
\qquad
&&
M 
= 
(I+A) + \frac{1}{a} \cH^{(\la)}_n (I+A^*)^{-1}\cH_{n-1}^{(\la)-1} \delta^{-1},
\\
\label{eq:ecuacion-Mdagger}
\qquad
& P_n^{(\la)}\cdot D^\dagger 
= 
M^\dagger \cdot P_n^{(\la)},
\qquad
&&
M^\dagger 
= 
\frac{1}{a}(I+A)^{-1} \delta + \cH_n^{(\la)} (I+A^\ast) \cH_n^{(\la)-1},
\end{alignat} 
where $D$ and $D^{\dagger}$ are  
$$
D = \eta \, (I+A), \qquad 
D^\dagger = \eta^{-1}\frac{x}{a}(I+A)^{-1}.
$$
Since the operators $D$ and $D^\dagger$ are associated to the weak Pearson equation \ref{eq:DDdaggerweak}, it follows from Remark \ref{rmk:weakPearson-hatF} that $D, D^\dagger \in \hat{\mathcal{F}}_R(P^{(\lambda)})$. Each dual family has two relevant isomorphisms  
\begin{alignat*}{5}
	\tau\colon \hat{\mathcal{F}}_L(P^{(\lambda)})&\rightarrow {}&& \hat{\mathcal{F}}^d_L(Q^{(\lambda)}), & \hspace{1.5cm} & \sigma_i \colon &\hat{\mathcal{F}}_R(P^{(\lambda)})&\rightarrow {} && \hat{\mathcal{F}}^d_R(Q^{(\lambda)}),\\
							  M&\mapsto     {}&& P_n^{(\la)}(0)^{-1}MP_n^{(\la)}(0),  		 & \hspace{1.5cm} &              &              D & \mapsto    {} && \Upsilon_i^{(\la)}(x) D\Upsilon_i^{(\la)}(x)^{-1},
\end{alignat*}
see \eqref{eq:definition-sigma} and \eqref{eq:definition-tau}. Note that by definition, $\tau$ is a conjugation by $P_n^{(\lambda)}(0)$ and is therefore independent of the family of dual polynomials. In this case we omit the subindex $i$. On the other hand, for each $i=1,2,3$ we have the isomorphism $\sigma_i$ which involves conjugation by $\Upsilon_i^{(\la)}$ and is different for each family of dual polynomials. 
\begin{thm}
Let $P_n^{(\lambda)}$ be the monic matrix Charlier polynomials and  Let $D, D^\dagger, M, M^\dagger$ be given by \eqref{eq:ecuacion-M} and \eqref{eq:ecuacion-Mdagger}. We fix a family of dual polynomials $Q_{x}^{(\la)}=Q_{x,i}^{(\la)}$ for $i=1,2,3$. Then the following relations hold true:
\begin{align*}
    \tau(M)\cdot Q_x^{(\la)}(\rho_i^{(\la)}(n)) &= Q_{x+1}^{(\lambda)}(\rho_i^{(\la)}(n))\, \Upsilon^{(\lambda)}_i(1)(I+A),\\
    \tau(M^\dagger)\cdot Q_x^{(\lambda)}(\rho_i^{(\la)}(n)) &= Q_{x-1}^{(\lambda)}(\rho_i^{(\la)}(n)) \,  \frac{x}{a}(I+A)^{-1}\Upsilon^{(\lambda)}_i(1)^{-1}.
\end{align*}
\end{thm}
\begin{proof}
From Theorem \ref{thm:FourierAlgebrasDiagram} we get
\begin{align*}
    \tau(M)\cdot Q_x^{(\lambda)}(\rho_i^{(\la)}(n)) = Q_{x}^{(\lambda)}(\rho_i^{(\la)}(n))\cdot \sigma_i(D), \qquad  \tau(M^\dagger)\cdot Q_x^{(\lambda)}(\rho_i^{(\la)}(n)) = Q_{x}^{(\lambda)}(\rho_i^{(\la)}(n))\cdot \sigma_i(D^\dagger).
\end{align*}
Let us fix $i\in\{1,2,3\}$ and suppose that $D_i=\eta F_1(x) + F_0(x) + \eta^{-1} F_{-1}(x)$, so that $\Upsilon_i(x) 
= F_1(0)^{-1} \cdots F_1(x-1)^{-1}$ as in Theorem \ref{thm:correspondenceDualalgebras}. Then we have that
$$\sigma_i(D) = \eta \Upsilon_i(x+1)(I+A)\Upsilon_i(x)^{-1},\qquad \sigma_i(D^\dagger) = \eta^{-1} \frac{x}{a}\Upsilon_i(x-1)(I+A)^{-1}\Upsilon_i(x)^{-1}.$$
Using that $F_1(x)=a(I+A)$ for $i=1$, $F_1(x)=-\Phi^{(\la)\ast}(x)$ for $i=2$ and 
$F_1(x) =- \Phi^{(\la-1)\ast}(x+1)$ for $i=3$ and using the explicit expression in \eqref{eq:phidecomp} we obtain
$$\Upsilon_i(x+1)(I+A)\Upsilon_i(x)^{-1} = \Upsilon_i(1)(I+A),
\qquad \Upsilon_i(x-1)(I+A)^{-1}\Upsilon_i(x)^{-1} = (I+A)^{-1}\Upsilon_i(1)^{-1}.$$
This completes the proof of the theorem.
\end{proof}
\begin{rmk}
Applying the isomorphism $\tau,$ the LDU decomposition of the norm (see Corollary \ref{cor:descomp-LDU-norma}), and properties of the matrix $J$ (see Lemma \ref{lem:immediate2}), we get
\begin{equation*}
\begin{aligned}
\tau(M) &=
L_0^{-1}(I+A)^\la
(I+\mathcal{A}_{n+\la})^{-n}
\left((I+A)-\frac{1}{a^2}\cD_n^{(\la)}(\cD_{n-1}^{(\la)})^{-1}\delta^{-1}\right)
(I+\mathcal{A}_{n+\la})^n(I+A)^{-\la}L_0,
\\
\tau(M^\dagger) &=L_0^{-1}(I+A)^\la(I+\mathcal{A}_{n+\la})^{-n}
\left(\cD_n^{(\la)}(I+A^*)(\cD_n^{(\la)})^{-1}-\delta\right)
(I+\mathcal{A}_{n+\la})^n(I+A)^{-\la}L_0.
\end{aligned}
\end{equation*}
These expressions are independent of the choice of the dual family.
\end{rmk}

\begin{rmk}
For the dual family that corresponds
to $D_1$ we have the simple expressions:
\begin{equation*}
\sigma_1(D)=\eta\frac{1}{a}, \quad \sigma_1(D^\dagger)=\eta^{-1}x, \quad
\sigma_1(\mathfrak{J}^{(\la)})
=
J + xI + \la (I+A)^{-1}.
\end{equation*}
\end{rmk}

\begin{rmk}
For the dual family corresponding to $D_2$ we have
\begin{align*}
\sigma_2(D)&=-\eta 2 a^{-1} (I+A)^\la(J(I+A^*)+\la I)^{-1}(I+A)^{-\la}, \\
\sigma_2(D^\dagger)&=-\eta^{-1}\frac{x}{2}(I+A)^\la(J(I+A^*)+\la I)(I+A)^{-\la}, \\
\sigma_2(\mathfrak{J}^{(\la)})&=(I+A)^\la(J(I+A^*)+\la I)^{-x}(J+(x+\la)I)(J(I+A^*)+\la I)^{x}(I+A)^{-\la}.
\end{align*}
For the dual family corresponding to $D_3$ we have
\begin{align*}
\sigma_3(D)&=-\eta2a^{-1}(I+A)^\la(J(I+A^*)+(\la-1) I)^{-1}(I+A)^{-\la}, \\
\sigma_3(D^\dagger)&=-\eta^{-1}\frac{x}{2}(I+A)^\la(J(I+A^*)+(\la-1) I)(I+A)^{-\la}, \\
\sigma_3(\mathfrak{J}^{(\la)})&=(I+A)^\la(J(I+A^*)+(\la-1) I)^{-x}(J+(x+\la)I)(J(I+A^*)+(\la-1) I)^{x}(I+A)^{-\la}.
\end{align*}
We note that expressions of $\sigma_i(D), \sigma_i(D^\dagger), \sigma_i(\mathfrak{J}^{(\la)})$ for $i=1,2$ are very similar. This is due to the fact that the coefficients of the operators $\Delta S^{(\lambda)}$ and $S^{(\lambda-1)}\Delta$ are related by simple shifts in $\lambda$ and $x$, see \eqref{eq:twoDiffs-1} and \eqref{eq:twoDiffs-2}.
\end{rmk}

\begin{rmk}
We finally include the expressions for the duals of the following relevant operators which appeared in \eqref{eq:3TROP} and Remark \ref{rmk:psiJ}.
\begin{align*}
\tau(\mathcal{L})
&=
P_n^{(\la)}(0)^{-1}P_{n+1}^{(\la)}(0)
\delta
+
P_n^{(\la)}(0)^{-1} B_n^{(\la)}
P_n^{(\la)}(0)
+
P_n^{(\la)}(0)^{-1} C_n^{(\la)}
P_{n-1}^{(\la)}(0)\delta^{-1},
\end{align*}
\begin{align*}
\tau(\psi^{-1}(\mathfrak{J}^{(\la)}))
&=
a \tau(M) + a \tau(M^\dagger) - \rho_1^{(\la)}(n).
\end{align*}
For the last expression we have
used that $\mathfrak{J}^{(\la)}=aD+aD^\dagger - \mathfrak{D}^{(\la)}$ which
follows from the definition of
$\mathfrak{D}^{(\la)}$ in \eqref{eq:Dfrak}.
\end{rmk}

\subsection{Dual square norm}
In Theorem \ref{thm:dualnorm1} we
proved an expression for
the square norms of the dual
MVOP \textit{assuming} 
that certain
functions are in the
closure of the span of the MVOP $(P_n)_n$.
Since we haven't proven this
assumption 
for our special case of
Charlier type MVOP, we will derive the square norms
in a more direct fashion that will ultimately agree with the earlier result.

It is more efficient to write the following proofs 
in terms of the 
$$
R_n^{(\la)}(x) 
= L_0(I+A)^{-n-\la}P_n^{(\la)}(x)(I+A)^{x+\la},
$$
rather than the $Q_x^{(\la)}$.
Using \eqref{eq:scale} and the fact that
for $i \in \{1,2,3\}$ we have
$$
Q_{x,i}^{(\la)}(\rho_i^{(\la)}(n))
=
P_n^{(\la)}(0)^{-1}P_n^{(\la)}(x)\Upsilon_i^{(\la)}(x)^{-1},
$$
we get that the square norm of the $Q_{x,i}^{(\la)}$ is given by
\begin{equation}\label{eq:dualnormdecomp}
\sW_{x,i}^{(\la)} = \left( (I+A^\ast)^{x+\la}\Upsilon_i^{(\la)\ast}(x)\right)^{-1} \sW_R^{(\la)}(x)
\left(\Upsilon_i^{(\la)}(x)(I+A)^{x+\la}\right)^{-1},
\end{equation}
where
$$
\sW_R^{(\la)}(x)
=
\sum_{n=0}^\infty R_n^{(\la)}(x)^\ast \D_n^{(\la)-1} R_n^{(\la)}(x).
$$

\begin{lem}
\label{lem:zeromom}
The 0-th moment of the dual weight $U^{(\la)}$
is given by
$$
\mathscr{W}_0^{(\la)}
=
W^{(\la)}(0)^{-1}.
$$
\end{lem}
\begin{proof}
The proof is given in Appendix \ref{app:zeromom}.
\end{proof}

\begin{prop}
The norm of the dual MVOP $Q_{x,i}^{(\la)}$
is given by
$$
\sW_{x,i}^{(\la)}
=
(\Upsilon_i^{(\la)}(x)
W^{(\la)}(x)
\Upsilon_i^{(\la)}(x)^\ast)^{-1}.
$$
\end{prop}
\begin{proof}
We can derive a recursion for 
$\sW_R^{(\la)}$ by rewriting
the shift equations in
\eqref{eq:shifts-for-Charlier} for
$R_n^{(\la)}$.
First $P_n^{(\la)}\cdot\Delta = nP_{n-1}^{(\la+1)}$ is swiftly rewritten as
$$
R_n^{(\la)}(x+1)
-
R_n^{(\la)}(x)(I+A)
=
n R_{n-1}^{(\la+1)}(x),
$$
and for $P_{n-1}^{(\la+1)}\cdot S^{(\la)} = G_n^{(\la)}P_n^{(\la)}$
we need
$$
L_0(I+A)^{-n-\la}G_n^{(\la)}(I+A)^{n+\la}L_0^{-1}
=
n\D_{n-1}^{(\la+1)}(\D_n^{(\la)})^{-1},
$$
and
$$
(I+A)^{-x-\la-1}\Phi^{(\la)\ast}(x)(I+A)^{x+\la}
=
T^{(\la+1)}(I+A^\ast)(T^{(\la)})^{-1},
$$
$$
(I+A)^{-x-\la}
(
\Phi^{(\la)\ast}(x)
-
\Psi^{(\la)\ast}(x))
(I+A)^{x+\la}
=
\frac{x}{a}
T^{(\la+1)}(T^{(\la)})^{-1},
$$
to get to
$$
-R_{n-1}^{(\la+1)}(x)T^{(\la+1)}(I+A^\ast)(T^{(\la)})^{-1}
+
R_{n-1}^{(\la+1)}(x-1)
\frac{x}{a}
T^{(\la+1)}(T^{(\la)})^{-1}
=
n\D_{n-1}^{(\la+1)}(\D_n^{(\la)})^{-1}
R_n^{(\la)}(x).
$$
We now take $x\geq 1$, since
the $x=0$ was covered in the previous result. We can then use the first of the shifts to get
$$
\mathscr{W}_R^{(\la)}(x)
=
(I+A^\ast)\sum_{n=0}^\infty R_n^{(\la)}(x-1)^\ast \D_n^{(\la)-1}R_n^{(\la)}(x)
+
\sum_{n=1}^\infty n R_{n-1}^{(\la+1)}(x-1)^\ast \D_n^{(\la)-1}R_n^{(\la)}(x).
$$
The first of the two sums is proportional to $\langle Q_{x-1}^{(\la)},Q_{x}^{(\la)}\rangle_d^{(\la)}$ which is zero due to the \textit{dual} orthogonality.
Inserting $\D_{n-1}^{(\la+1)-1}\D_{n-1}^{(\la+1)}$ we get
$$
\mathscr{W}_R^{(\la)}(x)
=
\sum_{n=1}^\infty R_{n-1}^{(\la+1)}(x-1)^\ast \D_{n-1}^{(\la+1)-1} n\D_{n-1}^{(\la+1)} \D_n^{(\la)-1}R_n^{(\la)}(x),
$$
which allows us to apply the second shift
$$
\sW_R^{(\la)}(x)
=
-
\sum_{n=1}^\infty R_{n-1}^{(\la+1)}(x-1)^\ast \D_{n-1}^{(\la+1)-1} 
R_{n-1}^{(\la+1)}(x)T^{(\la+1)}(I+A^\ast)(T^{(\la)})^{-1}
$$
$$
\hspace{5cm}
+
\sum_{n=1}^\infty R_{n-1}^{(\la+1)}(x-1)^\ast \D_{n-1}^{(\la+1)-1}
R_{n-1}^{(\la+1)}(x-1)
\frac{x}{a}
T^{(\la+1)}(T^{(\la)})^{-1}.
$$
Here again the first sum vanishes due to the dual orthogonality, since
it is proportional to $\langle Q_{x-1}^{(\la+1)},Q_{x}^{(\la+1)}\rangle_d^{(\la+1)}$. So we then arrive at
$$
\mathscr{W}_R^{(\la)}(x) = \mathscr{W}_R^{(\la+1)}(x-1)\frac{x}{a}
T^{(\la+1)}(T^{(\la)})^{-1},
$$
which after iteration gives $\mathscr{W}_R^{(\la)}(x)=\sW_R^{(\la+x)}(0) \frac{x!}{a^x}T^{(\la+x)}T^{(\la)-1} = \frac{x!}{a^x}T^{(\la)-1}$. Putting this result back 
into \eqref{eq:dualnormdecomp},
together with Lemma
\ref{lem:zeromom}, gives
the desired result.
\end{proof}

\subsection{Lie Algebras of difference operators associated to the Charlier weight}
As an aside, we describe Lie algebras generated by the operators of interest in this paper.
We denote by $\mathfrak{g}$ the linear span of $\{D, D^\dagger,  \mathfrak{J^{(\la)}}, x, I\}$, where $I$ is the identity matrix.
\begin{prop}
\label{lem:tau-sigma-corchete}
The linear space $\mathfrak{g}$ is a Lie algebra that contains the operator $\mathfrak{D}^{(\lambda)}$. More precisely, the following relations hold true
$$
[\mathfrak{J}^{(\la)},D]
=
D, 
\qquad 
[\mathfrak{J}^{(\la)},D^{\dagger}]
=
-D^{\dagger},
\qquad
[D,x]=-D,
\qquad
[D^\dagger,x]=D^\dagger,
\qquad
[D,D^\dagger]
=
-\frac{1}{a}I,
$$
where $[,]$ is the standard bracket $[E_1,E_2] = E_1E_2-E_2E_1$.
\end{prop}
\begin{proof}
The commutations relations in the lemma follow directly from the definitions of $\mathfrak{J}^{(\la)},$ in \eqref{eq:Jfrak} and $D, D^{\dagger}$ in \eqref{eq:newD}. This shows that $\mathfrak{g}$ is closed under the bracket operation. Finally it follows directly from \eqref{eq:Dfrak} that $\mathfrak{D}^{(\lambda)}\in \mathfrak{g}$.
\end{proof}
We observe that since each of the generators of $\mathfrak{g}$ is related to a weak Pearson equation, we have that $\mathfrak{g}\subset \hat{\mathcal{F}}_R(P)$. Therefore, we can transport this Lie algebra to $\hat{\mathcal{F}}_L(P), \hat{\mathcal{F}}^d_R(Q), \hat{\mathcal{F}}^d_L(Q) $ via the isomorphisms $\psi, \sigma, \tau$ and we obtain the diagram:
$$
\begin{tikzcd}
	\psi^{-1}(\mathfrak{g}) \arrow[r, "\psi"] \arrow[d, "\tau"]
	& \mathfrak{g} \arrow[d, "\sigma"] \\
	\psi^d(\sigma(\mathfrak{g}))
	& \arrow[l, "\psi^d"] \sigma(\mathfrak{g})
\end{tikzcd}
$$

The four dimensional Lie algebra $\mathfrak{g}_1$ generated by $\{D, D^\dagger,  \mathfrak{J^{(\la)}}, I\}$ is isomorphic to the Lie algebra $\mathcal{G}(0,1)$ introduced by W. Miller in \cite[chapter 2-5]{miller_1} via the identifications:
$$
J^{+}\rightarrow D,
\qquad 
J^{-} \rightarrow D^{\dagger}, 
\qquad 
J^3\rightarrow \mathfrak{J}^{(\la)}, 
\qquad \mathcal{E}\rightarrow \frac{1}{a}I.
$$
The corresponding Casimir operator is given by
$$
a C_{0,1}
=
a \left(J^{+}J^{-}-\mathcal{E}J^3
\right)
\rightarrow
a D D^{\dagger} 
- 
\mathfrak{J}^{(\la)}
= 
x - \mathfrak{J}^{(\la)},
$$
and commutes with every element in $\mathfrak{g}_1$. We observe that the Casimir operator is indeed an element of the center $\mathfrak{Z}(\mathfrak{g})$ of $\mathfrak{g}$, which is the two dimensional abelian subalgebra generated by $\{C_{0,1}, I \}$.

We observe that $C_{0,1}\in \hat{\mathcal{F}}_R(P)$ and therefore
$\psi^{-1}(C_{0,1})$ is 
a nontrivial second order operator in the bispectral algebra $\mathcal{B}_L(P)$ which commutes with the ladder operators $M, M^\dagger$ and the eigenvalue $\Gamma_n^{(\lambda)}$ of $\mathfrak{D}^{(\lambda)}$.

\begin{prop}
Let $\mathfrak{k}$ be the four dimensional Lie algebra generated by $\{ D, D^{\dagger}, x, I \}.$ Then $\mathfrak{k}$ is isomorphic with $\mathcal{G}(0,1)$ via the identifications
$$
J^{+}\rightarrow D, 
\qquad 
J^{-}\rightarrow D^{\dagger}, 
\qquad 
J^3\rightarrow x, 
\qquad 
\mathcal{E} \rightarrow \frac{1}{a}I.
$$
Moreover, under this identification we have have
$$\mathfrak{g}=\mathcal{G}(0,1) \oplus \mathbb{C} \, C_{0,1},$$
\end{prop}
\begin{proof}
The isomorphism between $\mathfrak{k}$ and $\mathcal{G}(0,1)$ follows directly from the relations in Proposition \ref{lem:tau-sigma-corchete}.
Since $C_{0,1} \notin \mathfrak{k}$, we immediately have that $\mathfrak{g}= \mathfrak{k}\oplus \mathbb{C} \, C_{0,1}$ as vectors spaces. Now we only need to show that both summands are ideals in $\mathfrak{g}$. Proposition \ref{lem:tau-sigma-corchete} gives $[\mathfrak{k},\mathfrak{g}]\subset \mathfrak{k}$ and $\mathbb{C} \, C_{0,1}$ is an ideal since $C_{0,1}\in \mathfrak{z}(\mathfrak{g})$. This completes the proof of the proposition.
\end{proof}

\begin{rmk}
Let $\mathfrak{h}$ be the three dimensional subspace generated by $\{D,D^\dagger, I\}$, then $\mathfrak{h}$ is isomorphic with the the Heisenberg Lie algebra of dimension 3, see for instance \cite[p. 57]{Faraut}.
\end{rmk}

\begin{rmk}
The quotient Lie algebra $\mathfrak{g}/\mathfrak{z}(\mathfrak{g})$ is generated by the elements:
$$e_1= D^\dagger +\mathfrak{z}(\mathfrak{g}), \qquad e_2= D +\mathfrak{z}(\mathfrak{g}), \qquad e_3 = x+ \mathfrak{z}(\mathfrak{g}),$$
with commutation relations
$$[e_1,e_2] = 0,\qquad [e_1,e_3]=e_1,\qquad [e_2,e_3]=-e_2.$$
Therefore $\mathfrak{g}/\mathfrak{z}(\mathfrak{g})$ is a three dimensional Lie algebra which is isomorphic with the Lie algebra $\mathfrak{p}(1,1)$ of the Poincaré group $P(1,1)$. This is the group of affine transformations of $\mathbb{R}^2$ which preserve the Lorentz metric, see \cite[\S 2.5.9]{Hall}.
\end{rmk}

\section{Dual--Dual polynomials}
\label{sec:dual-dual}
In this final section we will we investigate 4--tuples
$(\widetilde{P}_n^{(\la)}, \widetilde M_1, \widetilde M_2, \widetilde \nu^{(\la)})$ that are dual to the dual 4--tuple $(Q_{x,1}^{(\la)},P_n^{(\la)}(0),\Upsilon_1^{(\la)}(x),\rho_1^{(\la)})$ described in Section \ref{subsec:primera-familia}. These families will be called \emph{dual--dual polynomials}. The goal of this section is to indicate the steps in the construction of dual--dual families, but we will not describe in detail some sets that are analogues of those discussed in Section \ref{sec:duality} (like $\mathcal{B}_L^{d,2}(Q)$). More precisely, we will discuss
two distinct  dual-dual families $\widetilde{P}_n^{(\la)}$
: one will be trivial, giving us back $P_n^{(\la)}$,
and the second will not be.

The dual--dual polynomials will be a sequence of matrix polynomials $(\widetilde P_n^{(\la)})_n$ with a matrix argument $\mathcal{X}$ but on
the opposite side compared to the $Q_x^{(\la)}$
$$
\widetilde{P}_n^{(\la)}(\mathcal{X})
=
A_{n,n} \mathcal{X}^n  + A_{n,n-1}\mathcal{X}^{n-1} +\ldots + A_{n,0},
$$
see \eqref{eq:Qpoly} for the case corresponding to the $Q^{(\lambda)}_x$. To find such a dual--dual family
boils down to finding a sequence of polynomials $(P_n^{(\lambda)})_n$ together with matrix functions $\widetilde{\nu}^{(\la)}(x)$, $\widetilde{M}_1^{(\la)}(n)$, $\widetilde{M}_2^{(\la)}(x)$ 
such that
\begin{equation}
\label{eq:dual-dual}
Q_x^{(\la)}(\rho^{(\la)}(n))
=
\widetilde{M}_1^{(\la)}(n)
\widetilde{P}_n^{(\la)}(\widetilde{\nu}^{(\la)}(x))
\widetilde{M}_2^{(\la)}(x),
\qquad
\forall x,n \in \mathbb{N}_0,
\end{equation}
and such that the polynomials $\widetilde{P}_n^{(\la)}$ satisfy the monic three term recurrence
\begin{equation}
\label{eq:dualdual}
\widetilde{P}_n^{(\la)}(\widetilde{\nu}^{(\la)}(x))\widetilde{\nu}^{(\la)}(x)
=
\widetilde{P}_{n+1}^{(\la)}(\widetilde{\nu}^{(\la)}(x))
+
\widetilde{B}_n^{(\la)} \widetilde{P}_n^{(\la)}(\widetilde{\nu}^{(\la)}(x))
+
\widetilde{C}_n^{(\la)} \widetilde{P}_{n-1}^{(\la)}(\widetilde{\nu}^{(\la)}(x)).
\end{equation}
As in Section \ref{sec:duality}, a dual--dual family is determined by a 4--tuple $(\widetilde{P}^{(\lambda)}_n, \widetilde{M}_1^{(\la)}, \widetilde{M}_2^{(\la)}, \widetilde \nu^{(\lambda)})$ and we will 
also restrict the superfluous
degrees of freedom in a similar
fashion by setting $\widetilde{M}_1(0)=\widetilde{M}_2(0)=I$. This immediately leads
to $\widetilde{M}_2(x)= Q_x^{(\la)}(\rho^{(\la)}(0))=\Upsilon^{(\la)}(x)^{-1}$ and $\widetilde{M}_1(n) = ( \widetilde{P}_n^{(\la)}( \widetilde{\nu}^{(\la)}(0)))^{-1}$.

\subsection{First family of dual--dual polynomials: $P_n^{(\lambda)}$ as its own dual--dual}
The dual condition \eqref{eq:dual-dual} implies trivially that 
$$(\widetilde{P}^{(\lambda)}_n, \widetilde{M}_1^{(\la)}, \widetilde{M}_2^{(\la)}, \widetilde \nu^{(\lambda)}) = (P^{(\lambda)}_n, P_n^{(\lambda)}(0)^{-1}, (I+A)^x a^x, x)$$ 
is a dual family for $(Q^{(\lambda)}_x, P_n^{(\lambda)}(0), (I+A)^{-x} a^{-x}, \rho_1^{(\lambda)}(n))$, i.e. 
\begin{equation*}
Q_x^{(\la)}(\rho_1^{(\la)}(n) )
=
P^{(\lambda)}_n(0)^{-1}
P_n^{(\la)}(x) (I+A)^x a^x, \qquad
\forall x,n \in \mathbb{N}_0.
\end{equation*}
We observe that since $\widetilde \nu^{(\lambda)} = x$, the recurrence relation for $P_n^{(\lambda)}$ in \eqref{eq:dualdual} is precisely the three term recurrence relation \eqref{eq:3TR}.

As in Lemma \ref{lem:dual-to-bispectral}, this dual--dual family is related to a difference operator, namely $\tau(\mathcal{L})\in \hat{\mathcal{F}}_L^d(Q)$, where  $\mathcal{L}=\psi^{-1}(x)$ is defined in \eqref{eq:3TROP}. Explicitly we have
$$
\tau(\mathcal{L})
=
P_n^{(\la)}(0)^{-1} P_{n+1}^{(\la)}(0)
\delta
+
P_n^{(\la)}(0)^{-1} B_n^{(\la)}P_n^{(\la)}(0)
+
P_n^{(\la)}(0)^{-1} C_n^{(\la)}
P_{n-1}^{(\la)}(0)\delta^{-1}.
$$
We note that Corollary \ref{cor:Pn-en-cero} implies that the coefficient $P_n^{(\la)}(0)^{-1} P_{n+1}^{(\la)}(0)$ of $\delta$ in $\tau(\mathcal{L})$ is a rational function of $n$ which is invertible for all $n\in \N_0$. On the other hand,  the explicit expression in Theorem \ref{thm:coef-3term}  implies that $C_n$ is a rational function of $n$ which vanishes at $n=0$. Therefore, the coefficient of $\delta^{-1}$ in $\tau(\mathcal{L})$ is a rational function of $n$ and vanishes at $n=0$.
These conditions are analogous to those in Definition \ref{def:B2P}. For this dual-dual family, the commutative diagram in Theorem \ref{thm:FourierAlgebrasDiagram} is extended trivially in the following way:
\begin{equation*}
\label{eq:cd-dual-dual}
\begin{tikzcd}
	\hat{\mathcal{F}}_L(P_n^{(\lambda)})
	\arrow[r,"\tau"] \arrow[d, "\psi"]
	&
	\hat{\mathcal{F}}^d_L(Q_n^{(\lambda)}) 
	\arrow[r,"\tau^{-1}"]
	&
	\hat{\mathcal{F}}_L(P_n^{(\lambda)}) 
    \arrow[d,"\psi"]
    \\
	\hat{\mathcal{F}}_R(P) 
	\arrow[r, "\sigma_1"] 
	&  \hat{\mathcal{F}}^d_R(Q_n^{(\lambda)})
	\arrow[u, "\psi^d"]
	\arrow[r, "\sigma_1^{-1}"] 
	&  \hat{\mathcal{F}}_R(P_n^{(\lambda)})
\end{tikzcd}
\end{equation*}
In conclusion $P_n^{(\la)}$ is its own dual--dual.

\subsection{Second family of dual--dual polynomials}

In order to construct a nontrivial family of dual--dual polynomials, we look for second order difference operators in $\mathcal{B}_L^d(Q)$. There are two natural candidates for these operators, namely $\tau(\psi^{-1}(\mathfrak{J}^{(\la)}))$ and the Casimir operator $\tau(\psi^{-1}(C_{0,1}))$. Next we follow the construction for the first operator.  We have already seen in Remark \ref{rmk:psiJ} that
\begin{equation*}
\psi^{-1}(\mathfrak{J}^{(\la)})\cdot P_n^{(\la)} = P_n^{(\la)}\cdot \mathfrak{J}^{(\la)},
\end{equation*}
so that $\mathfrak{J}^{(\la)}\in\hat{\mathcal{F}}_R(P)$ and $\psi^{-1}(\mathfrak{J}^{(\la)})\in\hat{\mathcal{F}}_L(P)$.
Let us write out
$$
\psi^{-1}(\mathfrak{J}^{(\la)})
=
G_1(n)\delta + G_0(n) + G_{-1}(n)\delta^{-1},
$$
where the explicit expressions of the coefficients $G_1, G_0, G_{-1}$ are given in Remark \ref{rmk:psiJ}.
On the other hand, $\sigma(\mathfrak{J}^{(\la)}) = \Upsilon^{(\la)}(x)\mathfrak{J}^{(\la)}\Upsilon^{(\la)}(x)^{-1}$ and
$$
\tau\bigl( \psi^{-1}(\mathfrak{J}^{(\la)})\bigr)
=
P_n^{(\la)}(0)^{-1}G_1(n) P_{n+1}^{(\la)}(0)\delta 
+ P_n^{(\la)}(0)^{-1}G_0(n) P_n^{(\la)}(0)
+ P_n^{(\la)}(0)^{-1}G_{-1}(n) P_{n-1}^{(\la)}(0)\delta^{-1}.
$$
It follows directly from Remark \ref{rmk:psiJ} that the coefficient $P_n^{(\la)}(0)^{-1}G_1(n) P_{n+1}^{(\la)}(0)$ is is a rational function which is invertible for all $n\in \N_0$. Moreover, by Corollary \ref{cor:descomp-LDU-norma} we have that 
$$G_{-1}(n) = L_0^{-1} (I+A)^{n+\lambda}\D_n^{(\lambda)}(\D_{n-1}^{(\lambda)})^{-1} (I+A)^{-n-\lambda+1}L_0.$$
Therefore \eqref{eq:expresion-entrada-D} implies that $G_{-1}(n)$ is a rational function of $n$ which vanishes at $n=0$. 

\begin{thm}
\label{thm:dualdual2df}
Let $\widetilde P^{(\lambda)}_n$ be the sequence of monic polynomials defined by the recurrence relation
\begin{multline*}
x \widetilde P^{(\lambda)}_n(x) = \widetilde P^{(\lambda)}_{n+1}(x) \\ +
(J + n + a + \la (I+A)^{-1} + aL_0^{-1}(I+A)^\la \D_n^{(\la)}A^\ast \D_n^{(\la)-1}(I+A)^{-\la}L_0) \widetilde P^{(\lambda)}_n(x) \\
+
L_0^{-1}(I+A)^\la \D_n^{(\la)}\D_{n-1}^{(\la)-1}
(I+A)^{-\la}L_0 \widetilde P^{(\lambda)}_{n-1}(x).
\end{multline*}
Then $(\widetilde P^{(\lambda)}_n, P^{(\lambda)}_n(0)^{-1} (I+A)^{n}, \Upsilon^{(\lambda)}(x)^{-1}, \mathfrak{J}^{(\la)}(x))$ is dual to $Q_x^{(\lambda)}$ and the following duality condition holds true
$$Q^{(\lambda)}_x(\rho(n)) = P^{(\lambda)}_n(0)^{-1} (I+A)^{n} \widetilde P^{(\lambda)}_n(\mathfrak{J}^{(\la)}(x))) \Upsilon^{(\lambda)}(x)^{-1}, \qquad n,x\in \N_0.$$
\end{thm}
\begin{proof}
If we replace the dual condition in the relation $\tau\bigl( \psi^{-1}(\mathfrak{J}^{(\la)})\bigr)\cdot Q_x^{(\lambda)} = Q_x^{(\lambda)} \cdot \sigma(\mathfrak{J}^{(\la)})$, we obtain the recurrence relation of theorem with $x$ replaced by $\mathfrak{J}^{(\la)}(x)$. Since $\mathfrak{J}^{(\la)}$ satisfies the block Vandermonde condition with $n$ replaced by $x$, an argument similar to the one in \eqref{thm:correspondenceDualalgebras} completes the proof of the theorem.
\end{proof}

\begin{cor}
The dual-dual family $(\widetilde P^{(\lambda)}_n, P^{(\lambda)}_n(0)^{-1} (I+A)^{n}, \Upsilon^{(\lambda)}(x)^{-1}, \mathfrak{J}^{(\la)}(x))$ is related to the original monic polynomials $P_n^{(\lambda)}$ in the following way
$$
P_n^{(\la)}(x) 
=
(I+A)^n \widetilde{P}_n^{(\la)}(\mathfrak{J}^{(\la)}(x)).
$$
\end{cor}
\begin{proof}
Using the dual relation for $Q_x^{(\lambda)}$ and $\widetilde P_n^{(\lambda)}$ in Theorem \ref{thm:dualdual2df} as well
as the dual relation for $P_n^{(\lambda)}$ and $Q_x^{(\lambda)}$  in \eqref{eq:dual-1}
we get the result.
\end{proof}

\appendix

\section{Miscellaneous Proofs}

\subsection{Block Vandermonde Determinant}
\label{app:block-vandermonde}
This subsection will mostly involve upper triangular
matrices and we will focus our attention on
their diagonal parts. So when two upper triangular matrices
$\mathcal{M}_1, \mathcal{M}_2$ have the same
diagonal, we will write
$$
\mathcal{M}_1 = \mathcal{M}_2
+ \textrm{s.u.t.}
$$
where $\textrm{s.u.t.}$ stands for \textit{strictly upper triangular}.
All the matrices $\rho_i^{(\la)}(n)$ in Section \ref{sec:dualOP} satisfy the conditions
in Corollary \ref{cor:BVM}, but we will
collect a preliminary result for a
slightly simpler case first.

\begin{lem}\label{lem:BVM}
Consider the $n$-dependent upper triangular matrix of the following form
\begin{equation}\label{eq:SUT}
\mathcal{M}(n) =  n \mathcal{T} +  \mathcal{T}_0 +\mathrm{s.u.t.}
\end{equation}
where $\mathcal{T}$ and $\mathcal{T}_0$ are constant diagonal matrices, and the
$\textrm{s.u.t.}$ part is allowed to depend on $n$.
The determinant of its block Vandermonde 
matrix does not depend on the $\textrm{s.u.t.}$ part
or $\mathcal{T}_0$,
and is given by
$$
\det
\begin{pmatrix}
I & \mathcal{M}(n_0) & \dots & \mathcal{M}(n_0)^x \\
I & \mathcal{M}(n_1) & \dots & \mathcal{M}(n_1)^x \\
\vdots & \vdots & \, & \vdots
\\
I & \mathcal{M}(n_x) & \dots & \mathcal{M}(n_x)^x 
\end{pmatrix}
=
\det(\mathcal{T})^{\frac12 x(x+1)}
\left(\prod_{0\leq s < t \leq x} (n_t-n_s) \right)^N,
\qquad x \in \mathbb{N}_{>0},
$$
where $(n_j)_{j=0}^x$ is a list of
complex values for which $\mathcal{M}(n_j)$
is defined.
\end{lem}
\begin{proof}
This proof will be by induction on $x$.
On occasion we will denote the block Vandermonde as $\det \left(\mathcal{M}(n_j)^k \right)_{j,k=0}^x$
even though this a slight abuse of notation
in case $\mathcal{M}(n_j)$ is singular. 
We apply a block analogue of an elementary
row operation to get
\begin{align*}
&
\det
\begin{pmatrix}
I & \mathcal{M}(n_0) & \dots & \mathcal{M}(n_0)^x \\
I & \mathcal{M}(n_1) & \dots & \mathcal{M}(n_1)^x \\
\vdots & \vdots & \, & \vdots
\\
I & \mathcal{M}(n_x) & \dots & \mathcal{M}(n_x)^x 
\end{pmatrix}
\\
&=
\det
\begin{pmatrix}
I & \mathcal{M}(n_0) & \dots & \mathcal{M}(n_0)^x \\
I & \mathcal{M}(n_1) & \dots & \mathcal{M}(n_1)^x \\
\vdots & \vdots & \, & \vdots
\\
I & \mathcal{M}(n_x) & \dots & \mathcal{M}(n_x)^x 
\end{pmatrix}
\det
\begingroup 
\setlength\arraycolsep{1.2pt}
\begin{pmatrix*}
I & -\mathcal{M}(n_0) & 0 &  \dots & 0 \\
0 & I & -\mathcal{M}(n_0)  & \dots & 0 \\
\vdots  & \ddots & \ddots & \ddots & \vdots
\\
0 & \dots & 0 & I  & -\mathcal{M}(n_0) 
\\
0 & \dots & \dots & 0  & I 
\end{pmatrix*}
\endgroup
\\
&=
\det
\begin{pmatrix}
I & 0  &\dots & 0 \\
I & (n_1 - n_0)\mathcal{T} + \textrm{s.u.t.}  &\dots & \mathcal{M}(n_1)^{x-1}(n_1 - n_0)\mathcal{T}  + \textrm{s.u.t.}  \\
\vdots & \vdots & \, & \vdots
\\
I & (n_x - n_0)\mathcal{T} + \textrm{s.u.t.} & \dots & \mathcal{M}(n_x)^{x-1}(n_x - n_0)\mathcal{T}  + \textrm{s.u.t.}
\end{pmatrix}.
\end{align*}
In the last step we have used that
$\mathcal{M}(n_j) -  \mathcal{M}(n_0)
=
(n_j-n_0)\mathcal{T} + \textrm{s.u.t}$.
To proceed we notice that 
$
\mathcal{M}(n_j)^k
- \mathcal{M}(n_j)^{k-1}\mathcal{M}(n_0)
=
(n_j-n_0)\mathcal{T}\left(n_j\mathcal{T}+\mathcal{T}_0\right)^{k-1}
+\textrm{s.u.t.}$
and reduce the size of the determinent using
the Schur complement
\begin{align*}
&=
\det
\begin{pmatrix}
 (n_1 - n_0)\mathcal{T} + \textrm{s.u.t.}  &\dots & (n_1 - n_0)\mathcal{T}(n_1 \mathcal{T}+\mathcal{T}_0)^{x-1}  + \textrm{s.u.t.}  \\
 \vdots & \, & \vdots
\\
 (n_x - n_0)\mathcal{T} + \textrm{s.u.t.} & \dots & (n_x - n_0)\mathcal{T}(n_x \mathcal{T}+\mathcal{T}_0)^{x-1}  + \textrm{s.u.t.}
\end{pmatrix}
\\
&=\det\left(
\textrm{diag}((n_j-n_0)\mathcal{T} + \textrm{s.u.t.})_{j=1}^x
\right)
\det
\begin{pmatrix}
 I & n_1\mathcal{T} + \textrm{s.u.t.}  &\dots & (n_1\mathcal{T}+\mathcal{T}_0)^{x-1}  + \textrm{s.u.t.}  \\
 \vdots & \vdots & \, &\vdots
\\
 I & n_x\mathcal{T} + \textrm{s.u.t.}  &\dots & (n_x\mathcal{T}+\mathcal{T}_0)^{x-1}  + \textrm{s.u.t.}
\end{pmatrix}
\end{align*}
The block diagonal matrix in the second
line is meant to have \textit{exactly}
the entries of the first column of 
the matrix in the previous
line; in particular it has the same s.u.t. parts. 
We need this so that the first
column of the second determinant is 
exactly just identity matrices without 
any $\textrm{s.u.t.}$ part. 

So we have a reduction of 
the block Vandermonde to one of 
a smaller size
\begin{equation}\label{eq:induction_rec}
\det \left(\mathcal{M}(n_j)^k \right)_{j,k=0}^x
=
\det(\mathcal{T})^x \left(\prod_{j=1}^x (n_j-n_0) \right)^N
\det
\left(\mathcal{N}(n_{j+1})^k \right)_{j,k=0}^{x-1},
\end{equation}
with
$
\mathcal{N}(n) = n\mathcal{T}+\mathcal{T}_0 +\textrm{s.u.t.}$ So $\mathcal{N}(n)$ has
the same diagonal entries as $\mathcal{M}(n)$.
Then since the factor in front of the
block Vandermonde determinant with $\mathcal{N}(n)$ in
\eqref{eq:induction_rec}, does not depend
on the specifics of 
the $\textrm{s.u.t.}$ we can 
continue reducing the size of the block Vandermonde until we get the desired result.
\end{proof}

\begin{cor}\label{cor:BVM}
The determinant formula in Lemma \ref{lem:BVM}
still holds if we conjugate $\mathcal{M}(n)$
in \eqref{eq:SUT} by an invertible matrix 
$\mathcal{Q}$ that does
not depend on $n$.
\end{cor}
\begin{proof}
Let $\mathcal{R}(n)=\mathcal{Q}\mathcal{M}(n)\mathcal{Q}^{-1}$. Then it is easy to see that
$$
\det
\begin{pmatrix}
I & \mathcal{R}(n_0) & \dots & \mathcal{R}(n_0)^x \\
I & \mathcal{R}(n_1) & \dots & \mathcal{R}(n_1)^x \\
\vdots & \vdots & \, & \vdots
\\
I & \mathcal{R}(n_x) & \dots & \mathcal{R}(n_x)^x 
\end{pmatrix}
=
\frac{\det\left(
\textrm{diag}(\mathcal{Q})_{j=0}^x
\right)}{\det\left(
\textrm{diag}(\mathcal{Q})_{j=0}^x
\right)}
\det
\begin{pmatrix}
I & \mathcal{M}(n_0) & \dots & \mathcal{M}(n_0)^x \\
I & \mathcal{M}(n_1) & \dots & \mathcal{M}(n_1)^x \\
\vdots & \vdots & \, & \vdots
\\
I & \mathcal{M}(n_x) & \dots & \mathcal{M}(n_x)^x 
\end{pmatrix}
$$

\end{proof}

\subsection{Proof of Proposition \ref{lem:relacion-R} }
\label{app:R0}
In this section we prove the first part of Proposition \ref{lem:relacion-R}
which states that
\begin{equation}\label{eq:R0rec}
R_n^{(\la)}(0)
=
-a(I+\mathcal{A}_{n+\la})R_{n-1}^{(\la+1)}(0),
\end{equation}
using the matrix introduced
in \eqref{eq:defn-A-cal} to express
$
\mathcal{A}_{n+\la}
=
(N+\la+n-J)^{-1} J A^\ast .
$
\begin{proof}
Recall that the entries $R_n^{(\la)}(0)_{jk}=\xi_{j,k,n}^{(\la)}$
are given explicitly in Theorem \ref{thm:expresion-xi} and
that $A_{j+1,j} =\frac{\sqrt{N-j}}{\sqrt{a}}$.
So we start off by writing \eqref{eq:R0rec}
in terms of its entries
$$
\xi_{j,k,n}^{(\la)}
=
-a \xi_{j,k,n-1}^{(\la+1)}
- \frac{j \sqrt{a}\sqrt{N-j}}{(N+\la+n-j)}
\xi_{j+1,k,n-1}^{(\la+1)}.
$$

Note that since the $\xi_{j,k,n}^{(\la)}$
are given by two different expressions
(for $n+j\geq N$ and for $n+j<N$),
we should prove the desired entrywise
recursion for both cases. But it is 
\textit{not} necessary to prove a mixed
case because the two cases actually coincide
for $n+j=N$, so for any values of the parameters
the three terms can always be considered as belonging
to the same case.

Before specifying to either case, we compute
the following expressions which will come in handy
later on in the proof
\begin{align*}
    \frac{\mathfrak{X}(j,k,n-1,\la+1)}{\mathfrak{X}(j,k,n,\la)}
    &=
    -\frac{N+\la+1-j}{a(\la+1)},
    \qquad
    \frac{\mathfrak{X}(j+1,k,n-1,\la+1)}{\mathfrak{X}(j,k,n,\la)}
    =
    \frac{\sqrt{N-j}}{\sqrt{a}}\frac{N+\la+n-j}{j(\la+1)}.
\end{align*}

To derive the $n+j>N$ case we will need
a fairly standard hypergeometric identity that is easy
to check in general (for parameters with
which the series either converge or truncate)
$$
\mathfrak{d}
\left(
\rFs{3}{2}{\mathfrak{a},\mathfrak{b},\mathfrak{c}}{\mathfrak{d}, \mathfrak{e}}{1}
-
\rFs{3}{2}{\mathfrak{a},\mathfrak{b},\mathfrak{c}}{\mathfrak{d}+1, \mathfrak{e}}{1}
\right)
=
\mathfrak{b}
\left(
\rFs{3}{2}{\mathfrak{a},\mathfrak{b}+1,\mathfrak{c}}{\mathfrak{d}+1, \mathfrak{e}}{1}
-
\rFs{3}{2}{\mathfrak{a},\mathfrak{b},\mathfrak{c}}{\mathfrak{d}+1, \mathfrak{e}}{1}
\right),
$$
It follows from both sides being equal to
$$
\sum_{s=1} \frac{(\mathfrak{a})_s(\mathfrak{b})_s (\mathfrak{c})_s s}{s! (\mathfrak{d}+1)_s (\mathfrak{e})_s}.
$$
When we fill in the 
values $\mathfrak{a}=1-k$, $\mathfrak{b}=j-N$,
$\mathfrak{c}=n+\la+1$, $\mathfrak{d}=\la+1$
and $\mathfrak{e} = 1-N$ and collect
some similar terms we get
\begin{align*}
   \rFs{3}{2}{1-k,j-N,n+\la+1}{\la+1, 1-N}{1}
   &=
   \frac{N+\la+1-j}{\la+1}
   \rFs{3}{2}{1-k,j-N,n+\la+1}{\la+2, 1-N}{1}
   \\
   &\qquad  -
   \frac{N-j}{\la+1}
   \rFs{3}{2}{1-k,j-N+1,n+\la+1}{\la+2, 1-N}{1}.
\end{align*}
After multiplying this equation by
$\mathfrak{X}(j,k,n,\la)(1-N)_{k-1}$
we get the desired result for $n+j>N$ by using
the ratios of different $\mathfrak{X}$ expressed earlier in
this proof.

The $n+j \leq N$ case is slightly
different. Note that now instead of the
factor of $(1-N)_{k-1}$ we have $(1-n-j)_{k-1}$
so it is no longer a common factor
in the three terms of our desired result.
The necessary hypergeometric identity is somewhat
less standard,
$$
\rFs{3}{2}{\mathfrak{a},\mathfrak{b},\mathfrak{c}}{\mathfrak{d}, \mathfrak{e}}{1}
=
\frac{\mathfrak{c}}{\mathfrak{d}}\frac{\mathfrak{e}-\mathfrak{a}}{\mathfrak{e}}
\rFs{3}{2}{\mathfrak{a},\mathfrak{b}+1,\mathfrak{c}+1}{\mathfrak{d}+1, \mathfrak{e}+1}{1}
-
\frac{\mathfrak{c}-\mathfrak{d}}{\mathfrak{d}}
\rFs{3}{2}{\mathfrak{a},\mathfrak{b}+1,\mathfrak{c}}{\mathfrak{d}+1, \mathfrak{e}}{1},
$$
but it can be shown to hold using the methods described
in \cite{EbisuIwasaki}
for the cases where the series either converge or
truncate. We then put in the values
$\mathfrak{a}=1-k$, $\mathfrak{b}=-n$, $\mathfrak{c}=N-j+\la+1$, $\mathfrak{d}=\la+1$
and $\mathfrak{e}=1-n-j$, to get
\begin{align*}
   \rFs{3}{2}{1-k,-n,N-j+\la+1}{\la+1, 1-n-j}{1}
   &=
   \frac{N+\la+1-j}{\la+1}
   \frac{n+j-k}{n+j-1}
   \rFs{3}{2}{1-k,-n+1,N-j+\la+2}{\la+2, 2-n-j}{1}
   \\
   &\qquad  -
   \frac{N-j}{\la+1}
   \rFs{3}{2}{1-k,-n+1,N-j+\la+1}{\la+2, 1-n-j}{1}.
\end{align*}
After multiplying this equation by
$\mathfrak{X}(j,k,n,\la)(1-n-j)_{k-1}$
we get the desired result for $n+j \leq N$ by using
the ratios of different $\mathfrak{X}$ expressed earlier in
this proof.

\end{proof}

\subsection{Proof of Lemma \ref{lem:zeromom}}
\label{app:zeromom}
We want to find 
$\mathscr{W}_R^{(\la)}(0)$ 
so that we will have the zero
moment courtesy of \eqref{eq:dualnormdecomp}.
For this
we need $(\D_n^{(\la)})_{jj}$
and $ R_n^{(\la)}(0)_{j,k} = \xi_{j,k,n}^{(\la)}$.
The entrywise expression is
$$
\left(\mathscr{W}_R^{(\la)}(0)\right)_{j,k}
=
\sum_{n=0}^\infty
\sum_{r=1}^N
\frac{ \xi_{r,j,n}^{(\la)} \xi_{r,k,n}^{(\la)}  }{\left(\D_n^{(\la)}\right)_{rr}}.	
$$
Since this expression is clearly symmetric, we will take $j\geq k$ from
now on.
It will prove  useful to have
the standard weight of the scalar dual Hahn
polynomials nearby c.f. \cite{KoekS},

$$
w_{dH}(x;\gamma,\delta,\mathcal{N})
=
\frac{(2x+\gamma+\delta+1)(\gamma+x)!(x+\gamma+\delta)!\delta!(\mathcal{N}!)^2}{x!\gamma!(\delta+x)!(\mathcal{N}-x)!(x+\gamma+\delta+\mathcal{N}+1)!}.
$$
Let us introduce the following shorthand to distinguish
between two different expressions for $\xi$ that hold
for different parameter values,
$$
\xi_{r,k,n}^{(\la)}
=
\mathfrak{X}(r,k,n,\la)\times
\begin{cases}
\alpha_{r,k,n}^{(\la)} & n+r \geq N,
\\
\beta_{r,k,n}^{(\la)} & k \leq n+r < N,
\\
0 & n+r < k.
\end{cases}
$$
The
explicit expressions for 
$\alpha$ and $\beta$ can
be found in Theorem
\ref{thm:expresion-xi}.
We then write out the double sum, which we note will not have
any cross terms
\begin{equation*}
\begin{aligned}
\left(\mathscr{W}_R^{(\la)}(0)\right)_{j,k}
&=
\sum_{n=0}^\infty
\sum_{r=r_0}^N
\frac{\mathfrak{X}(r,j,n,\la)
\mathfrak{X}(r,k,n,\la)}
{\left(\D_n^{(\la)}\right)_{rr}}
 \alpha_{r,j,n}^{(\la)} \alpha_{r,k,n}^{(\la)}  
\\ 
 &\hspace{4cm} +
 \sum_{n=0}^{N-1}
\sum_{r=r_1}^{N-n-1}
\frac{\mathfrak{X}(r,j,n,\la)
\mathfrak{X}(r,k,n,\la)}
{\left(\D_n^{(\la)}\right)_{rr}}
 \beta_{r,j,n}^{(\la)} \beta_{r,k,n}^{(\la)}.
\end{aligned}
\end{equation*}
We denote $r_0 = \max(1,N-n)$ and $r_1 = \max(1,j-n)$
and whenever a sum is over the empty set, we take it to be zero.
It will come in handy to introduce another summation variable
$s=n+r$ and rearrange the summations
\begin{equation}\label{eq:2sums}
\begin{aligned}
\left(\mathscr{W}_R^{(\la)}(0)\right)_{j,k}
&=
\sum_{s=N}^\infty
\sum_{r=1}^N
\frac{\mathfrak{X}(r,j,s-r,\la)
\mathfrak{X}(r,k,s-r,\la)}
{\left(\D_{s-r}^{(\la)}\right)_{rr}}
 \alpha_{r,j,s-r}^{(\la)} \alpha_{r,k,s-r}^{(\la)}  
\\ 
 &\qquad +
 \sum_{s=j}^{N-1}
\sum_{n=0}^{s-1}
\frac{\mathfrak{X}(s-n,j,n,\la)
\mathfrak{X}(s-n,k,n,\la)}
{\left(\D_n^{(\la)}\right)_{s-n,s-n}}
 \beta_{s-n,j,n}^{(\la)} \beta_{s-n,k,n}^{(\la)}.
\end{aligned}
\end{equation}
The common part of the summands is (note that the rightmost factor
is \textit{not} a factorial)
\begin{equation*}
\frac{\mathfrak{X}(r,j,n,\la)
\mathfrak{X}(r,k,n,\la)}
{\left(\D_n^{(\la)}\right)_{rr}}
=
\frac{2^\la \sqrt{(N-j)!(N-k)!}(\la+n)!}{e^{a} a^{\la+\frac12(j+k)}n!(\la !)^2}
\frac{a^{r+n} (N+\la-r)!(N+1-r+n+\la)}{(r-1)!(N-r)!(N+n+\la)!},
\end{equation*}
where we have used the expression of the inverse
of the diagonal part of the square norm in terms of
factorials
$$
\left(\D_n^{(\la)}\right)_{rr}^{-1}
= e^{-a} 
\left( \frac{2}{a}\right)^\la 
\frac{(r-1)!}{n! a^n} 
\frac{(N-r+n+\la)!(N+1-r+n+\la)!}{(\la+n)!(N-r+\la)!(N+n+\la)!}.
$$
The distinct parts of the summands are
\begin{align*}
\alpha_{r,j,s-r}^{(\la)}
\alpha_{r,k,s-r}^{(\la)}
&=
(-1)^{j+k}\frac{(N-1)!^2}{(N-k)!(N-j)!}
\\
&\qquad \times
\mathcal{R}_{j-1}\bigl(\ell(N-r); \, \la , s-N, N-1\bigr)
\mathcal{R}_{k-1}\bigl(\ell(N-r); \, \la , s-N, N-1\bigr)
\\
\beta_{s-n,j,n}^{(\la)}
\beta_{s-n,k,n}^{(\la)}
&=
(-1)^{j+k}\frac{(s-1)!^2}{(s-k)!(s-j)!}
\\
&\qquad \times
\mathcal{R}_{j-1}\bigl(\ell(n); \, \la , N-s, s-1\bigr)
\mathcal{R}_{k-1}\bigl(\ell(n); \, \la , N-s, s-1\bigr).
\end{align*}
Fortunately we can find a corresponding dual Hahn weight
in the common parts as follows
\begin{align*}
\frac{\mathfrak{X}(r,j,s-r,\la)
\mathfrak{X}(r,k,s-r,\la)}
{\left(\D_{s-r}^{(\la)}\right)_{rr}}
&=
\frac{2^\la \sqrt{(N-j)!(N-k)!}}{e^a a^{\la-s+\frac12(j+k)}}
\frac{w_{dH}(N-r;\la,s-N,N-1)}{\la!(s-N)!((N-1)!)^2},
\end{align*}
and similarly in
\begin{align*}
\frac{\mathfrak{X}(s-n,j,n,\la)
\mathfrak{X}(s-n,k,n,\la)}
{\left(\D_n^{(\la)}\right)_{s-n,s-n}}
&=
\frac{  2^\la \sqrt{(N-j)!(N-k)!} }{e^a a^{\la -s+\frac12(j+k)} }
\frac{ w_{dH}(n;\la,N-s,s-1) }{\la!(N-s)!((s-1)!)^2}
.
\end{align*}

The first sum in \eqref{eq:2sums}
is then
\begin{align*}
&
\sum_{s=N}^\infty
\sum_{r=1}^N
\frac{\mathfrak{X}(r,j,s-r,\la)
\mathfrak{X}(r,k,s-r,\la)}
{\left(\D_{s-r}^{(\la)}\right)_{rr}}
 \alpha_{r,j,s-r}^{(\la)} \alpha_{r,k,s-r}^{(\la)}  
 \\
&\qquad =
\frac{e^{-a}2^\la  (-1)^{j+k} a^{-\la-\frac12(j+k)}}{\la!\sqrt{(N-j)!(N-k)!}}
\sum_{s=N}^\infty
\frac{a^s}{(s-N)!}
\sum_{r=1}^N
w_{dH}(N-r;\la,s-N,N-1)
\\
&\qquad \qquad \times
\mathcal{R}_{j-1}\bigl(\ell(N-r); \, \la , s-N, N-1\bigr)
\mathcal{R}_{k-1}\bigl(\ell(N-r); \, \la , s-N, N-1\bigr),
\end{align*}
which vanishes when $j\neq k$ due to the orthogonality
of the dual Hahn polynomials. When we do have $j=k$ we get
\begin{align*}
&\frac{e^{-a}2^\la  a^{-\la-j}}{\la!(N-j)!}
\sum_{s=N}^\infty
\frac{a^s}{(s-N)!}
\sum_{r=1}^N
w_{dH}(N-r;\la,s-N,N-1)
\mathcal{R}_{j-1}\bigl(\ell(N-r); \, \la , s-N, N-1\bigr)^2
\\
&\qquad =
e^{-a}\left(\frac{2}{a}\right)^\la
\frac{(j-1)!}{(\la+j-1)!}
\sum_{s=N}^\infty
\frac{a^{s-j}}{(s-j)!}.
\end{align*}

The second sum in \eqref{eq:2sums}
is then
\begin{align*}
&
 \sum_{s=j}^{N}
\sum_{n=0}^{s-1}
\frac{\mathfrak{X}(s-n,j,n,\la)
\mathfrak{X}(s-n,k,n,\la)}
{\left(\D_n^{(\la)}\right)_{s-n,s-n}}
 \beta_{s-n,j,n}^{(\la)} \beta_{s-n,k,n}^{(\la)}
 \\
&\qquad=
 (-1)^{j+k}
\sum_{s={j}}^{N}
\frac{ e^{-a} a^{s-\frac12(j+k)} }{(s-k)!(s-j)!}\left(\frac{2}{a}\right)^\la
\frac{ \sqrt{(N-j)!(N-k)!} }{ (N-s)!\la !}
\sum_{n=0}^{s-1}
w_{dH}(n;\la,N-s,s-1)
\\
&\qquad\qquad \times\mathcal{R}_{j-1}\bigl(\ell(n); \, \la , N-s, s-1\bigr)
\mathcal{R}_{k-1}\bigl(\ell(n); \, \la , N-s, s-1\bigr).
\end{align*}
which also
vanishes for $j\neq k$ and when we do have $j=k$ we get
\begin{align*}
&\sum_{s={j}}^{N-1}
\frac{ e^{-a} a^{s-j} }{(s-j)!^2}\left(\frac{2}{a}\right)^\la
\frac{ (N-j)!}{ (N-s)!\la !}
\sum_{n=0}^{s-1}
w_{dH}(n;\la,N-s,s-1)
\mathcal{R}_{j-1}\bigl(\ell(n); \, \la , N-s, s-1\bigr)^2
\\
&\qquad \qquad =
e^{-a}\left(\frac{2}{a}\right)^\la
\frac{(j-1)!}{(\la+j-1)!}
\sum_{s={j}}^{N-1} \frac{a^{s-j}}{(s-j)!}.
\end{align*}

The only thing that is
left to do is combine the sums and perform the summation over $s$
$$
\left(\mathscr{W}_R^{(\la)}(0)\right)_{jj}
=
\frac{e^{-a}(j-1)!}{(\la+j-1)!}
\left(\frac{2}{a}\right)^\la
\sum_{s=j}^\infty
\frac{a^{s-j}}{(s-j)!}
=
\frac{(j-1)!}{(\la+j-1)!}
\left(\frac{2}{a}\right)^\la
=
\left(T^{(\la)}\right)_{jj}^{-1}.
$$
The entries of $T^{(\la)}$ were given in \eqref{eq:parameters-sol}. Now all that is left is to 
use \eqref{eq:dualnormdecomp} to
arrive at the zero moment
$$
\sW_0^{(\la)} 
= (I+A^\ast)^{-\la}T^{(\la)-1}(I+A)^{-\la}
=
\left(W^{(\la)}(0)\right)^{-1}.
$$

\end{document}